\documentclass[11pt,reqno]{amsart}
\usepackage{amsfonts,amssymb,mathptmx}
\usepackage{a4wide}
\usepackage{tensor} 
\usepackage{bbm}
\usepackage{enumitem}
\usepackage{hyperref}

\newcommand{\R}{\mathbb{R}} \newcommand{\N}{\mathbb{N}}              
 \newcommand{\C}{\mathbb{C}}      
\newcommand{\Rn}{{\R}^n}

\newcommand{\cd}{{\mathcal D}}
\newcommand{\ce}{{\mathcal E}}
\newcommand{\cs}{{\mathcal S}}
\newcommand{\cf}{{\mathcal F}}
\newcommand{\cfi}{{\mathcal F}^{-1}}

\providecommand{\cal}[1]{\mathcal{#1}}

\newcommand{\loc}{\operatorname{loc}}
\newcommand{\id}{\operatorname{id}}
\newcommand{\ext}{\operatorname{ext}}
\newcommand{\univ}{\operatorname{u}}

\newcommand{\ind}[1]{\ensuremath{\mathbbm{1}_{#1}}}

\newcommand{\supp}{{\operatorname{supp}}\,}
\newcommand{\dist}{{\operatorname{dist}}\,}

\newtheorem*{theorem*}{{\bf Theorem\/}}

\newtheorem{theorem}{Theorem}\numberwithin{theorem}{section}
\newtheorem{corollary}[theorem]{Corollary}
\newtheorem{lemma}[theorem]{Lemma}
\newtheorem{proposition}[theorem]{Proposition}

\theoremstyle{definition}
\newtheorem{definition}[theorem]{Definition}
\newtheorem{remark}[theorem]{Remark}

\begin{document}

\title[Traces and Mixed Norms]{%
Anisotropic Lizorkin--Triebel Spaces with Mixed Norms --- Traces on Smooth Boundaries}
\author[Johnsen, Munch Hansen, Sickel]{%
J.~Johnsen,\address{Department of Mathematical Sciences, Aalborg University,
Fredrik Bajers Vej 7G, DK-9220 Aalborg {\O}st, Denmark} 
\email{jjohnsen@math.aau.dk} 
S.~Munch~Hansen,
\address{Department of Mathematical Sciences, Aalborg University,
Fredrik Bajers Vej 7G, DK-9220 Aalborg {\O}st, Denmark} 
\email{sabrina\underline{~}privat@hotmail.com}
 W.~Sickel
\address{Mathematisches Institut, Ernst-Abbe-Platz 2, D-07740 Jena, Germany} 
\email{Winfried.Sickel@uni-jena.de}
}
\enlargethispage{2\baselineskip}
\thanks{J.~Johnsen and S.~Munch Hansen were supported by the 
Danish Council for Independent Research, Natural Sciences (Grant no.~11-106598)
\\[5\jot]
{\tt Accepted by Mathematische Nachrichten, June 2014}}


 \begin{abstract}
This article deals with trace operators on anisotropic Lizorkin--Triebel spaces with 
mixed norms over cylindrical domains with smooth boundary.
As a preparation we include a rather self-contained exposition of 
Lizorkin--Triebel spaces on manifolds and extend these results to
mixed-norm Lizorkin--Triebel spaces on cylinders in Euclidean space. 
In addition Rychkov's universal extension operator for a half space is shown to be bounded with respect 
to the mixed norms, and a support preserving right-inverse of the trace is given explicitly
and proved to be continuous in the scale of mixed-norm Lizorkin--Triebel spaces.
As an application, the heat equation is considered in these spaces,
and the necessary compatibility conditions on the data are deduced.
 \end{abstract}
 
\maketitle
\sloppy
\section{Introduction}
The present paper departs from the work~\cite{JS08} of the first and
third author dealing with traces on hyperplanes of 
anisotropic Lizorkin--Triebel spaces $F^{s,\vec a}_{\vec p,q}(\Rn)$ with mixed norms. 

The application of these spaces to parabolic differential equations is to some extent known. It
was outlined in the introduction to~\cite{JS08} how they can apply to fully inhomogeneous initial and
boundary value problems: for such problems the $F^{s,\vec a}_{\vec p,q}$-spaces are in general 
\textit{inevitable} for a correct description of the boundary data. Previously, a somewhat similar conclusion
was obtained in works of Weidemaier \cite{wei98,wei02,wei05} (and also
by Denk, Hieber and Pr\"{u}ss \cite{DHP07}). He discovered the necessity of  
isotropic Lizorkin--Triebel spaces (for vector-valued functions) for an optimal description of the
time regularity of the boundary data. However, with integral exponents $p_x$ and $p_t$ in the space and time
directions, respectively, Weidemaier worked under the technical restriction that $p_x \le p_t$. 

For the reader's sake, it is recalled that the main purpose of \cite{JS08} was to 
extend the classical theory of trace operators to the $F^{s,\vec a}_{\vec p,q}$-scales.
However, because the mixed norms do not allow a change of
integration order, this meant that the techniques had to be worked out both for the `inner' and
`outer' traces given on, say smooth functions as
\begin{equation*}
  u(x_1,x'') \mapsto u(0,x''),\quad\text{resp.}\quad 
  u(x',x_n) \mapsto u(x',0).
\end{equation*}
When $u\in F^{s,\vec a}_{\vec p,q}(\Rn)$, then in the first case the trace was proved 
to be surjective on the mixed-norm Lizorkin--Triebel space $F^{s-a_1/p_1,a''}_{p'',p_1}(\R^{n-1})$
having the specific sum exponent $q=p_1$, while in the second case the trace space is (as usual)
a Besov space, namely $B^{s-a_n/p_n,a'}_{p',p_n}(\R^{n-1})$.

Previously Berkola{\u\i}ko~\cite{Ber85} obtained such results for the classical range
$1<p_j, q<\infty$.
As indicated, only traces on hyperplanes were covered in~\cite{JS08}; but the study included (almost) necessary and
sufficient conditions on $s$ in relation to $\vec a$, $\vec p$ and $q$, also in
combination with normal derivatives (Cauchy traces), and existence and continuity of
right-inverses. Furthermore,  Weidemaier's restriction on the integral
exponents was never encountered with the framework and methods adopted
in~\cite{JS08}. 

These investigations in~\cite{JS08} are in this work followed up with
a general study of trace operators and their right-inverses in the scales $F^{s,\vec a}_{\vec p,q}$
of anisotropic Lizorkin--Triebel spaces with mixed norms defined on smooth open cylinders $\Omega\times
I$ (where $I:=\,]0,T[\,$) and their curved boundaries $\Gamma\times I$ ($\Gamma:=\partial\Omega$).

In doing so, it is a main technical question to obtain invariance of the spaces $\overline F^{s,\vec a}_{\vec p,q}(U)$
under the map $f\mapsto f\circ\sigma$, when $U\subset\Rn$ is open and $\sigma$ is a $C^\infty$-bijection.
We addressed this question in our joint paper~\cite{HJS13b}, where we proved invariance e.g.~under the restriction that 
$\sigma$ only affects groups of coordinates $x_j$ for which the corresponding $p_j$ are equal in the vector of integral exponents $\vec p= ( p_1,\ldots,p_n)$; and similarly for the moduli of anisotropy $a_j$. 

This was done by generalising Triebel's method in~\cite[4.3.2]{tri92}. Indeed, having reduced to large $s$ using a lift operator, 
it relies on Taylor expansion of the inner and outer functions, whereby 
most terms are manageable when the $F^{s,\vec a}_{\vec p,q}$-spaces are normed via kernels of localised means developed in~\cite{HJS13a};
an underlying parameter-dependent estimate obtained in~\cite{HJS13a} finally gives control over the effects of the Jacobian matrices.

In this paper, we proceed to develop the consequences for trace operators. 
E.g.~the trace $r_0$ at $\{t=0\}$ of $u\in \overline F^{s,\vec  a}_{\vec p,q}(\Omega\times I)$  is given a
meaning in a pedestrian way using an arbitrary extension of $u$ to
$\R^{n+1}$ and applying the trace at $\{t=0\}$ from~\cite{JS08}. 
In terms of the splitting $\vec p = (p',p_t)$ with all entries in
$p'$ being equal to $p_0$ and likewise for~$\vec a$, we can abbreviate our result for $r_0$
as follows:

\begin{theorem*}
When $s>\frac{a_t}{p_t}$, $1\le p_0,p_t<\infty$, $1\le q\le\infty$, the operator $r_0$ is a bounded \emph{surjection}
\begin{equation*}
  r_0: \overline F^{s,\vec a}_{\vec p,q}(\Omega\times I)\to \overline B^{s-{a_t}/{p_t},a'}_{p',p_t}(\Omega).
\end{equation*}
Furthermore, $r_0$ has a right-inverse $K_0$ going the opposite way and it is bounded for every $s\in\R$,
\begin{equation*}
K_0: \overline B^{s-{a_t}/{p_t},a'}_{p',p_t}(\Omega)\to \overline F^{s,\vec a}_{\vec p,q}(\Omega\times I).
\end{equation*}
\end{theorem*}
The classical borderline $s=1/p$ is recovered from this in the isotropic case, 
as $\vec a=(1,\dots,1)$ then.
But the full statement in Theorem~\ref{main11} below requires $s$ to be larger 
(by $\frac{a_0}{p_0}-a_0$) if $0<p_0<1$, 
so for such $p_0$ even the borderline $s=a_t/p_t$ is shifted upwards.

It is more involved to give meaning to the trace $\gamma$ of $u$ at the curved boundary
$\Gamma\times I$, since it
requires to work locally first and then observe that the local pieces together give a globally 
defined trace. 
Using the splitting $\vec p=(p_0,p'')$, where $p''=(p_0,\dots,p_0,p_t)$ and
likewise for~$\vec a$, we may state the

\begin{theorem*}
When $\partial\Omega$ is compact and $s>\frac{a_0}{p_0}$, $1\le p_0,p_t<\infty$, $1\le q\le\infty$, 
then $\gamma$ is a bounded \emph{surjection}
\begin{equation*}
\gamma: \overline F^{s,\vec a}_{\vec p,q}(\Omega \times I) \to \overline F^{s-a_0/p_0,a''}_{p'',p_0}(\Gamma\times I).
\end{equation*}
Furthermore, $\gamma$ has a right-inverse $K_\gamma$ going the opposite way and it is bounded for every $s\in\R$,
\begin{equation*}
K_\gamma: \overline F^{s-a_0/p_0,a''}_{p'',p_0}(\Gamma\times I)\to \overline F^{s,\vec a}_{\vec p,q}(\Omega \times I).
\end{equation*}
\end{theorem*}

Here Theorem~\ref{thm:boundednessCurvedTrace} contains a stronger condition on $s$ if any of $p_0$,
$p_t$ or $q$ are given in $\,]0,1[\,$.

Note that the sum exponent of the codomain inherits the value $q=p_0$ after the normal variable (say
$x_1$ or $x_n$) has been eliminated by restriction to the boundary. While that is analogous to the
case for $r_0$ above ($q=p_t$), we should emphasise that for $p_t\ne p_0$ the trace space for
$\gamma$ does not just have a mixed-norm, it is moreover \emph{outside} the Besov scale because it equals
the Lizorkin--Triebel space $F^{s-a_0/p_0,a''}_{p'',p_0}$.

The operator $K_\gamma$ is constructed using the right-inverse in~\cite[Thm.~2.6]{JS08} to the trace
at $\{x_1=0\}$ and Rychkov's universal extension operator \cite{ryc99ext}. The latter is modified
to a version  ${\cal E}_{\operatorname{u}}$ with good properties in
anisotropic, mixed-norm Lizorkin--Triebel spaces over half-spaces in
Theorem~\ref{thm:rychkovExtension} below.

As a novelty, from the construction of ${\cal E}_{\operatorname{u}}$,
we derive in Theorem~\ref{thm:constructionOfRightInverseOnCylinderPerservingSupport} an \emph{explicit}  
construction of an operator $Q_\Omega$ going from $\Rn$ to $\R^{n+1}$,
which in $\Omega$ has $r_0$ as a left-inverse and yet it preserves support 
with respect to the $x$-variable:
\begin{equation*}
Q_\Omega: \overset{\circ}{B}{}^{s,a'}_{p',p_t}(\overline\Omega)\to \overset{\circ}{F}{}^{s,\vec a}_{\!\vec p,q}(\overline\Omega\times\R) \quad\text{for all }s\in\R.
\end{equation*}
This is important for reduction of parabolic problems to homogeneous
ones: e.g.\ surjectivity of $\gamma$ allows to get zero data on $\Gamma\times I$,
and $Q_\Omega$ gives a further reduction to zero initial data; cf.~Remark~\ref{rem:compatibility}.

Indeed, after an analysis of traces at the curved corner $\Gamma\times \{0\}$ of the cylinder
$\Omega\times I$, we follow up in Theorem~\ref{thm:compatibility} by extending
the necessity of the compatibility conditions of Grubb and Solonnikov~\cite{GS90} 
to solutions in the mixed-norm Lizorkin--Triebel spaces of the heat equation. 

\subsection*{Contents}
Section~\ref{sec:preliminariesMathNr} contains a review of our notation and the definition of anisotropic 
Lizorkin--Triebel spaces with mixed norms is recalled, together with some needed properties 
and a pointwise multiplier assertion.
Moreover, a basic lemma for elements in $F^{s,\vec a}_{\vec p,q}$ with e.g.\ compact support on
cross sections of the cylindrical domain is proved. 

In Section~\ref{sec:approach1MathNr} sufficient conditions for $f \mapsto f \circ \sigma$
to leave the spaces
$F^{s,\vec a}_{\vec p,q}(\Rn)$ invariant for a certain range of the parameters,
including negative values of $s$, are recalled.

Section~\ref{chap:domains} contains first a preparatory treatment of unmixed Lizorkin--Triebel spaces on general $C^\infty$-manifolds
and these results are then extended to $F^{s,\vec a}_{\vec p,q}$-spaces on the curved boundary of a cylinder.

Rychkov's universal extension operator in~\cite{ryc99ext} is modified to $\overline F^{s,\vec a}_{\!\vec p,q}(\Rn_+)$ in Section~\ref{sec:rychkovExtension}. 
Moreover, its properties on temperate distributions are analysed in addition.

Finally, Section~\ref{traces} contains a discussion of the trace at the flat as well as at the curved boundary of a cylindrical domain, including applications to e.g.\ the Dirichlet boundary problem for the heat equation.

\section{Preliminaries}
\label{sec:preliminariesMathNr}
\subsection{Notation}
The Schwartz space $\cs(\Rn)$ consists of the rapidly decreasing $C^\infty$-functions and
it is equipped with the family of seminorms, using $D^\alpha := (-\mathrm{i}\partial_{x_1})^{\alpha_1}\cdots(-\mathrm{i}\partial_{x_n})^{\alpha_n}$ for each multi-index $\alpha =(\alpha_1,\ldots,\alpha_n)$ with $\alpha_j\in\N_0:=\N\cup\{0\}$, $\mathrm{i}^2=-1$ and $\langle x\rangle^2 := 1+|x|^2$,
\begin{equation*}
p_M(\varphi) := \sup \big\{\, \langle x\rangle^M |D^\alpha \varphi(x)|\, \big|\, x\in\Rn, |\alpha|\leq M\big\},\quad M\in\N_0.
\end{equation*}
By duality, the Fourier transformation $\cf g(\xi)=\widehat{g}(\xi) = \int_{\Rn} e^{-\mathrm{i}
  x\cdot \xi} g(x)\, dx$ for $g\in\cs(\Rn)$ extends to the dual space $\cs'(\Rn)$ of temperate
distributions. $\langle u,\psi\rangle$ denotes the value of $u\in\cs'$ on $\psi\in\cs$.

Throughout, inequalities for vectors $\vec{p}=(p_1,\ldots,p_n)$ are understood componentwise; likewise for functions,
e.g.~$\vec{p}\,!=p_1!\cdots p_n!$, while $t_+ := \max(0,t)$ for $t\in\R$.
For $0<\vec p\leq \infty$ the space $L_{\vec{p}}(\Rn)$ consists of the Lebesgue measurable functions such that 
\begin{alignat*}{1}
\|\, u \, | L_{\vec{p}}(\Rn)\| := 
\bigg(\int_{-\infty}^\infty \bigg( \ldots 
\bigg( \int_{-\infty}^\infty |u(x_1,\ldots ,x_n)|^{p_1} dx_1
\bigg)^{p_2/p_1}  
\ldots  
\bigg)^{p_n/p_{n-1}} dx_n \bigg)^{1/p_n} <\infty;
\end{alignat*}
in case $p_j=\infty$, the essential supremum over $x_j$ is used.
When equipped with this quasi-norm, $L_{\vec p}$ is a quasi-Banach space (normed if $\vec p \ge 1$);
it was considered e.g.\ by Benedek and Panzone~\cite{BePa61}.

In addition, we shall for $0 < q \le \infty$ denote by 
$L_{\vec p} (\ell_q)(\Rn)$ the space of sequences 
$(u_k)_{k\in\N_0}$ of Lebesgue measurable
functions $u_k: \Rn \to \C$ such that
\[
\| \, (u_k)_{k\in\N_0} \, |L_{\vec p} (\ell_q)(\Rn)\| :=
\bigg\| \, \Big(\sum_{k=0}^\infty |u_k|^q
\Big)^{1/q}\, \bigg|L_{\vec p} (\Rn)\bigg\| < \infty ;
\]
with supremum over $k$ in case $q=\infty$.
For brevity, $\| \, (u_k)_{k\in\N_0} \, |L_{\vec p} (\ell_q)(\Rn)\|$ 
is written
$\| \, u_k \, |L_{\vec p} (\ell_q)\| $
and when $\vec{p}=(p,\ldots,p)$, then $L_{\vec{p}}$ is simplified to $L_p$ etc.
We recall that sequences of 
$C_0^\infty$-functions are dense in $L_{\vec p} (\ell_q)$ if $\max (p_1, \ldots, p_n, q)<\infty$. 

Generic constants will be denoted by $c$ or $C$, with their dependence on certain parameters explicitly stated when 
relevant. Lastly, the closure of an open set $U\subset\Rn$ is denoted $\overline U$ and $B(0,r)$ is
the ball centered at $0$ with radius $r>0$; the dimension of the surrounding Euclidean space will be
clear from the context or otherwise stated explicitly.

\subsection{Anisotropic Lizorkin--Triebel Spaces with Mixed Norms}
This section only contains the Fourier-analytic definition of the mixed-norm Lizorkin--Triebel
spaces and a few essential properties used below. For an introduction to these spaces
we refer the reader to~\cite{JS07} and~\cite[Sec.~3]{JS08}. 

First we recall the definition of
the anisotropic distance function $|\cdot|_{\vec a}$, where $\vec{a}=(a_1,\ldots,a_n)\in [1,\infty[\,^n$, 
on $\Rn$ and some of its properties.
Using the quasi-homogeneous dilation 
$t^{\vec a}x:=(t^{a_1}x_1,\ldots,t^{a_n}x_n)$ for $t\ge0$, $|x|_{\vec a}$
is for $x\in\Rn\setminus\{0\}$ defined as the unique $t>0$ such that $t^{-\vec a}x\in S^{n-1}$ ($|0|_{\vec a}:=0$), i.e.~
\begin{equation*}
  \frac{x_1^2}{t^{2a_1}}+\dots+  \frac{x_n^2}{t^{2a_n}}=1.
\end{equation*}
For basic properties of $|\cdot|_{\vec a}$ we refer to~\cite[Sec.~3]{JS07} or \cite{Y1}.

The Fourier-analytic definition also relies on a Littlewood--Paley decomposition, i.e.~$1=\sum_{j=0}^\infty \Phi_j(\xi)$,
which is based on a (for convenience fixed) $\psi\in C_0^\infty$ such that $0\leq\psi(\xi)\leq 1$ for all $\xi$, $\psi(\xi)=1$ if $|\xi|_{\vec a}\leq 1$
and $\psi(\xi)=0$ if $|\xi|_{\vec a}\geq 3/2$. Setting $\Phi = \psi - \psi(2^{\vec a}\cdot)$, we define
\begin{equation}\label{unity}
\Phi_0 (\xi) = \psi (\xi), \qquad \Phi_j (\xi) = \Phi(2^{-j\vec a}\xi), \quad j=1,2,\ldots 
\end{equation}

\begin{definition}
The Lizorkin--Triebel space $F^{s,\vec a}_{\vec p,q}(\Rn)$ with $s\in\R$, $0<\vec{p} < \infty$ and
$0 < q \le \infty$ consists of the $u\in\cs'(\R^n)$ such that
\[
\| \, u \, | F^{s,\vec a}_{\vec p,q}(\Rn)\| := 
\bigg\|\, \bigg(\sum_{j=0}^\infty 2^{jsq} \left|\cfi \left(\Phi_j (\xi) 
\cf u(\xi)\right) ( \cdot )\right|^q \bigg)^{1/q} \, \bigg| L_{\vec{p}}(\Rn)
\bigg\| < \infty .
\]
\end{definition}

The number $q$ is a sum exponent (sometimes called the microscopic or fine index)
and the entries in $\vec p$ are integral exponents,
while $s$ is a smoothness index.
In case $\vec a = (1,\ldots,1)$, the parameter $\vec a$ is omitted.

Let us also recall (cf.\ \cite[Sec.~2.3]{HJS13b} or \cite[Prop.~2.10]{JS08})
that there is an identification with the well-known anisotropic Bessel potential spaces,
\begin{equation*}
  H^{s,\vec a}_{\vec p}(\Rn) = F^{s,\vec a}_{\vec p, 2}(\Rn) \quad\text{for $s\in\R$, $1<\vec p<\infty$}.
\end{equation*}

When studying traces on the flat boundary of a cylinder, Besov spaces are inevitable:

\begin{definition}
The Besov space $B^{s,\vec a}_{\vec p,q}(\R^n)$ with $s\in\R$ and $0<\vec{p},q\leq \infty$ consists of
the $u\in\cs'(\R^n)$ such that
\[
  \|\, u\, | B^{s,\vec a}_{\vec p,q}(\R^n) \| :=
  \bigg( \sum_{j=0}^\infty 2^{jsq}  \|\, \cfi \left(\Phi_j   \cf u\right) \, | L_{\vec{p}}(\R^n) \|^q\bigg)^{1/q} <\infty.
\]
\end{definition}

For a general reference on the mixed-norm spaces $F^{s,\vec a}_{\vec p,q}$ and $B^{s,\vec a}_{\vec p,q}$ 
the reader may consult the book of Besov, Il'in and Nikol'ski{\u\i}~\cite{BIN79,BIN96}, or that of
Schmeisser and Triebel~\cite{ScTr87} ($n=2$). The set-up with $\vec a$ was analysed by Yamazaki~\cite{Y1}.
Here we review a few facts needed below.

Both $F^{s,\vec a}_{\vec p,q}$ and $B^{s,\vec a}_{\vec p,q}$ are quasi-Banach spaces (normed if $\min(p_1,\ldots,p_n,q)\geq 1$) and the quasi-norm is subadditive when raised to the power $d:=\min(1,p_1,\ldots,p_n,q)$,
\begin{equation}\label{eq:subadditivity}
\| \, u+v \, |F^{s,\vec a}_{\vec p,q}\|^d \le 
\| \, u \, |F^{s,\vec a}_{\vec p,q}\|^d  +  \| \, v \, |F^{s,\vec a}_{\vec p,q} \|^d,\quad 
u,v \in F^{s,\vec a}_{\vec p,q}(\R^n).
\end{equation}
Different choices of the anisotropic decomposition of unity give the same space (with equivalent quasi-norms)
and there are continuous embeddings
\begin{equation}\label{eq:SFembedding}
\cs (\Rn) \hookrightarrow F^{s,\vec a}_{\vec p,q} (\Rn) \hookrightarrow \cs' (\Rn),
\end{equation}
where $\cs$ is dense in $F^{s,\vec a}_{\vec p,q}$ for $q<\infty$.
Both \eqref{eq:subadditivity} and \eqref{eq:SFembedding} hold verbatim for $B^{s,\vec a}_{\vec p,q}(\Rn)$ as well.

\begin{lemma}\label{lem:lambdaLemmaBesov}
For $\lambda>0$ so large that $\lambda \vec a\ge 1$, the spaces $B^{s,\vec a}_{\vec p,q}(\Rn)$,
$F^{s,\vec a}_{\vec p,q}(\Rn)$ coincide with $B^{\lambda s,\lambda\vec a}_{\vec p,q}(\Rn)$,
respectively $F^{\lambda s,\lambda \vec a}_{\vec p,q}(\Rn)$ and the corresponding quasi-norms are
equivalent. 
\end{lemma}

The proof of this lemma for Besov spaces follows that of Lizorkin--Triebel spaces, which can be found in~\cite[Lem.~3.24]{JS08}. Indeed, the only exception is that~\cite[Lem.~3.23]{JS08} needs to be adapted to Besov spaces, but this is easily done using the modifications indicated just above Lemma 3.21 there.

In view of Lemma~\ref{lem:lambdaLemmaBesov}, one could envisage that most results obtained for the
scales when $\vec a\geq 1$ can be extended to the range $0< \vec a < \infty$.  
For details on this we refer to~\cite[Rem.~2.6]{HJS13a}.

The Banach space $C_{\operatorname{b}}(\R^n)$ of continuous, bounded functions is equipped with the sup-norm, while the subspace $L_{1,\operatorname{loc}}(\Rn)\subset \cd'(\Rn)$ of locally integrable functions is endowed with the Fr\'{e}chet space topology defined from the seminorms
$u\mapsto  \int_{|x|\leq j} |u(x)|\, dx$, $j\in\N$. 

\begin{lemma}[{\cite[Lem.~1]{HJS13b}}]\label{derivativeTraces}
Let $s\in\R$ and $\alpha\in\N^n_0$ be arbitrary.\\
\indent\indent{\rm (i)} The differential operator $D^\alpha$ is bounded $F^{s,\vec a}_{\vec p,q}(\Rn) \to F^{s-\alpha \cdot \vec{a}, \vec{a}}_{\vec{p},q} (\Rn)$.\\
\indent\indent{\rm (ii)} For $s> \sum_{\ell = 1}^n \big(\frac{a_\ell}{p_\ell} - a_\ell\big)_+$ there is an embedding $F^{s,\vec a}_{\vec p,q} (\Rn)\hookrightarrow L_{1,\operatorname{loc}}(\Rn)$.\\
\indent\indent{\rm (iii)} The embedding $F^{s,\vec a}_{\vec p,q}(\Rn)\hookrightarrow C_{\operatorname{b}}(\Rn)$ holds for $s>\frac{a_1}{p_1}+\cdots+\frac{a_n}{p_n}$.
\end{lemma}

Next, we recall a paramultiplication result and refer to~\cite[Sec.~2.4]{HJS13b} for details.

\begin{lemma}\label{multTraces}
Let $s \in \R$ and take $s_1 > s$ such that also
\begin{equation}\label{eq:assumptionMult}
s_1 > \sum_{\ell=1}^n \Big(\frac{a_\ell}{\min(1,q,p_1, \ldots , p_\ell)} -a_\ell \Big) -s.
\end{equation}
Then each $u\in B^{s_1, \vec{a}}_{\infty,\infty}(\Rn)$ defines a pointwise multiplier of $F^{s,\vec a}_{\vec p,q}(\Rn)$ and
\[
\| \, u\cdot v \, | F^{s,\vec a}_{\vec p,q}\| \le c \, 
\| \, u \, | B^{s_1, \vec{a}}_{\infty,\infty}\|\cdot \|\, v \, | F^{s, \vec{a}}_{\vec p,q}\|,\quad v\in F^{s,\vec a}_{\vec p,q}(\Rn).
\]
In particular, it holds for $u$ in $C^\infty_{L_\infty}(\Rn) := \{\, g\in C^\infty(\Rn)\, |\,
\forall\alpha\in \N_0^n:\, D^\alpha g\in L_\infty(\Rn)\}$. 
\end{lemma}

The characterisation of $F^{s,\vec a}_{\vec p,q}(\Rn)$ by kernels of local and localised
means as developed in~\cite[Thm.~5.2]{HJS13a} is utilised below, hence it is included here for convenience, using the notation
\begin{equation}\label{eq:defOfSubscript_j}
  \varphi_j(x) = 2^{j|\vec a|}\varphi(2^{j\vec a}x),\quad \varphi\in\cs,\enskip j\in\N.
\end{equation}

\begin{theorem}\label{thm:local}
Let $k_0,k^0 \in {\mathcal{S}}({\mathbb{R}}^n)$ such that $\int k_0(x)\, dx \neq 0 \neq \int k^0(x)\, dx$ and
set $k (x)= \Delta^N k^0(x)$ for some $N\in{\mathbb{N}}$.
When $0 < \vec{p} <\infty$, $0< q \le \infty$, and $s < 2N\, \min(a_1,\ldots,a_n)$,
then a distribution $f \in {\mathcal{S}}'({\mathbb{R}}^n)$ belongs to $F^{s,\vec{a}}_{\vec{p},q}({\mathbb{R}}^n)$ 
if and only if 
\begin{equation}\label{eq:charOfF}
\|\, f \, | F^{s,\vec{a}}_{\vec{p},q}\|^* :=
\| \, k_0 * f\, |L_{\vec{p}}\| + \| \{2^{sj} k_j * f\}_{j=1}^\infty \, | L_{\vec{p}}(\ell_q)\|<\infty .
\end{equation}
Furthermore, $\|\, f \, | F^{s,\vec{a}}_{\vec{p},q}\|^*$ is an equivalent
quasi-norm on $F^{s,\vec{a}}_{\vec{p},q}({\mathbb{R}}^n)$.
\end{theorem}

For spaces over open sets $U\subset\Rn$ we use the notation introduced by
H\"{o}rmander~\cite[App.~B.2]{hor07} and place a bar over $F$ and $B$, to indicate that it is a space
of restricted distributions:

\begin{definition}\label{def:Fsubspace}
$1^\circ$ The space $\overline F^{s,\vec a}_{\vec p,q}(U)$ is defined as the set of 
$u \in \cd'(U )$ such that there exists a distribution $f \in F^{s,\vec a}_{\vec p,q}({\R}^{n})$
satisfying
\begin{equation}\label{eq:distributionRestriction}
 \langle f,\varphi\rangle = \langle u,\varphi\rangle \quad \mbox{for all}\quad \varphi \in 
C_0^\infty (U) .
\end{equation}
We equip $\overline F^{s,\vec a}_{\vec p,q}(U)$ with the quotient quasi-norm, which is a norm if $\vec p,q\geq 1$, 
\begin{equation*}
\|\, u\, | \overline F^{s,\vec a}_{\vec p,q}(U)\| := \inf_{r_U f=u} \|\, f\, |F^{s,\vec a}_{\vec p,q}(\Rn)\|.
\end{equation*}

$2^\circ$ The space ${\overset{\circ}{F}}{}_{\!\vec p,q}^{s,\vec a}(\overline U)$ consists of the distributions 
$u\in F^{s,\vec a}_{\vec p,q}(\Rn)$ that satisfy $\supp u\subset\overline U$. 

$3^\circ$ The Besov spaces $\overline B^{s,\vec a}_{\vec p,q}(U)$ and
${\overset{\circ}{B}}{}_{\!\vec p,q}^{s,\vec a}(\overline U)$ are defined analogously. 
\end{definition} 

Recall that since $F^{s,\vec a}_{\vec p,q}(\Rn)$ is a quasi-Banach space, $\overline F^{s,\vec
  a}_{\!\vec p,q}(U)$ is seen to be so too, using the usual arguments for quotient spaces modified
to exploit the subadditivity in~\eqref{eq:subadditivity}. 

In \eqref{eq:distributionRestriction} it is tacitly understood that on the left-hand side $\varphi$
is extended by 0 outside $U$. For this we henceforth use the operator notation $e_U \varphi$.
Likewise $r_U$ denotes restriction to $U$, whereby $u=r_U f$ in \eqref{eq:distributionRestriction}.
We shall refer to such $f$ as an extension of $u$.

\begin{remark}\label{rem:normedByInfnorm}
Theorem~\ref{thm:local} induces an equivalent quasi-norm $\|\, u\, | \overline F^{s,\vec a}_{\vec p,q}(U)\|^*$ on $\overline F^{s,\vec a}_{\vec p,q}( U)$	by taking the infimum of $\|\, f\, |F^{s,\vec a}_{\vec p,q}(\Rn)\|^*$ for $r_U f = u$.
\end{remark}

As a preparation we include a slightly modified version of~\cite[Lem.~8]{HJS13b}:

\begin{lemma}\label{lem:extensionBy0}
Let $U\subset\Rn$ be open. When $\overline F^{s,\vec a}_{\vec p,q}(U\times\R)$ is normed as in Remark~\ref{rem:normedByInfnorm} using kernels of local means with $\supp k_0,\supp k\subset B(0,r)$ for an $r>0$, and when $K\subset U$ is a compact set fulfilling
\begin{equation}\label{eq:choiceOfr}
  \dist (K,\Rn\setminus U)>2r,
\end{equation}
then it holds for every $f\in \overline F^{s,\vec a}_{\vec p,q}(U\times\R)$ with $\supp f\subset K\times\R$ that
\begin{equation*}
  \|\, f\, | \overline F^{s,\vec a}_{\vec p,q}( U\times\R)\|^* = \|\, e_{U\times\R}f\, |F^{s,\vec a}_{\vec p,q}(\R^{n+1})\|^*.
\end{equation*}
That is, the infimum is for such $f$ attained at $e_{U\times\R}f$.
\end{lemma}

\begin{proof}
For an arbitrary extension $\widetilde f$ of $f$, it holds for $g:=\widetilde f-e_{U\times\R}f$ that $\supp e_{U\times\R} f\cap \supp g =\emptyset$, hence by~\eqref{eq:choiceOfr},
\begin{equation*}
  \supp(k_j*e_{U\times\R}f)\cap \supp(k_j*g)=\emptyset,\quad j\in\N_0.
\end{equation*}
When $g\neq 0$, there exists $j\in\N_0$ such that $\supp(k_j*g)\neq\emptyset$, thus $k_j*g(x)\neq 0$ on an open set disjoint from $\supp(k_j*e_{U\times\R}f)$. This term therefore effectively contributes to the $L_{\vec p\,}$-norm in the local means characterisation, yielding $\|\, \widetilde f\, | F^{s,\vec a}_{\vec p,q}(\R^{n+1})\|^* > \|\, e_{U\times\R} f \,| F^{s,\vec a}_{\vec p,q}(\R^{n+1})\|^*$.
\end{proof}

For temperate distributions vanishing in the time direction, we let $e_{I \to I'}$ denote extension by 0 from $\R^{n-1}\times I$ to $\R^{n-1}\times I'$ for open intervals $I\subset I'$. Then we similarly get

\begin{lemma}\label{lem:extensionBy0inLastCoord}
Let $I=\,]b,c[\,$ and $I'=\,]a,c[\,$ where $-\infty\le a<b<c\le \infty$. When $\overline F^{s,\vec a}_{\vec p,q}(\R^{n-1}\times I)$ is normed as in Remark~\ref{rem:normedByInfnorm} using kernels of local means with $\supp k_0,\supp k\subset B(0,r)$ for an $r>0$, 
then it holds for every $f\in \overline F^{s,\vec a}_{\vec p,q}(\R^{n-1}\times I)$ satisfying 
$f(\cdot,t)=0$ for $t \in \, ]b,b+2r[\,$ 
that 
\begin{equation}\label{eq:extensionBy0inLastCoord}
\|\, f\, |\overline F^{s,\vec a}_{\vec p,q}(\R^{n-1}\times I)\|^* 
= \|\, e_{I\to I'} f\, |\overline F^{s,\vec a}_{\vec p,q}(\R^{n-1}\times I')\|^*.
\end{equation}
A similar equality holds for extension from $I=\,]a,b[\,$ to $I'$, when $f(\cdot,t)=0$ for $t\in \,]b-2r,c[\,$.
\end{lemma}

\begin{proof}
The inequality $\le$ follows immediately, since the distributions considered in the infimum on the right-hand side in~\eqref{eq:extensionBy0inLastCoord} also are considered on the left-hand side. 

To prove equality we assume that $<$ holds. Then there exists an extension $\widetilde f$ of $f$ which is not among the distributions considered in the infimum on the right-hand side, and which, with an infimum over $r_{\R^{n-1}\times I'}h=e_{I\to I'}f$, moreover fulfils 
\begin{equation}\label{eq:equationToObtainContradiction}
\|\, \widetilde f\, |F^{s,\vec a}_{\vec p,q}(\Rn)\| < \inf \|\, h\, |F^{s,\vec a}_{\vec p,q}(\Rn)\|.
\end{equation}

Actually, it suffices to consider those~$h$ for which $h\equiv 0$ on $\R^{n-1}\times\, ]-\infty,b+2r[\,$. Indeed, for any other $h$ the distribution $(1-\chi(t))h(\cdot,t)$, where $\chi\in C^\infty(\R)$ with $\chi(t)=1$ for $t\in\,]-\infty,a[\,$ and $\chi(t)=0$ for $t\in \,]b,\infty[\,$, has a smaller quasi-norm than $h$. This can be verified similarly to the proof of Lemma~\ref{lem:extensionBy0}, using that the distance between $\supp h \cap (\R^{n-1}\times\, ]-\infty,a[\,)$ and $\supp (1-\chi)h$ is at least $2r$. 

Now for such $h$ we have $\supp h \subset \R^{n-1}\times [b+2r,\infty[\,$, and since $\widetilde f(t)\not\equiv 0$ for $a<t<b$ it is easily seen by the proof strategy of Lemma~\ref{lem:extensionBy0} that $\|\, \widetilde f\,| F^{s,\vec a}_{\vec p,q}(\Rn)\|>\|\, h\, | F^{s,\vec a}_{\vec p,q}(\Rn)\|$, which contradicts~\eqref{eq:equationToObtainContradiction}.
\end{proof}

For simplicity of notation the ${}^*$ on the quasi-norm is omitted in the following.

\section{Invariance under Diffeomorphisms}
\label{sec:approach1MathNr}
To introduce Lizorkin--Triebel spaces on manifolds, it is essential that the spaces $\overline F^{s,\vec a}_{\vec p,q}( U)$ for certain open subsets $U\subset \Rn$ are invariant under suitable $C^\infty$-bijections $\sigma$. An extensive treatment of this subject can be found in~\cite{HJS13b}, but for convenience we recall the needed results. These hold for $0<\vec p<\infty$, $0<q\le\infty$ and $s\in\R$ unless additional requirements are specified.
First a result on isotropic spaces:

\begin{theorem}\label{thm:isotropicCompactSupport}
When $\sigma: U\to V$ is a $C^\infty$-bijection between open sets $U,V\subset\Rn$ and $f\in \overline F^s_{p,q}( V)$ has compact support, then $f\circ\sigma \in \overline F^s_{p,q}( U)$ and 
\begin{equation}\label{eq:isotropicCompactSupport}
  \|\, f\circ \sigma\, | \overline F^s_{p,q}( U)\| 
  \leq c \|\, f\, | \overline F^s_{p,q}( V)\|
\end{equation}
holds for a constant $c$ depending only on $\sigma$ and the set $\supp f$.
\end{theorem}

In the anisotropic situation it cannot be expected, e.g.~if $\sigma $
is a rotation, that
$f\circ \sigma$ has the same regularity as~$f$, nor that $f\circ \sigma \in L_{\vec{p}}$ when $f\in L_{\vec{p}}$.
We therefore restrict to $\vec p$ of the form
\begin{eqnarray}\label{eq:localVersionConditionsOna}
\vec{p} = (\underbrace{p_1, \ldots, p_1}_{N_1}, \underbrace{p_{2} , \ldots , p_2}_{N_2}, \,
\ldots \, ,\underbrace{p_m, \ldots, p_m}_{N_m}), \quad  N_1 + \dots + N_m=n,\enskip m\ge 2,
\end{eqnarray}
and $\vec a$ having the same structure.

\begin{theorem}\label{thm:localVersion}
Let $\sigma_j: U_j\to V_j$, $j=1,\ldots,m$, be $C^\infty$-bijections, where $U_j,V_j \subset \R^{N_j}$ are open.
When $\vec a$, $\vec p$ fulfil~\eqref{eq:localVersionConditionsOna} and $f\in \overline F^{s,\vec
  a}_{\vec p,q}(U_1\times\dots\times U_m)$ has compact support,
then~\eqref{eq:isotropicCompactSupport} holds for $U=U_1\times\cdots\times U_m$ and
$V=V_1\times\cdots\times V_m$ for a constant $c$ depending only on $\sigma$ and the set $\supp f$.
\end{theorem}

For traces at the curved boundary of cylinders, the next special case is useful:

\begin{theorem}\label{thm:infiniteCylinder}
Let $U,V\subset\R^{n-1}$ be open and let $\sigma: U\times\R\to V\times\R$ be a $C^\infty$-bijection on the form
\begin{equation*}
\sigma (x) = (\sigma'(x_1, \ldots , x_{n-1}),x_n) 
\quad \mbox{for all}\quad x\in U\times\R.
\end{equation*}
When $\vec a, \vec p$ satisfy~\eqref{eq:localVersionConditionsOna} with $m=2$, $N_1=n-1$, $N_2=1$
and $f\in \overline F^{s,\vec a}_{\vec p,q}(V\times\R)$ has $\supp f \subset K\times\R$, whereby $K\subset V$ is compact, then $f\circ\sigma\in \overline F^{s,\vec a}_{\vec p,q}(U\times\R)$ and 
\begin{equation*}
\|\, f\circ\sigma\, | \overline F^{s,\vec a}_{\vec p,q}(U\times\R) \|
\leq c(\supp f,\sigma) \|\, f\, | F^{s,\vec a}_{\vec p,q}(\R^{n+1}) \|.
\end{equation*}
\end{theorem}

The above three theorems can be found with proofs as Theorems~6--8, respectively, in~\cite{HJS13b}.
As needed, we shall tacitly apply these results in situations with $n+1$ variables; when the last of
these is interpreted as time, then we let $t=x_{n+1}$. 

\section{Function Spaces on Manifolds}
\label{chap:domains}
To develop Lizorkin--Triebel spaces over cylinders and to settle the necessary notation, we first
review distributions on manifolds, for the reader's convenience.

\subsection{Distributions on Manifolds}
\label{sec:distributionsMfd}
To allow comparison with existing literature on partial differential equations, 
we follow \cite[Sec.~8.2]{gru09} and~\cite[Sec.~6.3]{hor90}. E.g.~a 
diffeomorphism is in the following a bijective $C^\infty$-map between open sets, and we recall

\begin{definition}
An $n$-dimensional manifold $X$ is a second-countable Hausdorff space which is locally homeomorphic to $\Rn$. 
The manifold $X$ is $C^\infty$ (or smooth), if it is equipped with a $C^\infty$-structure, i.e.~a
family $\cf$ of homeomorphisms $\kappa$ 
mapping open sets $X_\kappa\subset X$ onto open sets $\widetilde{X}_\kappa\subset\R^n$, with
$X=\bigcup_{\kappa\in\cf} X_\kappa$, such that the maps 
\begin{equation}
\label{eq:manifoldDiffCond}
  \kappa\circ \kappa_1^{-1}: \kappa_1(X_\kappa \cap X_{\kappa_1}) \to \kappa(X_\kappa \cap X_{\kappa_1}), \quad
  \kappa,\kappa_1 \in\cf,
\end{equation}
are diffeomorphisms, and $\cf$ contains every homeomorphism $\kappa_0: X_{\kappa_0}\to \widetilde{X}_{\kappa_0}$, 
for which the compositions in \eqref{eq:manifoldDiffCond} with $\kappa=\kappa_0$  are diffeomorphisms.

A subfamily of $\cf$ where the $X_\kappa$ cover $X$ is called a (compatible) atlas, and $\cf_1\subset\cf_2$ means that
every chart $\kappa$ in $\cf_1$ is also a member of $\cf_2$. 
(The definition of a $C^\infty$-manifold $X$ means that a maximal atlas has been chosen on the set $X$.)
\end{definition}

Unless otherwise stated, $X$ denotes an $n$-dimensional $C^\infty$-manifold and $\cf$ is the maximal atlas.
A partition of unity $1=\sum_{j\in\N}\psi_j(x)$ with $\psi_j\in C_0^\infty(X)$ and $\psi_j(x)\ge 0$ for $x\in X$ is said to be subordinate to $\cf$ (instead of to the covering $X=\bigcup_{\kappa\in\cf}X_\kappa$), 
when for each $j\in\N$ there exists a chart $\kappa(j)\in\cf$ such that $\supp \psi_j\subset X_{\kappa(j)}$.
It is locally finite, when $1=\sum\psi_j(x)$ for every $x\in X$ has only finitely many non-trivial terms in some neighbourhood of $x$. Note that for each compact set $K\subset X$, this finiteness extends to an open set $U\supset K$.

We recall the definition of a distribution on a $C^\infty$-manifold, using the notation $\varphi^* u$ 
for the pullback of a distribution $u$ by a function $\varphi$ \cite{hor90}; when $u$ is a
function then $\varphi^* u=u\circ \varphi$. 

\begin{definition}\label{def:distributionMfd}
The space $\cd'(X)$ consists of the families $\{u_\kappa\}_{\kappa\in\cf}$, where $u_\kappa\in \cd'(\widetilde{X}_\kappa)$ and
which for all $\kappa,\kappa_1\in\cf$ fulfil
\begin{equation}
\label{eq:distFamily}
  u_{\kappa_1} = (\kappa\circ \kappa_1^{-1})^* u_\kappa \quad \mbox{on}\quad \kappa_1(X_\kappa\cap X_{\kappa_1}).
\end{equation}
($\cd'(X)$ only identifies with the dual of $C_0^\infty(X)$ if there is a positive density on $X$; cf.~\cite[Ch.~6]{hor90}.)
\end{definition}

Each $u \in C^k(X)$, $k\in \N_0$, can be identified with the family 
$u_\kappa:=u\circ\kappa^{-1}$ of functions in $C^k(\widetilde{X}_\kappa)$, which evidently transform as in~\eqref{eq:distFamily}.
Thus $C^k(X)\subset \cd'(X)$ is obvious. 
For any $u\in\cd'(X)$, the notation $u\circ\kappa^{-1}$ is also used to denote $u_\kappa$.

In~\eqref{eq:distFamily} restriction of e.g.~$u_\kappa$ to $\kappa(X_\kappa\cap X_{\kappa_1})$ is tacitly understood.
To ease notation we will in the rest of the paper, when composing with a chart, suppress such restriction to the chart's co-domain.

\begin{lemma}[{\cite[Thm.~6.3.4]{hor90}}]\label{lem:hor634}
For any atlas $\cf_1\subset\cf$, each family $\{u_\kappa\}_{\kappa\in\cf_1}$ of elements $u_\kappa\in\cd'(\widetilde X_\kappa)$ fulfilling~\eqref{eq:distFamily} for $\kappa,\kappa_1\in\cf_1$ is obtained from a unique $v\in\cd'(X)$ by ``restriction" to $\cf_1$, i.e.~$v\circ\kappa^{-1} = u_\kappa$ for every $\kappa\in\cf_1$.
\end{lemma}

So if an open set $U\subset\Rn$ is seen as a manifold $X$, then $\cf_1 = \{\id_U\}$ at once gives $\cd'(U)\hookrightarrow\cd'(X)$;
the surjectivity of this map follows by gluing together, cf.~\cite[Thm.~2.2.4]{hor90}.

For $Y\subset X$ open, the restriction of $u\in\cd'(X)$ to $Y$ is the family $r_Y u := \{ r_{\kappa(Y\cap
  X_\kappa)} u_\kappa\}$ with $\kappa$ running through the charts in $\cf$ for which $X_\kappa\cap
Y\neq \emptyset$. If we only consider an atlas $\cf_1\subset\cf$, then the
corresponding subfamily identifies with a distribution $u_Y\in\cd'(Y)$, cf.~Lemma~\ref{lem:hor634},
and since this is unique $r_Y u = u_Y$, i.e.~it suffices to consider an arbitrary atlas when
determining the restriction of a distribution.  

A distribution $u\in\cd'(X)$ is said to be 0 on an open set $Y\subset X$ if $r_Y u = 0$. Using this,
\begin{equation}\label{eq:supportDist}
\supp u := X\setminus \bigcup \{\, Y\subset X \text{ open}\, |\, u = 0 \text{ on } Y\},
\end{equation}
and it is easily seen that for any atlas $\cf_1\subset\cf$,
\begin{equation}\label{eq:equivSupport}
\supp u = \bigcup_{\kappa_1\in\cf_1} \kappa_1^{-1}(\supp u_{\kappa_1}).
\end{equation}
The space $\ce'(X)$ consists of the distributions $u\in\cd'(X)$ having compact support, while $\ce'(K)$ for an arbitrary $K\subset X$ consists of the $u\in\ce'(X)$ with $\supp u\subset K$.
Any $u\in\ce'(Y)$, where $Y\subset X$ is open, has an ``extension by 0"; even locally in a chart:

\begin{corollary}\label{cor:extensionByZero}
When $Y\subset X$ is open and $u\in \ce'(Y)$, then there exists $v\in\ce'(X)$ such that $r_Y v = u$ and $\supp v = \supp u$.
Moreover, when given $u_\kappa\in \ce'(\widetilde{X}_\kappa)$ for a single $\kappa\in\cf$, then there exists $v\in\ce'(X)$ such that $v_\kappa = u_\kappa$ and $\supp v = \kappa^{-1} (\supp u_\kappa)$.
\end{corollary}

\begin{proof}
In the case that $\supp u \subset X_\kappa\subset Y$ for some $\kappa\in\cf$, then there exists an open set $U \subset X$ such that $\supp u \subset U\subset\overline{U}\subset X_\kappa$ ($X$ is normal). The family $\cf_1 := \{\kappa\}\cup \{\kappa_1\in\cf\, |\, \overline U \cap X_{\kappa_1} = \emptyset\}$ is an atlas, since its domains cover $X$.
Setting $v_\kappa = u_\kappa$ and $v_{\kappa_1} = 0$ for the other $\kappa_1\in\cf_1$, the family $\{v_{\kappa_1}\}_{\kappa_1\in\cf_1}$ clearly transforms as in~\eqref{eq:distFamily}, hence defines a $v\in\cd'(X)$, cf.~Lemma~\ref{lem:hor634}. 
From~\eqref{eq:equivSupport} it is clear that $\supp v = \supp u$; and $r_Y v = u$ is evident in the atlas $\cf_1$.

In the general case, we use that any $u\in\ce'(Y)$ can be written as a finite sum $u=\sum \psi_j u$,
where $1=\sum\psi_j$ is a locally finite partition of unity subordinate to the atlas $\{\,
\kappa{|_{Y \cap X_\kappa}}\, |\, \kappa\in\cf: Y\cap X_\kappa\neq\emptyset\}$ on $Y$. Since $\supp
\psi_j u \subset Y \cap X_{\kappa(j)}$ is compact for each summand, the above gives the existence of
a $v_j \in\cd'(X)$ such that $r_Y v_j = \psi_j u$ and $\supp v_j = \supp \psi_j u$. Because the
restriction operator is linear, taking $v=\sum v_j$ proves the statement. 

For the last part, consider $X_\kappa$ as a manifold with the atlas containing only the chart $\kappa$. Lemma~\ref{lem:hor634} gives a $w\in \cd'(X_\kappa)$ such that $w_\kappa = u_\kappa$,
hence the special case above applied to $w$ and $Y=X_\kappa$ gives the existence of some $v\in\ce'(X)$
such that $v_\kappa = w_\kappa$, and $\supp v = \supp w$.
\end{proof}

\subsection{Isotropic Lizorkin--Triebel Spaces on Manifolds}\label{sec:LTspacesMfd}
Since we later need a few isotropic results, and since the proofs are much cleaner for isotropic spaces, we shall fix ideas in this section by working with arbitrary
$s\in\mathbb{R}$, $0<p<\infty$ and $0<q\leq\infty$.
Let us add that most references on isotropic spaces over manifolds just describe the outcome without referring directly to the general definitions in~\cite[Ch.~6]{hor90}, thus being inadequate for our generalisations here.

\subsubsection{Manifolds in General}
We first recall that when $U\subset\Rn$ is open,  
$u\in\cd'(U)$ is said to belong to the Lizorkin--Triebel space $\overline F^s_{p,q}( U)$ locally, 
if $\varphi u\in \overline F^s_{p,q}( U)$ for all $\varphi\in C^\infty_0(U)$; the set of such elements is denoted $F^s_{p,q;\loc}( U)$.
Here we use the notation without bar, since $\varphi$ has compact support in $U$. 
This can be generalised to

\begin{definition}\label{def:localLTmanifold}
The local Lizorkin--Triebel space $F^s_{p,q;\loc}(X)$ consists of the $u\in\cd'(X)$ such that $u_\kappa\in  F^s_{p,q;\loc}(\widetilde{X}_\kappa)$ for every $\kappa\in\cf$.
\end{definition}

For $u\in\cd'(X)$ to belong to $F^s_{p,q;\loc}(X)$, it suffices that $u_{\kappa_1}$ is in $ F^s_{p,q;\loc}(\widetilde{X}_{\kappa_1})$ for each $\kappa_1$ in an atlas $\cf_1\subset\cf$. Indeed given $\varphi\in C_0^\infty(\widetilde{X}_\kappa)$, a partition of unity yields a reduction to the case where $\supp(\varphi\circ\kappa)\subset X_\kappa \cap X_{\kappa_1}$, 
and the transition rule in~\eqref{eq:distFamily} gives
\begin{equation*}
(\kappa\circ\kappa_1^{-1})^* (\varphi u_\kappa) = \varphi\circ(\kappa\circ\kappa_1^{-1}) u_{\kappa_1}.
\end{equation*}
Since $\varphi\circ(\kappa\circ\kappa_1^{-1})$ is in $C_0^\infty\left(\kappa_1(X_\kappa\cap X_{\kappa_1})\right)$, the product by $u_{\kappa_1}$ 
is in $\overline F^s_{p,q}\left(\kappa_1(X_\kappa\cap X_{\kappa_1})\right)$ by assumption on $\cf_1$; so by Theorem~\ref{thm:isotropicCompactSupport} one has $\varphi u_\kappa \in \overline F^s_{p,q}(\widetilde{X}_\kappa)$.

For example, when $X$ is an open set $U\subset\Rn$, the identification $\cd'(X)\simeq \cd'(U)$ implies that $F^s_{p,q;\loc}(X)\simeq  F^s_{p,q;\loc}(U)$ as it according to the above suffices to consider the atlas $\{\id_U\}$. 

For a partition of unity $1=\sum_{j=1}^\infty \psi_j$ subordinate to $\cf$, we shall for brevity use
\begin{equation*}
\widetilde\psi_j := \psi_j \circ \kappa(j)^{-1}.
\end{equation*}
The partition is of course already subordinate to $\cf_1 := \{\, \kappa(j)\, |\, j\in\N\}$, which by the above suffices for determining $F^s_{p,q;\loc}(X)$. This is moreover true, when the cut-off functions $\widetilde\psi_j$ of a locally finite partition of unity are invoked:

\begin{lemma}\label{lem:equivalentCondition}
A distribution $u\in\cd'(X)$ belongs to $F^s_{p,q;\loc}(X)$ if and only if 
\begin{equation}
\label{eq:condLTOnManifold}
\widetilde\psi_j u_{\kappa(j)} \in \overline F^s_{p,q}(\widetilde{X}_{\kappa(j)}),\quad j\in\N.
\end{equation}
\end{lemma}

\begin{proof}
Since $\widetilde\psi_j\in C_0^\infty(\widetilde X_{\kappa(j)})$, this condition is necessary for $u$ to be in $F^s_{p,q;\loc}(X)$.

Conversely, for an arbitrary $\varphi\in C_0^\infty(\widetilde{X}_\kappa)$ we obtain $\varphi u_\kappa = \sum_{j\in I} \psi_j\circ\kappa^{-1} \varphi u_\kappa$ with summation over a finite index set $I\subset\N$, 
because $(\psi_j)_{j\in\N}$ is locally finite.
As $\supp (\psi_j\circ\kappa^{-1} \varphi) \subset \kappa (X_\kappa\cap X_{\kappa(j)})$, \cite[Lem.~8]{HJS13b} and then
Theorem~\ref{thm:isotropicCompactSupport} applied to $\kappa(j)\circ\kappa^{-1}$ yields, cf.~\eqref{eq:distFamily},
\begin{equation}
\label{eq:estimateLocalu_kappa}
\|\, \varphi u_\kappa\, | \overline F^s_{p,q}(\widetilde{X}_\kappa)\|
\leq c_\kappa \sum_{j\in I} \big\|\, \widetilde\psi_j\cdot (\varphi\circ \kappa\circ\kappa(j)^{-1})\cdot u_{\kappa(j)}\, \big| \overline F^s_{p,q}(\kappa(j)(X_\kappa\cap X_{\kappa(j)}))\big\|.
\end{equation}
After multiplication with $\chi_j\in C_0^\infty(\widetilde{X}_{\kappa(j)})$ chosen such that $\chi_j\equiv 1$ on $\supp \widetilde\psi_j$, 
we obtain by applying Lemma~\ref{multTraces} with some $s_1>s$ satisfying~\eqref{eq:assumptionMult} and suppressing extension by 0 to $\Rn$ that
\begin{equation}
\label{eq:normEstimatesSubsets}
\|\, \varphi u_\kappa\, | \overline F^s_{p,q}(\widetilde{X}_\kappa)\| 
\leq c_\kappa \sum_{j\in I}\, \|\, 
\varphi\circ \kappa\circ\kappa(j)^{-1}\chi_j\, | B^{s_1}_{\infty,\infty}(\Rn) \|
\|\, \widetilde\psi_j u_{\kappa(j)}\, | \overline F^s_{p,q}(\widetilde{X}_{\kappa(j)}) \|.
\end{equation}
The right-hand side is by~\eqref{eq:condLTOnManifold} finite, hence $u_\kappa\in  F^s_{p,q;\loc}(\widetilde{X}_\kappa)$ for each~$\kappa\in\cf$.
\end{proof}

The space $F^s_{p,q;\loc}(X)$ can be topologised through a separating family of quasi-seminorms, 
\begin{equation}\label{eq:quasiSeminorms}
\mu_j(u):= \|\, \widetilde\psi_j u_{\kappa(j)}\, | \overline F^s_{p,q}(\widetilde{X}_{\kappa(j)}) \|,\quad j\in \N.
\end{equation}
Indeed, if $u$ is non-zero in $F^s_{p,q;\loc}(X)$, then there exists $\kappa\in\cf$ and $\varphi\in C_0^\infty(\widetilde{X}_\kappa)$ such that $\varphi u_\kappa \neq 0$, i.e.\ $\|\, \varphi u_\kappa\, | \overline F^s_{p,q}(\widetilde{X}_\kappa)\|>0$. So~\eqref{eq:normEstimatesSubsets} gives that $\mu_j(u)>0$ for at least one $j\in\N$.

Going a step further, one obtains an equivalent family of quasi-seminorms even for a ``restricted" family $\{v_{\kappa_1}\}_{\kappa_1\in\cf_1}$:

\begin{lemma}\label{lem:compatibleAtlas}
Let $1=\sum\varphi_k$ be a locally finite partition of unity subordinate to some atlas $\cf_1\subset\cf$ and let $\widetilde\varphi_k = \varphi_k\circ\kappa_1(k)^{-1}$.
When a family of distributions $v_{\kappa_1} \in \cd'(\widetilde{X}_{\kappa_1})$, $\kappa_1\in\cf_1$, transforms as in~\eqref{eq:distFamily} and $\widetilde\varphi_k v_{\kappa_1(k)} \in \overline F^s_{p,q}(\widetilde{X}_{\kappa_1(k)})$ for every $k\in\N$, then there exists a unique $u\in F^s_{p,q;\loc}(X)$ such that $u_{\kappa_1} = v_{\kappa_1}$ for all $\kappa_1\in \cf_1$ and
\begin{align}\label{eq:equivalenceSeminorms}
\|\, \widetilde\psi_j u_{\kappa(j)}\, | \overline F^s_{p,q}(\widetilde{X}_{\kappa(j)})\|
\leq c_j \max \|\, \widetilde\varphi_k u_{\kappa_1(k)}\, |\overline F^s_{p,q}(\widetilde{X}_{\kappa_1(k)})\| ,\quad j\in\N,
\end{align}
with maximum over $k\in\N$ for which $\supp \psi_j\cap\supp \varphi_k\neq\emptyset$, cf.~\eqref{eq:supportDist}.
\end{lemma}

\begin{proof}
There exists a unique distribution $u\in \cd'(X)$ such that $u_{\kappa_1}=v_{\kappa_1}$ for all $\kappa_1\in\cf_1$, cf.~Lemma~\ref{lem:hor634}, and
using~\eqref{eq:normEstimatesSubsets} with $\varphi = \widetilde\psi_j$ and $1=\sum \varphi_k$ as the partition of unity readily shows~\eqref{eq:equivalenceSeminorms}. Consequently $u\in F^s_{p,q;\loc}(X)$.
\end{proof}

Since the opposite inequality of~\eqref{eq:equivalenceSeminorms} can be shown similarly from~\eqref{eq:estimateLocalu_kappa}--\eqref{eq:normEstimatesSubsets}, we obtain

\begin{corollary}\label{cor:equivalentSpaces}
The space $F^s_{p,q;\loc}(X)$ can be equivalently defined from any atlas $\cf_1\subset\cf$. Lemma~\ref{lem:equivalentCondition} holds for any locally finite partition of unity subordinate to $\cf_1$, and the resulting system of quasi-seminorms is equivalent to~\eqref{eq:quasiSeminorms}.
\end{corollary}

As a preparation we include an obvious consequence of the proof of Corollary~\ref{cor:extensionByZero}:

\begin{corollary}\label{cor:extensionByZeroLTspaces}
When given $u_\kappa\in\ce'(\widetilde X_\kappa)\cap \overline F^s_{p,q}(\widetilde X_\kappa)$ for a single $\kappa\in\cf$, then there exists $v\in\ce'(X)\cap F^s_{p,q;\loc}(X)$ such that $v_\kappa=u_\kappa$ and $\supp v = \kappa^{-1}(\supp u_\kappa)$.
\end{corollary}

When an open set $U\subset\Rn$ is seen as a manifold $X$, then $ F^s_{p,q;\loc}(U)$ obviously coincides with  $F^s_{p,q;\loc}(X)$, since it by Corollary~\ref{cor:equivalentSpaces} suffices to consider $\cf_1=\{\id_U\}$ and any partition of unity $1=\sum_{j=1}^\infty \psi_j$ on $U$. On $ F^s_{p,q;\loc}(U)$, the family in~\eqref{eq:quasiSeminorms} gives the usual structure of a Fr\'{e}chet space if $p,q\ge 1$, and in general we have:

\begin{theorem}\label{thm:frechetSpace}
The space $F^s_{p,q;\loc}(X)$ is a complete topological vector space with a translation invariant metric; for $p,q\ge 1$ it is locally convex, hence a Fr\'{e}chet space.
\end{theorem}

\begin{proof}
It follows straightforwardly from~\cite[Thm.~B.5]{gru09}, which is based on a separating family of seminorms, that the separating family $(\mu_j^d)_{j\in\N}$, whereby $d:=\min(1,p,q)$, 
of subadditive functionals can be used to construct a topology, which turns $F^s_{p,q;\loc}(X)$ into a topological vector space. Indeed, only a minor modification in the proof of continuity of scalar multiplication is needed, since the $\mu^d_j$ are not positive homogeneous --- unless $p,q\ge 1$, and in this case the positive homogeneity implies that $F^s_{p,q;\loc}(X)$ is locally convex.

A translation invariant metric can be defined as in~\cite[Thm.~B.9]{gru09}, i.e.~
\begin{equation}\label{eq:translationInvariantMetricIsotropic}
d'(u,v) = \sum_{j=1}^\infty \frac 1{2^j} \frac{\mu_j(u-v)^d}{1+\mu_j(u-v)^d},
\end{equation}
and the arguments there immediately yield that $d'$ defines the same topology as $(\mu^d_j)_{j\in\N}$.

For an arbitrary Cauchy sequence $(u_m)$ in $F^s_{p,q;\loc}(X)$, the sequence $(\widetilde\psi_j u_{m,\kappa(j)})$, where $u_{m,\kappa(j)}:=u_m\circ\kappa(j)^{-1}$, is Cauchy in $\overline F^s_{p,q}(\widetilde X_{\kappa(j)})$ for each $j\in\N$. Since this space is complete, there exists $\widetilde v_{\kappa(j)}\in \overline F^s_{p,q}(\widetilde X_{\kappa(j)})$ such that 
\begin{equation}\label{eq:convergenceInLocalCoor}
  \|\, \widetilde\psi_j u_{m,\kappa(j)} - \widetilde v_{\kappa(j)}\, | \overline F^s_{p,q}(\widetilde X_{\kappa(j)})\| 
  \to 0\quad\text{for}\quad m\to\infty.
\end{equation}
Clearly $\widetilde v_{\kappa(j)}\in\ce'(\widetilde X_{\kappa(j)})$, hence it follows from Corollary~\ref{cor:extensionByZeroLTspaces} that there exists a~$v^{(\kappa(j))}\in\ce'(X)\cap F^s_{p,q;\loc}(X)$ so that $\supp v^{(\kappa(j))} = \kappa(j)^{-1}(\supp\widetilde v_{\kappa(j)})\subset\supp\psi_j$ and $v^{(\kappa(j))}_{\kappa(j)} = \widetilde v_{\kappa(j)}$. 

To find a limit for $(u_m)$, we note that $\widetilde u_{\kappa(j)} :=\sum_{l\in\N} v^{(\kappa(l))}_{\kappa(j)}$ is well defined in $\cd'(\widetilde X_{\kappa(j)})$, since
on every compact set $K\subset \widetilde X_{\kappa(j)}$ there are only finitely many non-trivial terms. This family transforms as in~\eqref{eq:distFamily}, for in $\cd'(\kappa(j)(X_{\kappa(j)}\cap X_{\kappa(k)}))$,
\begin{equation}\label{eq:transformationIsotropicCase}
  \widetilde u_{\kappa(k)}\circ\kappa(k)\circ\kappa(j)^{-1}=
  \sum_l v^{(\kappa(l))}_{\kappa(k)}\circ\kappa(k)\circ\kappa(j)^{-1}= \sum_l v^{(\kappa(l))}_{\kappa(j)}=\widetilde u_{\kappa(j)}.
\end{equation}
Since $\widetilde\psi_j\widetilde u_{\kappa(j)}=\sum_l \widetilde\psi_j v^{(\kappa(l))}_{\kappa(j)}$ has finitely many summands, hence yields an element of $\overline F^s_{p,q}(\widetilde X_{\kappa(j)})$, existence of $u\in F^s_{p,q;\loc}(X)$
with $u_{\kappa(j)}=\widetilde u_{\kappa(j)}$ for all $j$ follows from Lemma~\ref{lem:compatibleAtlas}.

To show the convergence of $u_m$ to $u$ in $F^s_{p,q;\loc}(X)$, we rely on extra copies of the locally finite partition of unity to estimate by finitely many terms,
\begin{align*}
  \mu_j(u_m-u)^d
  \leq 
   \sum_{\supp \psi_j \cap \supp \psi_k \neq \emptyset}
   \big\|\,\widetilde\psi_j \big( \psi_k\circ\kappa(j)^{-1}  u_{m,\kappa(j)} - v^{(\kappa(k))}_{\kappa(j)}\big)\, 
   \big| \overline F^s_{p,q}(\widetilde X_{\kappa(j)}) \big\|^d.
\end{align*}
For $k\neq j$ the domains can clearly be changed to $\kappa(j)(X_{\kappa(j)} \cap X_{\kappa(k)})$, since $v^{(\kappa(k))}$ and the $\psi_k$ have compact support in $X_{\kappa(k)}$. Using Theorem~\ref{thm:isotropicCompactSupport}, each term can then be estimated by,
\begin{align*}
  c \big\|\, \psi_j \circ \kappa(k)^{-1} (\widetilde\psi_k u_{m,\kappa(k)}-\widetilde v_{\kappa(k)})\,
  \big| \overline F^s_{p,q}\big(\kappa(k)(X_{\kappa(j)} \cap X_{\kappa(k)})\big) \big\|^d.
\end{align*}
By means of a cut-off function equal to 1 on the compact supports, one can extend by 0 to $\Rn$ and apply Lemma~\ref{multTraces} and~\cite[Lem.~8]{HJS13b}, which yields
\begin{align*}
  \mu_j(u_m-u)^d\leq
  c\sum_{\supp \psi_j \cap \supp \psi_k \neq\emptyset}
  \|\, \widetilde\psi_k u_{m,\kappa(k)} - \widetilde v_{\kappa(k)}\,|\overline F^s_{p,q}(\widetilde X_{\kappa(k)})\|^d.
\end{align*}
Each term converges to 0, cf.~\eqref{eq:convergenceInLocalCoor}, hence $F^s_{p,q;\loc}(X)$ is complete.
\end{proof}

\subsubsection{Compact Manifolds} \label{ssect:CompManf}
For trace operators on cylinders, compact manifolds are of special interest, since the intersection of the curved and the flat boundary is often of such nature. 

When $X$ is compact there exists a finite atlas $\cf_0$ and a partition of unity $1=\sum_{\kappa\in\cf_0} \psi_\kappa$ such that $\supp \psi_\kappa\subset X_\kappa$ is compact for each $\kappa\in\cf_0$. 
The space $F^s_{p,q;\loc}(X)$ is in this case just denoted $F^s_{p,q}(X)$, since the elements satisfy a global condition according to

\begin{theorem}\label{thm:compact}
When $X$ is a compact $C^\infty$-manifold, then $F^s_{p,q}(X)$ is a quasi-Banach space (normed if $p,q\geq 1$) when equipped with
\begin{equation}\label{eq:quasinormOnLocSpace}
\|\, u\, | F^s_{p,q}(X)\| := \Big( \sum_{\kappa\in\cf_0} \|\, \widetilde\psi_\kappa u_\kappa\, |\overline F^s_{p,q}(\widetilde{X}_\kappa)\|^d\Big)^{1/d} , \quad d := \min(1,p,q),
\end{equation}
and $\|\, \cdot\, | F^s_{p,q}(X) \|^d$ is subadditive.
\end{theorem}

\begin{proof}
Positive homogeneity and subadditivity are inherited from the quasi-norms on the $\overline F^s_{p,q}(\widetilde{X}_\kappa)$ and then
the quasi-triangle inequality follows for $d<1$ by using dual exponents $\frac{1}{d},\frac1{1-d}$,
\begin{equation*}
\|\, u+v\, | F^s_{p,q}(X)\| \leq 
2^{\frac{1-d}d} \big( \|\, u\, | F^s_{p,q}(X)\| 
+ \|\, v\, | F^s_{p,q}(X)\| \big), \quad u,v \in F^s_{p,q}(X).
\end{equation*}

For any $u\in F^s_{p,q}(X)$ with $\|\, u\, |F^s_{p,q}(X)\|=0$, clearly $\widetilde\psi_\kappa u_\kappa=0$ on $\widetilde{X}_\kappa$, $\kappa\in\cf_0$.
Also $\psi_\kappa \circ \kappa_1^{-1} u_{\kappa_1} = 0$ for $\kappa,\kappa_1\in\cf_0$ with $X_\kappa\cap X_{\kappa_1} \neq \emptyset$, as~\eqref{eq:distFamily} applies on $\kappa_1(X_\kappa\cap X_{\kappa_1})$.
Therefore $u_{\kappa_1} = \sum_{\kappa\in\cf_0} (\psi_\kappa\circ\kappa_1^{-1}) u_{\kappa_1} = 0$ for all $\kappa_1\in\cf_0$, hence $u=0$.

Completeness follows from Theorem~\ref{thm:frechetSpace}, since we for $X$ compact have a partition
of unity with only finitely many non-zero elements, hence the topology there is equal to the one
defined from~\eqref{eq:quasinormOnLocSpace}. 
\end{proof}

\begin{remark}
  Alternatively, one could work with an additional Riemannian structure as done by Triebel~\cite{tri92}.
\end{remark}

\subsection{Isotropic Besov Spaces on Manifolds}
\label{sec:isotropicBesovManifolds}
For later reference, it is briefly mentioned that all the definitions and results in Section~\ref{sec:LTspacesMfd} can be adapted to Besov spaces $B^s_{p,q;\loc}(X)$. E.g.~they are complete, when endowed with the quasi-seminorms 
\begin{equation*}
\mu_j(u):= \|\, \widetilde \psi_j u_{\kappa(j)}\, | \overline B^s_{p,q}(\widetilde X_{\kappa(j)})\|, \quad j\in\N,
\end{equation*} 
and for $p,q\ge 1$ even Fr\'{e}chet spaces.
Moreover, when $X$ is compact, $B^s_{p,q}(X)$ is a quasi-Banach space under the norm
\begin{equation}\label{eq:theNormOnBesovOverManifold}
\|\, u\, | B^s_{p,q}(X)\| := \Big( \sum_{\kappa\in\cf_0} \|\, \widetilde\psi_\kappa u_\kappa\, | \overline B^s_{p,q}(\widetilde{X}_\kappa)\|^d\Big)^{1/d} , \quad d := \min(1,p,q).
\end{equation}
Indeed, this results since the arguments in Section~\ref{sec:LTspacesMfd} merely rely on Lemma~\ref{multTraces} and Theorem~\ref{thm:isotropicCompactSupport}. For one thing, the paramultiplication result in Lemma~\ref{multTraces} is simply replaced by a Besov version, cf.~\cite{JJ94mlt},~\cite{RS} or~\cite[4.2.2]{tri92},
while we now indicate the needed modifications of the invariance result in Theorem~\ref{thm:isotropicCompactSupport}:

The proof of~\cite[Thm.~6]{HJS13b}, i.e.\ Theorem~\ref{thm:isotropicCompactSupport}, was divided into two steps. For large $s$, the arguments carry over to $B^s_{p,q}$ using~\cite[Sec.~2.7.1,~Rem.~2]{tri83} instead of~\cite[Lem.~1(iii)]{HJS13b} and also using the characterisation of isotropic Besov spaces by kernels of local means, cf.~\cite[Thm.~BPT]{ryc99} or \cite[Thm.~1.10]{tri06}. This characterisation also readily gives a variant of~\cite[Lem.~8]{HJS13b} for $B^s_{p,q}$. 

Then~\cite[Lem.~2]{HJS13b} is replaced by~\cite[Cor.~3.3]{HJS13a} and it is noted that~\cite[Thm.~4.4]{HJS13a} carries over to the quasi-norm $\|\,\cdot\,|\ell_q(L_p)\|$. Indeed, the only modification is to apply the inequality in~\cite[(21)]{ryc99} instead of~\cite[Lem.~2.7]{HJS13a} in the last line of the proof. 

Finally, the reference to~\cite[Thm.~2]{HJS13b} is changed to~\cite[(23)]{ryc99}. However, Rychkov's starting point \cite[(34)]{ryc99} was flawed, as mentioned in~\cite[Rem.~1.1]{HJS13a}, but it can be derived from our anisotropic version in~\cite[Prop.~4.6]{HJS13a},
as the elementary inequality $\prod (1+|2^{ja_l}z_l|)^{r_0}\ge (1+|2^{j\vec a}z|)^{r_0}$
brings us back at once to the isotropic maximal functions. Our anisotropic 
dilations by $2^{j\vec a}$ disappear when invoking the majorant
property of the maximal function (cf.\ its proof in~\cite[p.~57]{ste93}).

For small $s$, the lift operator 
\begin{equation}\label{eq:liftOperatorIsotropic} 
I_r u = \cf^{-1}(\langle\xi\rangle^r \cf u)
\end{equation}
is used instead of~\cite[(22)]{HJS13b}, because application of~\cite[2.3.8]{tri83} then readily gives an $h\in B^{s+r}_{p,q}(\Rn)$ for some even integer $r>s_1-s$ such that $e_V f = I_r h$. Since
\begin{equation*}
I_r h = (1-\Delta)^\frac{r}{2} h,
\end{equation*} 
the rest of the proof is easily carried over to a full proof of the fact that a $C^\infty$-bijection $\sigma: U\to V$ sends $\overline B^s_{p,q}( V)$ boundedly into $\overline B^s_{p,q}( U)$.

\subsection{Mixed-Norm Lizorkin--Triebel Spaces on Curved Boundaries}
\label{cylinders}
As a motivation, we first note that in case of evolution equations, the function $u(x,t)$, depending on the location $x$ in space and time $t$, 
describes to each $t$ in an open interval $I\subset\R$ the state of a system (as a function of $x$ in an open subset $\Omega\subset\Rn$). 
Thus solutions are sought in $C_{\operatorname{b}}(\R_t,L_{\vec r}(\Omega))$, say for some $\vec r\ge 1$, equipped with the norm 
\begin{equation*}
\sup_{t\in I}\|\, u(x,t)\, |L_{\vec r}(\Omega)\|.
\end{equation*}
Thus it should be natural to work in the scale of mixed-norm Lizorkin--Triebel spaces $\overline F^{s,\vec a}_{\vec p,q}(\Omega\times I)$, in which $t$ is taken as the outer integration variable in the norm of $L_{\vec p}$; i.e.~we take $t=x_{n+1}$ with associated weight $a_t$ and integral exponent $p_t$ (when it eases notation, they will be written with $n+1$ as index).

The results in Section~\ref{sec:LTspacesMfd} can be carried over to $\overline F^{s,\vec a}_{\vec p,q}(\Omega\times I)$ under the assumptions that
\begin{equation}\label{eq:conditionsOn_ap_anisotropic}
a_0:=a_1= \ldots  = a_n, \qquad p_0:= p_1 = \ldots =p_n,
\end{equation}
and that $\Omega$ is $C^\infty$ in the sense adopted e.g.~by~\cite{gru09}:

\begin{definition}\label{defn:smoothDomain}
An open set $\Omega\subset\R^n$ with boundary $\Gamma$ is $C^\infty$ (or smooth), when for each boundary point $x\in\Gamma$
there exists a diffeomorphism $\lambda$ defined on an open neighbourhood $U_\lambda\subset \R^n$ such that 
$\lambda: U_\lambda\to B(0,1)\subset\R^n$ is surjective and
\begin{alignat*}{1}
\lambda(x) &= 0,\\
\lambda(U_\lambda\cap\Omega) &= B(0,1)\cap \R^n_+,\\
\lambda(U_\lambda\cap\Gamma) &= B(0,1)\cap \R^{n-1},
\end{alignat*}
whereby $\Rn_+:=\{(x_1,\ldots,x_n)\in\Rn\, |\, x_n>0\}$ and $\R^{n-1}\simeq\R^{n-1}\times\{0\}$.
\end{definition}

The unit ball in $\Rn$ will below be denoted by $B$ and in $\R^{n-1}$ by $B'$.

\subsubsection{Curved Boundaries in General}
Let $I\subset\R$ be an open interval. As $\Gamma\times I$ is a $C^\infty$-manifold, $\cd'(\Gamma\times I)$ is a special case of Definition~\ref{def:distributionMfd} and therefore the results regarding distributions on manifolds in Section~\ref{sec:distributionsMfd} are applicable. The manifold can be equipped with e.g.~the atlas $\cf\times\mathcal{N}$, where  $\cf=\{\kappa\}$ and $\mathcal{N}=\{\eta\}$ are maximal atlases on $\Gamma$, respectively on $I$.

Locally finite partitions of unity $1=\sum\psi_j(x)$ and $1=\sum \varphi_l(t)$ subordinate to $\cf$, respectively to $\mathcal{N}$ give a locally finite partition of unity $1=\sum\psi_j\otimes\varphi_l$ on $\Gamma\times I$.
Note that we formally should sum with respect to a fixed enumeration of the pairs $(j,l)$ in $\N\times\N$, but for simplicity's sake we avoid this. (The sums are locally finite anyway.)
As above, we use the notation $\widetilde{\psi_j\otimes\varphi_l} = (\psi_j \otimes\varphi_l) \circ (\kappa(j)^{-1}\times\eta(l)^{-1})$.

Since the maximal atlas on $\Gamma\times I$ contains charts that do not respect the splitting into $t$ and the $x$-variables, it is not obviously useful for the anisotropic spaces. We have therefore chosen to adopt Lemma~\ref{lem:equivalentCondition} as our point of departure for the 
$F^{s,\vec a}_{\vec p,q}$-spaces on the curved boundary. Because $\Gamma$ is of dimension $n-1$, it is noted that the parameters $\vec a$, $\vec p$ for these spaces only contain $n$ entries.

\begin{definition}\label{def:LTspaceCurvedBoundary}
The space $F^{s,\vec a}_{\vec p,q;\loc}(\Gamma\times I)$ consists of all the $u\in\cd'(\Gamma\times I)$ for which
\begin{equation*}
\widetilde{\psi_j \otimes\varphi_l} u_{\kappa(j)\times\eta(l)} \in \overline F^{s,\vec a}_{\vec p,q}(\widetilde{\Gamma}_{\kappa(j)}\times\widetilde I_{\eta(l)}),\quad j,l\in\N.
\end{equation*}
\end{definition}

The family in~\eqref{eq:quasiSeminorms} and Corollary~\ref{cor:equivalentSpaces} adapted to this set-up, cf.~Theorem~\ref{thm:localVersion}, give that 
\begin{equation}\label{eq:familySeminormsAnisotropicCase}
\mu_{j,l}(u) := \big\|\, \widetilde{\psi_j\otimes\varphi_l} u_{\kappa(j)\times\eta(l)}\, \big| \overline F^{s,\vec a}_{\vec p,q}(\widetilde{\Gamma}_{\kappa(j)}\times\widetilde I_{\eta(l)}) \big\|,\quad j,l\in\N,
\end{equation}
is a separating family of quasi-seminorms and that $F^{s,\vec a}_{\vec p,q;\loc}(\Gamma\times I)$ can be equivalently defined from any atlas $\cf_1\times\mathcal{N}_1$, where $\cf_1\subset\cf$ and $\mathcal{N}_1\subset\mathcal{N}$; with the same topology.

\begin{theorem}\label{thm:frechetSpaceAnisotropic}
The space $F^{s,\vec a}_{\vec p,q;\loc}(\Gamma\times I)$ is a complete topological vector space with a translation invariant metric; for $p_0,p_t,q\ge 1$ it is locally convex, hence a Fr\'{e}chet space.
\end{theorem}

\begin{proof}
For $d:=\min(1,p_0,p_t,q)$ the separating family $(\mu_{j,l}^d)_{j,l\in\N}$, cf.~\eqref{eq:familySeminormsAnisotropicCase}, is used to construct a topology as in Theorem~\ref{thm:frechetSpace}. This immediately gives that $F^{s,\vec a}_{\vec p,q;\loc}(\Gamma\times I)$ is a topological vector space and even locally convex, when $d\ge 1$.

The metric is in this case obtained by letting the $\mu_{j,l}$ enter the summation formula for $d'(u,v)$, cf.~\eqref{eq:translationInvariantMetricIsotropic}, as any enumeration of the $(j,l)$ gives the same sum;
adapting the arguments in the proof of \cite[Thm.~B.9]{gru09} to two summation indices is straightforward.

Completeness follows as in the isotropic case, but with application of Theorem~\ref{thm:localVersion} instead of Theorem~\ref{thm:isotropicCompactSupport} when showing the convergence.
\end{proof}

\subsubsection{Curved Boundaries in the Compact Case}\label{subsec:curvedBoundaries}
When $\Gamma$ is compact, a finite atlas on the boundary can e.g.~be obtained 
from the composite maps $\kappa=\widetilde\gamma_{0,n}\circ\lambda$, where
$\widetilde\gamma_{0,n}\colon(x_1,\ldots,x_n)\mapsto (x_1,\ldots,x_{n-1},0)$ in local coordinates. 
Indeed, according to Definition~\ref{defn:smoothDomain} and the compactness of 
$\Gamma$ there exists on $\Gamma$ a finite open cover $\{U_\lambda\}$, where $\lambda$ runs 
in an index set $\Lambda$, which together with $\Omega$ gives 
an open cover of $\overline{\Omega}$.
Each $\lambda\in\Lambda$ induces a diffeomorphism $\kappa: \Gamma_{\kappa}\to B'$ on $\Gamma_{\kappa}:= U_\lambda\cap \Gamma$ by $\kappa = \widetilde\gamma_{0,n} \circ\lambda$.
These maps form an atlas $\cf_0$ on $\Gamma$ and thereby an atlas $\{\kappa\times\id_\R\}_{\kappa\in\cf_0}$ on $\Gamma\times \R$.

A partition of unity is obtained by using a function $\chi \in C^\infty(\Rn)$ such that $\chi \equiv 1$ on $\Omega\setminus \bigcup_\lambda U_\lambda$ to slightly generalise \cite[Thm.~2.16]{gru09}. This yields a family of functions $\{\psi_\lambda\}\cup\{\psi\}$ with $\psi_\lambda\in C_0^\infty(U_\lambda)$ and $\psi\in C^\infty(\Rn)$ with $\supp\psi\subset\Omega$ 
such that $\sum_\lambda \psi_\lambda(x) + \psi(x) = 1$ for $x\in \overline\Omega$.
(Existence of such $\chi$ is similar to~\cite[Cor.~2.14]{gru09}, where $K$ need not be compact.)

In addition, the functions $\psi_\kappa := \psi_\lambda|_\Gamma\in C_0^\infty(\Gamma_\kappa)$ constitute a finite partition of unity of $\Gamma$ subordinate to $\cf_0$. Hence $1=\sum_{\kappa\in\cf_0} \psi_\kappa\otimes\ind \R$, with $\ind \R$ denoting the characteristic function of $\R$, is a partition of unity on $\Gamma\times \R$.

Recalling that $F^{s,\vec a}_{\vec p,q;\loc}(\Gamma\times I)$ is equivalently defined from any atlas $\cf_1\times\mathcal{N}_1$, where $\cf_1\subset\cf$ and $\mathcal{N}_1\subset\mathcal{N}$, we obtain

\begin{theorem}\label{thm:LTcurvedBoundary}
Let $\Gamma$ be compact and $J\subset\R$ be a compact interval. 
The space
\begin{equation}\label{eq:LTcurvedBoundarySupport}
  \overset{\circ}{F}{}^{s,\vec a}_{\!\vec p,q}(\Gamma\times J)
  := \{\, u\in F^{s,\vec a}_{\vec p,q;\loc}(\Gamma\times\R) \, | \, \supp u\subset\Gamma\times J\} 
\end{equation}
is closed 
and a quasi-Banach space (normed if $\vec p,q\geq 1$), when equipped with the quasi-norm	
\begin{equation}\label{eq:quasinormOnGammaTimesI}
\|\, u\, | \overset{\circ}{F}{}^{s,\vec a}_{\!\vec p,q}(\Gamma\times J)\| 
:= \Big(\sum_{\kappa\in\cf_0} \big\|\, \widetilde{\psi_\kappa\otimes \ind \R} u_{\kappa\times\id_\R}\, \big| \overline F^{s,\vec a}_{\vec p,q}(B'\times\R)\big\|^d \Big)^{1/d},
\end{equation}
where $d:=\min (1,p_0,p_t,q)$. Furthermore, $\|\, \cdot\, |\overset{\circ}{F}{}^{s,\vec a}_{\!\vec p,q}(\Gamma\times J)\|^d$ is subadditive.
\end{theorem}

The support condition in~\eqref{eq:LTcurvedBoundarySupport} means
$\bigcup_{\kappa\in\cf_0} (\kappa^{-1}\times\id_\R)(\supp u_{\kappa\times\id_\R}) \subset\Gamma\times J$, 
cf.~\eqref{eq:equivSupport}, hence
\begin{equation}
\label{eq:compactSupportLTspaceAtlas}
  \supp u_{\kappa\times\id_\R}\subset B'\times J.
\end{equation}
This implies that each summand in~\eqref{eq:quasinormOnGammaTimesI} is finite, since the factor $\ind \R$ can be replaced by some $\chi\in C_0^\infty(\R)$ where $\chi=1$ on $J$; and this $\chi$ can be chosen as a finite sum of the $\varphi_l$ from Definition~\ref{def:LTspaceCurvedBoundary}.

\begin{proof}
By the same arguments as in Theorem~\ref{thm:compact}, the expression in~\eqref{eq:quasinormOnGammaTimesI} is a quasi-norm.
It gives the same topology on $\overset{\circ}{F}{}^{s,\vec a}_{\!\vec p,q}(\Gamma\times J)$ as the family $(\mu_{j,l}^d)_{j,l\in\N}$, since there exist $c_1,c_2>0$ such that for each $u\in \overset{\circ}{F}{}^{s,\vec a}_{\!\vec p,q}(\Gamma\times J)$, cf.~\eqref{eq:familySeminormsAnisotropicCase},
\begin{equation}\label{eq:doubleInequalityEquivNorms}
c_1 \mu_{j,l}(u)^d \le \|\, u\, | \overset{\circ}{F}{}^{s,\vec a}_{\!\vec p,q}(\Gamma\times J) \|^d \le c_2 \sideset{}{'}\sum_{j',l'\in\N} \mu_{j',l'}(u)^d,
\end{equation}
where the prime indicates that the summation is over finitely many integers.

Indeed, Theorem~\ref{thm:localVersion} yields that $\mu_{j,l}(u)^d$ is bounded from above by 
\begin{align*}
  \sum_{\kappa\in\cf_0} &\big\|\, (\psi_\kappa\otimes\ind \R) \circ (\kappa(j)^{-1}\times\eta(l)^{-1}) \widetilde{\psi_j\otimes\varphi_l} u_{\kappa(j)\times\eta(l)}\, 
  \big| \overline F^{s,\vec a}_{\vec p,q}\big(\kappa(j)\times\eta(l)(\Gamma_{\kappa(j)}\cap\Gamma_\kappa\times\R_{\eta(l)})\big) \big\|^d\\
  &\le c\sum_{\kappa\in\cf_0} \big\|\, \widetilde{\psi_\kappa\otimes\ind \R} (\psi_j\otimes\varphi_l) \circ(\kappa^{-1}\times\id_\R) u_{\kappa\times\id_\R}\, \big| \overline F^{s,\vec a}_{\vec p,q}(\widetilde\Gamma_\kappa\times\R) \big\|^d.
\end{align*}
Using for each $\kappa\in\cf_0$ some function $\chi_\kappa\in C^\infty_{L_\infty}(\Rn)$ chosen such that $\chi_\kappa = 1$ on $\supp\widetilde\psi_\kappa\cap \supp(\psi_j\circ\kappa^{-1})$ and $\supp \chi_\kappa \subset \kappa(\Gamma_{\kappa(j)}\cap\Gamma_\kappa)$, we extend by 0 to $\R^{n+1}$ and apply Lemma~\ref{multTraces} to obtain the left-hand inequality in~\eqref{eq:doubleInequalityEquivNorms}.

The right-hand inequality can be shown similarly by first replacing $\ind \R$ in~\eqref{eq:quasinormOnGammaTimesI} with some $\chi\in C_0^\infty(\R)$ where $\chi=1$ on $J$, as discussed above.

To prove that $\overset{\circ}{F}{}^{s,\vec a}_{\!\vec p,q}(\Gamma\times J)$ is closed, we consider an arbitrary sequence $(u_m)_{m\in\N}$, which belongs to $\overset{\circ}{F}{}^{s,\vec a}_{\!\vec p,q}(\Gamma\times J)$ and converges in $F^{s,\vec a}_{\vec p,q;\loc}(\Gamma\times\R)$ to some $u$.
Since $u_{m,\kappa\times\id_\R}$ converges to $u_{\kappa\times\id_\R}$ in $\cd'(B'\times\R)$ and~\eqref{eq:compactSupportLTspaceAtlas} holds for each $u_{m,\kappa\times\id_\R}$, it follows that $\supp u_{\kappa\times\id_\R}\subset B'\times J$,
whence $\supp u\subset\Gamma\times J$.

Completeness follows immediately, since each Cauchy sequence in $\overset{\circ}{F}{}^{s,\vec a}_{\!\vec p,q}(\Gamma\times J)$ converges to some $u$ in $F^{s,\vec a}_{\vec p,q;\loc}(\Gamma\times\R)$ and closedness of $\overset{\circ}{F}{}^{s,\vec a}_{\!\vec p,q}(\Gamma\times J)$ then gives that $\supp u \subset\Gamma\times J$.
\end{proof}

\section{Rychkov's Universal Extension Operator}
\label{sec:rychkovExtension}
A key ingredient in our construction of right-inverses to the trace operators is a modification of Rychkov's extension operator, introduced in~\cite{ryc99ext} for bounded or special Lipschitz domains $\Omega\subset\Rn$,
\begin{equation}\label{eq:rychkovExt}
 \ce_{\univ,\Omega}: \overline F^s_{p,q}(\Omega) \to F^s_{p,q}(\Rn).
\end{equation}
The linear and bounded operator $\ce_{\univ,\Omega}$ works for all $0<p<\infty$, $0<q\le\infty$, $s\in\R$, cf.~\cite[Thm~4.1]{ryc99ext}; and it also applies to Besov spaces ($p=\infty$ included). Thus it was termed a \emph{universal} extension operator.

In Section~\ref{subsec:traceCurvedBoundaryMathNr} below it will be clear that we for $\Omega=\Rn_+$ also need an extension operator for anisotropic spaces with mixed norms. We therefore modify $\ce_{\univ,\Omega}$ accordingly, relying on the proof strategy in~\cite{ryc99ext}, yet we present significant simplifications in the proof of Proposition~\ref{prop:variantCalderon} and add e.g.\ Proposition~\ref{prop:theAdditionToRychkovsPresentation}. The reader may choose to skip the proofs in a first reading. 

We take another approach than Rychkov when defining $\overline\cs'(\Rn_+)$; this can be justified by~\cite[Prop.~3.1]{ryc99ext} and the remark prior to it.
Similarly to~\cite[App.~A.4]{gru96} we use the following:

\begin{definition}\label{def:DefinitionOfTempDOnSubsets}
For any open set $U\subset\Rn$, the space $\overline\cs'(U)$ is defined as the set of $f\in\cd'(U)$ for which there exists $\widetilde f\in\cs'(\Rn)$ such that $r_U\widetilde f=f$. 

The spaces $\overset{\circ}{\cs}(\overline U)$ and $\overset{\circ}{\cs}{}'(\overline U)$ consist of the functions in $\cs(\Rn)$, respectively the distributions in $\cs'(\Rn)$ supported in $\overline U$.
\end{definition}

We define the convolution $\varphi*f(x)$ for $x\in\Rn_+$, when $f\in\overline\cs'(\Rn_+)$, cf.\ Definition~\ref{def:DefinitionOfTempDOnSubsets}, and when $\varphi\in\cs(\Rn)$ has its support in the opposite half-space $\overline\R^n_-$, that is~$\varphi\in\overset{\circ}{\cs}(\overline\R^n_-)$.
This is done by using an arbitrary extension $\widetilde f\in\cs'(\Rn)$ of $f$, i.e.
\begin{equation}\label{eq:definitionConvolutionSubset}
  \varphi*f(x) := \langle\widetilde f,\varphi(x-\cdot)\rangle, \quad x\in\Rn_+,
\end{equation}
which is well defined, since it as a function on $\Rn_+$ clearly does not depend on the choice of extension $\widetilde f$. 

First of all this is used in a variant of Calder\'{o}n's reproducing formula (cf.\ Proposition~\ref{prop:variantCalderon} below),
\begin{equation}\label{eq:variantCalderon}
  f = \sum_{j=0}^\infty \psi_j * (\varphi_j * f) \quad \text{in} \quad \cd'(\Rn_+),
\end{equation}
to give meaning to each $\psi_j * (\varphi_j * f)$; cf.~\eqref{eq:defOfSubscript_j} for the subscript notation.
Indeed, $\varphi_j*\widetilde f\in C^\infty(\Rn)\cap\cs'(\Rn)$ is an extension of $\varphi_j*f$ by~\eqref{eq:definitionConvolutionSubset}, so~\eqref{eq:definitionConvolutionSubset} also yields that by definition
\begin{equation}\label{eq:meaningOfTripleConvolution}
  \psi_j*(\varphi_j*f)(x) = \psi_j*(\varphi_j*\widetilde f)(x),\quad x\in\Rn_+.
\end{equation}

The idea in Rychkov's extension operator $\ce_{\univ}$ is to use \emph{another} extension of $\varphi_j*f$, namely
\begin{align*}
  e_+ (\varphi_j*f)(x) := 
\begin{cases}
0 &\text{for } x\in\overline\R^n_-,\\
\varphi_j*\widetilde f(x) &\text{for } x\in\Rn_+;
\end{cases}
\end{align*}
for brevity, we use $e_+ = e_{\Rn_+}$ and $r_+ = r_{\Rn_+}$. 
Indeed, $e_+(\varphi_j*f)$ is $C^\infty$ for $x_n\neq 0$, hence measurable, and in $L_{1,\loc}(\Rn)$. Moreover $e_+(\varphi_j*f)$ is in $\cs'(\Rn)$, because it is $O((1+|x|^2)^N)$ for a large $N$. 
So, using~\eqref{eq:definitionConvolutionSubset}, we obtain the alternative formula
\begin{equation}\label{eq:welldefinedConvolution}
  \psi_j*(\varphi_j*f)(x) = \psi_j*e_+(\varphi_j*f)(x),\quad x\in\Rn_+.
\end{equation}

Here we can exploit that $\psi_j*e_+(\varphi_j*f)$ is defined on all of $\Rn$, hence by substituting this into the right-hand side of~\eqref{eq:variantCalderon}, $\ce_{\univ}$ is obtained simply by letting $x$ run through not just $\Rn_+$, but $\Rn$, i.e.
\begin{equation}\label{eq:theInfiniteSeriesForEu}
\ce_{\univ}(f) := \sum_{j=0}^\infty \psi_j *e_+(\varphi_j*f) \quad\text{for}\enskip f\in\overline\cs'(\R^n_+).
\end{equation}
To make this description more precise, we first justify~\eqref{eq:variantCalderon}. 
So we recall that a function $\varphi\in\cs(\Rn)$ fulfils moment conditions of order $L_\varphi$, when
\begin{equation*}
  D^\alpha(\cf \varphi)(0) = 0 \quad\text{for}\quad |\alpha|\le L_\varphi.
\end{equation*}

\begin{proposition}\label{prop:variantCalderon}
There exist 4 functions $\varphi_0, \varphi, \psi_0, \psi\in\cs(\Rn)$ supported in $\Rn_-$ and with 
$L_\varphi = \infty= L_\psi$ such that~\eqref{eq:variantCalderon} holds for all $f\in\overline\cs'(\R^n_+)$.
\end{proposition}

\begin{proof}
We shall exploit the existence of a real-valued function $g\in\cs(\R)$ with 
\begin{equation*}
  \int g(t) \, dt \ne 0, \qquad \int t^k g(t)\, dt = 0\quad \text{for all} \quad k\in\N,
\end{equation*}
and $\supp g\subset [1,\infty[\,$. (This may be obtained as in~\cite[Thm.~4.1(a)]{ryc99ext}.)

With $\varphi_0(x) := g(-x_1)\cdots g(-x_n)/c^n$ for $c=\int g\, dt$,
the properties of $g$ immediately give
\begin{align*}
\supp \varphi_0 &\subset \{x\in\Rn\, |\, x_k<0, k=1,\ldots,n\},\nonumber\\
\int \varphi_0\, dx = 1&, \quad \int x^\alpha\varphi_0(x)\, dx = 0 \quad\text{for}\quad |\alpha|>0.\nonumber
\end{align*}
Thus the support of $\varphi:=\varphi_0- 2^{-|\vec a|}\varphi_0(2^{-\vec a}\cdot)$ lies in $\Rn_-$, and $L_\varphi=\infty$ since for $|\alpha|\ge 0$,
\begin{equation*}
  \int x^\alpha\varphi(x)\, dx 
= \int x^\alpha\varphi_0(x)\, dx - 2^{\vec a\cdot \alpha}\int x^\alpha\varphi_0(x)\, dx = 0.
\end{equation*}

The functions $\psi_0,\psi\in\cs(\Rn)$ are conveniently defined via $\cf$,
\begin{equation}\label{eq:defOfPsi0andPsi}\begin{split}
\widehat\psi_0(\xi)&=\widehat\varphi_0(\xi)(2-\widehat\varphi_0(\xi)^2),\quad\\
\widehat\psi(\xi)&=\big(\widehat\varphi_0(\xi)+\widehat\varphi_0(2^{\vec a}\xi)\big)\big(2-\widehat\varphi_0(\xi)^2-\widehat\varphi_0(2^{\vec a}\xi)^2\big).
\end{split}\end{equation}
Since
$\widehat\varphi_j(\xi)=\widehat\varphi(2^{-j\vec a}\xi)=\widehat\varphi_0(2^{-j\vec a}\xi)-\widehat\varphi_0(2^{(1-j)\vec a}\xi)$
for $j\ge 1$, we obtain by basic algebraic rules,
\begin{align*}
\widehat\psi_j(\xi)\widehat\varphi_j(\xi)
&=\big(2-\widehat\varphi_0(2^{-j\vec a}\xi)^2-\widehat\varphi_0(2^{(1-j)\vec a}\xi)^2\big)\big(\widehat\varphi_0(2^{-j\vec a}\xi)^2-\widehat\varphi_0(2^{(1-j)\vec a}\xi)^2\big)\\
&=2\big(\widehat\varphi_0(2^{-j\vec a}\xi)^2-\widehat\varphi_0(2^{(1-j)\vec a}\xi)^2\big)-\big(\widehat\varphi_0(2^{-j\vec a}\xi)^4-\widehat\varphi_0(2^{(1-j)\vec a}\xi)^4\big).
\end{align*}
This gives a telescopic sum:
\begin{align}\label{eq:telescopicSum}
  \sum_{j=0}^\infty \widehat\psi_j(\xi)\widehat\varphi_j(\xi)
= 2\lim_{N\to\infty} \widehat\varphi_0(2^{-N\vec a}\xi)^2-\lim_{N\to\infty}\widehat\varphi_0(2^{-N\vec a}\xi)^4
= 1,
\end{align}
using that $\widehat\varphi_0(0)=1$.
As the convergence is in $\cs'(\Rn)$, the inverse Fourier transformation yields
\begin{align}\label{eq:deltaIdentityByConvolution}
\sum_{j=0}^\infty \psi_j *\varphi_j = \delta.
\end{align}
The fact that $L_\psi=\infty$ is obvious from~\eqref{eq:defOfPsi0andPsi}, since $D^\alpha\widehat\varphi_0(0)=0$ for all $\alpha\in\N_0^n$.
The inclusion $\supp\psi_0\subset\Rn_-$ is clear, because $\psi_0=\varphi_0*(2\delta-\varphi_0*\varphi_0)$. Similarly $\supp\psi\subset\Rn_-$, since $\psi$ is a sum of convolutions of functions with such support.

To show~\eqref{eq:variantCalderon}, we note that when $\widetilde f\in\cs'(\Rn)$ fulfils $r_+\widetilde f=f$, then~\eqref{eq:deltaIdentityByConvolution} entails
\begin{align}\label{eq:convergenceOfSeriesForAllTempD}
\widetilde f= \sum_{j=0}^\infty \psi_j*(\varphi_j*\widetilde f)\quad\text{in}\quad\cs'(\Rn).
\end{align}
More precisely, to circumvent that the summands in~\eqref{eq:deltaIdentityByConvolution} need not have compact supports,
one can show that $\sum_{j<N} \widehat\psi_j\widehat\varphi_j\cf\widetilde f$ converges to $\cf\widetilde f$ in $\cs'(\Rn)$ by using~\eqref{eq:telescopicSum} and a function in $\cs(\Rn)$. 
Then~\eqref{eq:meaningOfTripleConvolution} gives,
\begin{align*}
f=r_+\widetilde f = \sum_{j=0}^\infty r_+\big(\psi_j*(\varphi_j*\widetilde f)\big) = \sum_{j=0}^\infty \psi_j*(\varphi_j*f)\quad \text{in} \quad \cd'(\Rn_+),
\end{align*}
in view of the continuity of $r_+: \cd'(\Rn)\to \cd'(\Rn_+)$.
\end{proof}

As a novelty, one can now show directly that $\ce_{\univ}$ has nice properties on the space $\overline\cs'(\Rn_+)$ of restricted temperate distributions:

\begin{proposition}\label{prop:theAdditionToRychkovsPresentation}
The series for $\ce_{\univ}(f)$ in~\eqref{eq:theInfiniteSeriesForEu}
converges in $\cs'(\Rn)$ for every~$f\in\overline\cs'(\R^n_+)$,
and the induced map $\ce_{\univ}: \overline\cs'(\Rn_+)\to \cs'(\Rn)$ is w$^*$-continuous. 
\end{proposition}

\begin{remark}
The space $\overline\cs'(\Rn_+)$ is endowed with the seminorms $f\mapsto |\langle\widetilde f,\varphi\rangle|$ 
for $\varphi\in\overset{\circ}{\cs}(\overline\R^n_+)$ and $r_+\widetilde f=f$, using the well-known fact that
it is the dual of $\overset{\circ}{\cs}(\overline\R^n_+)$.
(I.e.~$f_\nu\to 0$ means that for some (hence every) net $\widetilde f_\nu$ of extensions, one has
$\langle\widetilde f_\nu,\varphi\rangle\to 0$ for all $\varphi\in\overset{\circ}{\cs}(\overline\R^n_+)$.)
\end{remark}

\begin{proof}
It suffices according to the limit theorem for $\cs'$ 
to obtain convergence of the series 
\begin{equation}\label{eq:infiniteSumForMathcalEu}
  \sum_{j=0}^\infty \langle e_+(\varphi_j*f), \check\psi_j*\eta\rangle
\quad\text{for $\eta\in\cs(\Rn)$},
\end{equation}
where $\check\psi(x)=\psi(-x)$ as usual. Since $L_\psi=\infty$, it follows at once from 
\cite[Lem.~4.2]{HJS13a} that the second entry tends rapidly to zero, i.e.\ for any seminorm $p_M$
one has
\begin{equation}
  p_M(\check\psi_j*\eta)=O(2^{-jN}) \quad\text{for every $N>0$}.
\label{eq:bigOForSeminorm}
\end{equation}

For the first entries, a test against an arbitrary $\phi\in\cs(\Rn)$ gives,
for some $M$,
\begin{align}
    |\langle e_+(\varphi_j*f), \phi\rangle|
  &=\big|\int \langle \tilde f(y),\varphi_j(x-y)\rangle 1_{\Rn_+}(x)\phi(x)\,dx\big|\nonumber
\\
  &=| \langle 1_{\Rn_+}\otimes\tilde f(x,x-y),\phi\otimes\varphi_j)\rangle_{\Rn\times\Rn}|\label{eq:theFirstEntriesTested3paper}
\\
  &\le c p_M(\phi\otimes\varphi_j)
  \le c'p_M(\phi)p_M(\varphi_j).\nonumber
\end{align}
Here $p_M(\varphi_j)=p_M(2^{j|\vec a|}\varphi(2^{j\vec a}\cdot))=O(2^{j(|\vec a|+Ma^0)})$ grows at a fixed rate.
Therefore the choice $\phi=\check\psi_j*\eta$ 
shows via \eqref{eq:bigOForSeminorm} that the series has rapidly decaying terms, hence converges. 

To obtain continuity of $\ce_{\univ}$, it clearly suffices to show that $T\eta := \sum_{j=0}^\infty \check \varphi_j * \big(\ind {\R^n_+}(\check\psi_j * \eta)\big)$
defines a transformation $T: \cs(\Rn)\to \overset{\circ}{\cs}(\overline\R^n_+)$
satisfying
\begin{equation}\label{eq:theIdentityThatTheOperatorTsatisfies}
\langle\ce_{\univ}(f),\eta\rangle 
=\langle\widetilde f,T\eta\rangle\quad\text{for all } \eta\in\cs(\Rn).
\end{equation}
To this end we may let $\ind {\Rn_+}$ act first in~\eqref{eq:theFirstEntriesTested3paper}, which via~\eqref{eq:infiniteSumForMathcalEu} gives
\begin{equation}\label{eq:calculationsTheAdjointOfTheOperatorMathcalE}
\langle\ce_{\univ}(f),\eta\rangle 
= \sum_{j=0}^\infty \big\langle\widetilde f, \int \check\psi_j * \eta(x) \ind {\Rn_+}(x) \varphi_j(x-y)\, dx\big\rangle.
\end{equation}
The integral is in $\cs(\Rn)$ as a function of $y$ (cf.~the theory of tensor products), and since $\supp\varphi_j\subset\overline\R^n_-$ it is only non-zero for $y_n\ge x_n>0$. Hence the summands in $T\eta$ belong to $\overset{\circ}{\cs}(\overline\R^n_+)$, so $T$ has range in this subspace, if its series converges in $\cs(\Rn)$.
But by the completeness, this follows since any seminorm $p_M$ applied to $\int \check\psi_j * \eta(x) \ind {\Rn_+}(x) \varphi_j(x-y)\, dx$ is estimated by $c p_M(\check\varphi_j) p_{M+n+1}(\check\psi_j * \eta)$, which tends rapidly to 0 as above.

Finally,~\eqref{eq:calculationsTheAdjointOfTheOperatorMathcalE} now yields~\eqref{eq:theIdentityThatTheOperatorTsatisfies} by summation in the second entry.
\end{proof}

In the next convergence result, the familiar dyadic corona condition, cf.\ e.g.~\cite[Lem.~3.20]{JS08}, has been weakened to one involving convolution with a function $\psi$ satisfying a moment condition of infinite order. It appeared implicitly in~\cite{ryc99ext}.

\begin{lemma}\label{lemma:XinRychkov}
Let $(g^j)_{j\in\N_0}$ be a sequence of measurable functions on $\Rn$ such that
\begin{equation*}
\|\,(g^j)\,\| := \|\, 2^{js}G^j\, | L_{\vec p}(l_q)\| < \infty,
\end{equation*}
where for some $\vec r>0$, 
\begin{equation*}
  G^j(x) = \sup_{y\in\Rn} \frac{|g^j(y)|}{\prod_{l=1}^n (1+2^{ja_l}|x_l-y_l|)^{r_l}},\quad x\in\Rn.
\end{equation*}
When $\psi_0,\psi\in\cs(\Rn)$ with $L_\psi=\infty$, then $\sum_{j=0}^\infty \psi_j*g^j$ converges in $\cs'(\Rn)$ for any such $(g^j)_{j\in\N_0}$ and
\begin{equation}\label{eq:210inRychkov}
\Big\|\, \sum_{j=0}^\infty \psi_j*g^j\, \Big|F^{s,\vec a}_{\vec p,q}\Big\|\le c_{q,s} \|\, (g^j)\, \|
\end{equation}
with a constant $c_{q,s}$ independent of $(g^j)_{j\in\N_0}$.
\end{lemma}

\begin{proof}
By assumption $\|\, (g^j)\, \|<\infty$, hence $G^j(\widetilde x)<\infty$ for an $\widetilde x\in\Rn$, $j\in\N_0$, implying $|g^j(x)|\le G^j(\widetilde x)\prod_{l=1}^n (1+2^{ja_l}|\widetilde x_l-x_l|)^{r_l}$. Thereby, $g^j$ belongs to $L_{1,\loc}(\Rn)$ and grows at most polynomially, thus $g^j$ and therefore also $\psi_j*g^j$ are in $\cs'(\Rn)$.

Using $\Phi_l$ from~\eqref{unity}, the following estimate holds for $l\in\N_0$, $x\in\Rn$,
\begin{equation}\label{eq:estimateOfIjl}
|\cfi\Phi_l * \psi_j*g^j(x)| \le \int |\cfi\Phi_l*\psi_j(z)| |g^j(x-z)|\, dz \le I_{j,l}\cdot G^j(x),
\end{equation}
where
\begin{equation*}
I_{j,l} = \int |\cfi\Phi_l*\psi_j(z)|\prod_{l=1}^n (1+2^{ja_l}|z_l|)^{r_l}\, dz.
\end{equation*}
Since $L_\psi=\infty=L_{\cfi\Phi}$, 
a straightforward application of~\cite[Lem.~4.5]{HJS13a} yields the following estimate of the anisotropic dilations in $I_{j,l}$: for every $M>0$ there is some $C_M>0$ such that 
\begin{equation*}
  I_{j,l} \le C_M 2^{-|l-j|M} \quad\text{for all}\enskip j,l\in\N_0.
\end{equation*}
For $M=\varepsilon+|s|$, where $\varepsilon>0$ is arbitrary, we obtain from~\eqref{eq:estimateOfIjl},
\begin{equation}\label{eq:calculation210inRychkov}
  2^{ls}|\cfi\Phi_l*\psi_j*g^j(x)| \le c_s 2^{js} 2^{-|l-j|\varepsilon} G^j(x),
\end{equation}
which implies, using $|j-l|\ge j-l$, 
\begin{align*}
  \|\, \psi_j*g^j\, |F^{s-2\varepsilon,\vec a}_{\vec p,1}\|
  \le c_s \Big(\sum_{l=0}^\infty 2^{(-|j-l|-2l)\varepsilon}\Big) \|\, 2^{js}G^j\, | L_{\vec p}\| 
  \le c_{s} 2^{-j\varepsilon} \|\, (g^j)\, \|.
\end{align*}
This yields for $d:=\min(1,p_1,\ldots,p_n)$,
\begin{equation*}
\sum_{j=0}^\infty \|\, \psi_j*g^j\, |F^{s-2\varepsilon,\vec a}_{\vec p,1}\|^d \le c_{s}^d \|\, (g^k)\,\|^d \sum_{j=0}^\infty 2^{-j\varepsilon d}<\infty,
\end{equation*}
hence $\sum_{j=0}^\infty \psi_j*g^j$ converges in the quasi-Banach space $F^{s-2\varepsilon,\vec a}_{\vec p,1}$ and thus in $\cs'$.

Finally, by~\eqref{eq:calculation210inRychkov} and~\cite[Lem.~2.7]{HJS13a} applied to $(2^{js}G^j)_{j\in\N_0}$,
\begin{align*}
\Big\|\, \sum_{j=0}^\infty \psi_j*g^j\,\Big|F^{s,\vec a}_{\vec p,q}\Big\|
\le c_{q,s} \Big\|\,\Big(\sum_{j=0}^\infty 2^{-|l-j|\varepsilon} 2^{js}G^j\Big)_{l\in\N_0}\, \Big|L_{\vec p}(\ell_q)\Big\|
\le c_{q,s} \|\, 2^{js}G^j\, |L_{\vec p}(\ell_q)\|,
\end{align*}
which shows~\eqref{eq:210inRychkov}.
\end{proof}

We recall a variant $\varphi_j^+$ of the Peetre-Fefferman-Stein maximal operators induced by $(\varphi_j)_{j\in\N_0}$, where $\varphi_0,\varphi\in \cs(\Rn)$ are supported in $\Rn_-$; i.e.~for $f\in\overline\cs'(\R^n_+)$ and $\vec r>0$,
\begin{equation}\label{eq:maximalFunctionOnAllRn}
  \varphi_j^+ f(x) = \sup_{y\in\Rn_+} \frac{|\varphi_j*f(y)|}{\prod_{l=1}^n (1+2^{ja_l}|x_l-y_l|)^{r_l}},\quad x\in \Rn_+,\enskip j \in\N_0.
\end{equation}
Now we are ready to state the main theorem of this section:

\begin{theorem}\label{thm:rychkovExtension}
When $\varphi_0,\varphi,\psi_0,\psi\in\cs(\Rn)$ are functions as in Proposition~\ref{prop:variantCalderon}, then 
\begin{equation}\label{eq:expressionForExt}
\ce_{\univ}(f) := \sum_{j=0}^\infty \psi_j * e_+(\varphi_j * f)
\end{equation}
is a linear extension operator from $\overline\cs'(\R^n_+)$ to $\cs'(\Rn)$, i.e.~$r_+\ce_{\univ} f=f$ in $\Rn_+$ for every $f\in\overline\cs'(\R^n_+)$.  
Moreover, $\ce_{\univ}$ is bounded for all $s\in\R$, $0<\vec p<\infty$ and $0<q\le\infty$,
\begin{equation*}
  \ce_{\univ}: \overline F^{s,\vec a}_{\!\vec p,q}(\Rn_+)\to F^{s,\vec a}_{\vec p,q}(\Rn).  
\end{equation*}
\end{theorem}

\begin{proof}
First it is shown using~\eqref{eq:maximalFunctionOnAllRn} that for any $f\in \overline F^{s,\vec a}_{\vec p,q}(\Rn_+)$ and $\vec r >\min(q,p_1,\ldots,p_n)^{-1}$,
\begin{equation}\label{eq:47inRychkov}
\|\, 2^{js} \varphi_j^+f\, | L_{\vec p}(\ell_q)(\Rn)\|
\le c \|\, f\, |\overline F^{s,\vec a}_{\vec p,q}(\Rn_+)\|.
\end{equation}
Besides $\varphi_j^+ f$, we shall use the well-known maximal operator $\varphi_j^* f$, where the supremum in~\eqref{eq:maximalFunctionOnAllRn} is replaced by supremum over $\Rn$. Hence for every $g\in F^{s,\vec a}_{\vec p,q}(\Rn)$ such that $r_+g=f$, we get from~\eqref{eq:definitionConvolutionSubset} that
\begin{equation}\label{eq:estimateTwoMaxFunctions}
\varphi_j^+ f(x) = \sup_{y\in\Rn_+}\frac{|\varphi_j*g(y)|}{\prod_{l=1}^n (1+2^{ja_l}|x_l-y_l|)^{r_l}} \le \varphi_j^*g(x),\quad x\in\Rn_+.
\end{equation}
This yields~\eqref{eq:47inRychkov} when combined with the following, obtained from techniques behind~\cite[Thm.~5.1]{HJS13a}:
\begin{equation}\label{eq:step3inUniversal}
\inf_{r_+ g=f} \|\, 2^{js}\varphi_j^* g\,|L_{\vec p}(\ell_q)(\Rn)\|
\le c \inf_{r_+ g=f} \|\, g\, |F^{s,\vec a}_{\vec p,q}(\Rn)\|
= c \|\, f\,|\overline F^{s,\vec a}_{\vec p,q}(\Rn_+)\|.
\end{equation}

More precisely, since we only have $L_\varphi = \infty$ available, it is perhaps simplest to exploit that the Tauberian conditions are fulfilled by the functions $\cfi \Phi_0$, $\cfi \Phi$ appearing in the definition of $F^{s,\vec a}_{\vec p,q}$, cf.~\eqref{unity}. Therefore~\cite[Thm.~4.4]{HJS13a} yields that the quasi-norm on the left-hand side in~\eqref{eq:step3inUniversal} is estimated by $\|\, 2^{js}(\cfi\Phi_j)^* g\, |L_{\vec p}(\ell_q)\|$, which in turn is estimated by $\|\, g\, |F^{s,\vec a}_{\vec p,q}\|$ using~\cite[Thm.~4.8]{HJS13a}.

To apply Lemma~\ref{lemma:XinRychkov}, we estimate $\|\, (e_+(\varphi_j*f))\,\|$ using the extension of~\eqref{eq:maximalFunctionOnAllRn} to $\Rn$, that is
\begin{align*}
\widetilde\varphi_j^+ f(x) := \sup_{y\in\Rn_+} \frac{|\varphi_j*f(y)|}{\prod_{l=1}^n (1+2^{ja_l}|x_l-y_l|)^{r_l}},\quad x\in \Rn,\enskip j \in\N_0,
\end{align*}
with which it is immediate to see that
\begin{align*}
  \|\, (e_+(\varphi_j*f))\,\| = \|\, 2^{js} \widetilde\varphi_j^+f\,|L_{\vec p}(\ell_q)\|.
\end{align*}
A splitting of the integral on the right-hand side in one over $\Rn_+$, respectively one over $\Rn_-$ yields, using the obvious inequality $\widetilde\varphi_j^+ f(x',x_n)\le \varphi_j^+ f(x',-x_n)$ 
for $x\in\Rn_-$ and~\eqref{eq:47inRychkov}, cf.\ Lemma~\ref{lemma:XinRychkov},
\begin{equation*}
  \|\, \ce_{\univ} f\,|F^{s,\vec a}_{\vec p,q} \|
  \le c\|\, (e_+(\varphi_j*f))\,\| \le 2c\|\, 2^{js}\varphi_j^+f\,|L_{\vec p}(\ell_q)(\Rn_+) \|
  \le 2c \|\, f\,|\overline F^{s,\vec a}_{\!\vec p,q}(\Rn_+)\|.
\end{equation*}

Finally, continuity of $r_+:\cd'(\Rn)\to\cd'(\Rn_+)$ together with~\eqref{eq:welldefinedConvolution} and Proposition~\ref{prop:variantCalderon} give
\begin{equation*}
r_+(\ce_{\univ} f)
= \sum_{j=0}^\infty r_+ \big(\psi_j*e_+(\varphi_j*f)\big)
= \sum_{j=0}^\infty \psi_j*(\varphi_j*f) = f,
\end{equation*}
hence $\ce_{\univ} f$ is an extension of $f$.
\end{proof}

In the study of trace operators, it will be necessary to extend from more general domains.
Indeed, using the splitting $x=(x',x_n)$ on $\Rn$ and writing $f(x',C-x_n)$ as $f(\cdot,C-\cdot)$, 
the fact that $x\mapsto (x',C-x_n)$ is an involution 
easily gives a universal extension from the half-line $\,]-\infty,C[\,$:

\begin{corollary}\label{cor:rychkovExtension}
For any $C\in\R$, the operator 
\begin{equation*}
  \ce_{\univ,C} f(x) := \ce_{\univ}\big(f(\cdot,C-\cdot)\big)(x',C-x_n)
\end{equation*}
is a linear and bounded extension from $\overline F^{s,\vec a}_{\vec p,q}(\R^{n-1}\times]-\infty,C[)$ to $F^{s,\vec a}_{\vec p,q}(\Rn)$.
\end{corollary}

\begin{proof}
The quasi-norm on $F^{s,\vec a}_{\vec p,q}(\Rn)$ is invariant under translations $\tau_h u = u(\cdot-h)$, cf.~\cite[Prop.~3.3]{JS08}, and under the reflection $\mathcal R u = u(\cdot,-\cdot)$, when $\Phi_0, \Phi$ are invariant under $\mathcal R$,
as we may assume up to equivalence.
So, clearly $u(x',C-x_n)$ is in $F^{s,\vec a}_{\vec p,q}(\Rn)$ with the same quasi-norm as $u$.
 
By Definition~\ref{def:Fsubspace} this readily implies that the change of coordinates is also continuous from the space $\overline F^{s,\vec a}_{\vec p,q}(\R^{n-1}\times]-\infty,C[)$ to $\overline F^{s,\vec a}_{\vec p,q}(\R^{n-1}\times]0,\infty[)$. Thus
\begin{align*}
\|\, \ce_{\univ,C} f\, |F^{s,\vec a}_{\vec p,q}(\Rn)\| 
&\le c\|\, \ce_{\univ}(f(\cdot,C-\cdot))\,|F^{s,\vec a}_{\vec p,q}(\Rn)\|\\
&\le c\|\, f(\cdot,C-\cdot)\, |\overline F^{s,\vec a}_{\vec p,q}(\R^{n-1}\times]0,\infty[)\|
\le c\|\, f\,|\overline F^{s,\vec a}_{\vec p,q}(\R^{n-1}\times]-\infty,C[)\|,
\end{align*}
and the linearity of $\ce_{\univ,C}$ follows directly from the linearity of $\ce_{\univ}$.
\end{proof} 

In comparison with the well-known half-space extension by Seeley~\cite{see64} we note that the above construction is applicable for all $s\in\R$, even in the mixed-norm case. Also it has the advantage that several results from~\cite{HJS13a} can be utilised, making the argumentation less cumbersome.

\section{Trace Operators}\label{traces}
Under the assumption in~\eqref{eq:conditionsOn_ap_anisotropic}, we study the trace at the flat boundary of a cylinder $\Omega\times I$, where $\Omega\subset\Rn$ is $C^\infty$ and $I:=\,]0,T[\,$, possibly $T=\infty$. The trace at the curved boundary is studied only for $T<\infty$ and under the additional assumption that $\partial\Omega$ is compact.
The associated operators are 
\begin{alignat*}{2}
&r_0 : \enskip &f(x_1, \ldots , x_n,t)  &\mapsto   f(x_1, \ldots , x_n,0)  , \\
&\gamma :  \enskip  &f(x_1, \ldots , x_n,t) &\mapsto  f(x_1, \ldots , x_n,t){|_{\Gamma}} .
\end{alignat*}
As a preparation (for a discussion of compatibility conditions), the chapter ends with a discussion of traces on both the flat and the curved boundary at the corner $\partial\Omega\times\{0\}$ of the cylinder.

For the reader's sake, we recall some notation from~\cite{JS08}, namely that the trace 
at the hyperplane where $x_k=0$ is denoted by $\gamma_{0,k}$:
\begin{equation}\label{eq:definitionOfTraceAtHyperplanek}
\gamma_{0,k}:  \enskip  f(x_1, \ldots , x_n,t) \mapsto  f(x_1, \ldots,0,\ldots , x_n,t).
\end{equation}
It will be 
convenient for us to use $p':=(p_1,\ldots,p_{k-1})$, $p'':=(p_{k+1},\ldots,p_n,p_t)$, analogously for $\vec{a}$, and to set $r_l = \max(1,p_l)$. 
Furthermore, we recall that $x_{n+1}=t$, $a_{n+1}=a_t$, $p_{n+1}=p_t$, hence we shall work with $\vec a,\vec p$ of the form, cf.~\eqref{eq:conditionsOn_ap_anisotropic},
\begin{equation}\label{eq:conditionsOn_ap_anisotropic2}
  \vec a = (a_0,\ldots,a_0,a_t), \qquad \vec p = (p_0,\ldots,p_0,p_t)<\infty,
\end{equation}
where the finiteness of $\vec p$ is assumed in order to apply the results in~\cite{JS08}.

\subsection{The Trace at the Flat Boundary}\label{subsec:traceFlatBoundaryMathNr}
The trace $r_s$, defined by evaluation at $t=s$, is for each $s\in I$ well defined on the subspace,
\begin{equation}\label{eq:embeddingContDist}
C(I,\cd'(\Omega)) \subset \cd'(\Omega\times I),
\end{equation}
where the embedding can be seen by modifying the proof of~\cite[Prop.~3.5]{joh00}.
On the smaller subspace $C(\overline I,\cd'(\Omega))$ consisting of the elements having a continuous extension in $t$ to $\R$, even the trace $r_0$ is well defined (and it induces a similar operator also denoted $r_0$).
Indeed, for $u\in C(\overline I,\cd'(\Omega))$ all extensions $f$ are equal in $\Omega\times I$ and by continuity therefore also at $t=0$, hence 
\begin{equation}\label{eq:defOfr0WhenfDependsSmoothlyOnt}
  r_0 u := f(\cdot,0).
\end{equation}

Now, it was shown in~\cite[Thm.~2.4]{JS08} that
\begin{equation}
F^{s,\vec a}_{\vec p,q}(\R^{n+1})\hookrightarrow C_{\operatorname{b}}(\R,L_{r'}(\Rn)) \quad\text{when}\quad
  s>\frac{a_t}{p_t}+ n\Big(\frac{a_0}{\min(1,p_0)}-a_0\Big),
  \label{spqn-cnda}
\end{equation}
and this induces an embedding $\overline F^{s,\vec a}_{\!\vec p,q}(\Omega\times I)\hookrightarrow C(\overline I,L_{r'}(\Omega))$,
so the trace $r_0$ can be applied to $u$ in $\overline F^{s,\vec a}_{\vec p,q}(\Omega\times I)$, i.e.~for an arbitrary extension $f$ in $F^{s,\vec a}_{\vec p,q}(\R^{n+1})$,
\begin{equation}\label{eq:r0AppliedToFspace}
  r_0 u = r_\Omega f(\cdot,0).
\end{equation}

To define a right-inverse of $r_0$ when applied to $\overline F^{s,\vec a}_{\vec p,q}(\Omega\times I)$, we recall that a bounded right-inverse $K_{n+1}$ of the analogous trace $\gamma_{0,n+1}$ on Euclidean space, cf.~\cite[Thm.~2.6]{JS08},
\begin{equation}\label{eq:rightInverseToFlatTrace}
K_{n+1}: B^{s-\frac{a_t}{p_t},a'}_{p',p_t}(\Rn)\to F^{s,\vec a}_{\vec p,q}(\R^{n+1}), \quad s\in\R,
\end{equation}
is given by the following, where $\psi\in C^\infty(\R)$ such that $\psi(0)=1$ and $\supp \cf\psi\subset[1,2]$,
\begin{equation}\label{eq:defOfK1}
K_{n+1}v(x) := \sum_{j=0}^\infty \psi(2^{ja_{n+1}}x_{n+1}) \cf^{-1}(\Phi_j(\xi',0)\cf v(\xi'))(x').
\end{equation}
A right-inverse was given in this form by Triebel~\cite{tri83} in the isotropic case.

\begin{theorem}\label{main11}
When $\vec a, \vec p$ fulfil~\eqref{eq:conditionsOn_ap_anisotropic2} and $s$ satisfies the inequality in~\eqref{spqn-cnda},
then 
\begin{equation*}
  r_0: \overline F^{s,\vec a}_{\vec p,q}(\Omega\times I)\to \overline B^{s-\frac{a_t}{p_t},a'}_{p',p_t}(\Omega)
\end{equation*}
is a bounded \emph{surjection} and it has a right-inverse $K_0$.
More precisely, the operator $K_0$ can be chosen so that $K_0: \overline B^{s-\frac{a_t}{p_t},a'}_{p',p_t}(\Omega)\to \overline F^{s,\vec a}_{\!\vec p,q}(\Omega\times I)$ is bounded for all $s\in\R$.
\end{theorem}

\begin{proof}
The analogue of this theorem on Euclidean spaces, cf.~\cite[Thm.~2.5]{JS08}, yields for any  
$f\in F^{s,\vec a}_{\vec p,q}(\R^{n+1})$ the existence of a constant $c$ (only depending on 
$s,\vec p,q,\vec a$)  such that
\begin{equation*}
\big\|\, \gamma_{0,n+1} f\, \big| B^{s-\frac{a_t}{p_t},a'}_{p',p_t}(\Rn)\big\|
\leq c \|\, f\, | F^{s,\vec a}_{\vec p,q}(\R^{n+1})\|.
\end{equation*}
Choosing $f$ in~\eqref{eq:r0AppliedToFspace} so the right-hand side is bounded by $2c \|\, u\, |\overline F^{s,\vec a}_{\vec p,q}(\Omega\times I)\|$,  we obtain boundedness of~$r_0$, since $r_\Omega (\gamma_{0,n+1} f) = r_0 u$, cf.~\eqref{eq:definitionOfTraceAtHyperplanek}.

A right-inverse $K_0$ is constructed using $K_{n+1}$ in~\eqref{eq:rightInverseToFlatTrace} and Rychkov's extension operator in~\eqref{eq:rychkovExt}:
\begin{equation}\label{eq:rightInverseToFlatTraceMathNRNew}
  K_0 := r_{\Omega\times I} \circ K_{n+1}\circ \ce_{\univ,\Omega}: \overline B^{s-\frac{a_t}{p_t},a'}_{p',p_t}(\Omega)\to \overline F^{s,\vec a}_{\vec p,q}(\Omega\times I).
\end{equation}
(Since~\eqref{eq:rychkovExt} applies only to isotropic spaces over $\Omega\subset\Rn$, one can exploit~\eqref{eq:conditionsOn_ap_anisotropic2} to make rescalings $(s,a')\leftrightarrow s/a_0$, 
cf.\ Lemma~\ref{lem:lambdaLemmaBesov}.)

It is bounded for all $s\in\R$, because $K_{n+1}$ and $\mathcal{E}_{u,\Omega}$ are so.
Finally,~\eqref{eq:r0AppliedToFspace} yields for any $v\in \overline B^{s-\frac{a_t}{p_t},a'}_{p',p_t}(\Omega)$,
\begin{align*}
r_0 \circ K_0 v = r_\Omega ( K_{n+1}\circ \ce_{\univ,\Omega} v)(x_1,\ldots,x_n,0) = r_\Omega \circ \gamma_{0,n+1} \circ K_{n+1} \circ \ce_{\univ,\Omega} v = v,
\end{align*}
hence $K_0$ is a right-inverse of $r_0$.
\end{proof}

\subsection{A Support Preserving Right-Inverse}
Having treated the trace at $\{t=0\}$, we now construct an \emph{explicit} 
support preserving right-inverse of it. 
It is useful for parabolic problems e.g.\ in the reduction to homogeneous boundary conditions. 
At no extra cost, general $\vec a$ and $\vec p$ are treated in most of this section.

It is known from \cite{JS08} that whenever $s>\frac{a_t}{p_t}+\sum_{k\le n} \Big(\frac{a_k}{\min(1,p_1,\ldots,p_k)}-a_k\Big)$,
then $r_0$ is bounded,
\begin{equation*}
  r_0: F^{s,\vec a}_{\vec p,q}(\Rn\times\R)\to B^{s-\frac{a_t}{p_t},a'}_{p',p_t}(\Rn).
\end{equation*}

The particular right-inverse in~\eqref{eq:defOfK1} shall now be replaced by a finer construction
of a right-inverse $Q$ having the useful property that
\begin{equation}
  \supp u \subset \overline\R^n_+  \implies \supp Qu \subset \overline\R^n_+ \times \R.
\label{suppQu-eq}
\end{equation}
Roughly speaking the idea is to replace the use of Littlewood--Paley
decompositions by the kernels of localised means $(k_j)_{j\in\N_0}$; cf.~Theorem~\ref{thm:local}.  
That is, we tentatively take $Q$ of the form
\begin{equation}
  Qu(x,t)=\sum_{j=0}^\infty \eta(2^{ja_t}t)k_j*u(x).
\label{Qu-id}
\end{equation}
Hereby the auxiliary function $\eta\in\cs(\R)$ is again chosen with $\eta(0)=1$ and such that
$\supp\widehat\eta\subset [1,2]$.

The main reason for this choice of $Qu$ is that the property
\eqref{suppQu-eq} will eventually result when the kernels $k_j$ are so
chosen that 
\begin{equation}
  \supp u \subset \overline\R^n_+  \implies \supp k_j*u \subset \overline\R^n_+.
\label{suppkju-eq}
\end{equation}
By the support rule for convolutions, this follows if $\supp k_j\subset\overline\R^n_+$.
However, in order to choose the $k_j$, we shall first draw on the construction of 
${\cal E}_{\operatorname{u}}$ in Section~\ref{sec:rychkovExtension} and
take functions $\varphi_0,\varphi,\psi_0,\psi$ in $\cs(\Rn)$ with support in $\overline\R^n_+$ 
and satisfying 
\begin{equation}
  \int\varphi_0\, dx=1=\int\psi_0\, dx,\qquad L_\varphi=\infty=L_\psi,
\label{moment-cnd}
\end{equation}
in such a way that by setting e.g.\  $\psi_j(x)=2^{j|\vec a|}\psi(2^{j\vec a}x)$ for $j\ge1$,
one has Calder\'{o}n's reproducing formula
\begin{equation}
  u=\sum_{j=0}^\infty \psi_j*\varphi_j*u\quad\text{ for}\enskip
    u\in\cs'(\Rn).
\label{calderon-id}
\end{equation}
Existence of these functions may be obtained as in the proof of
Proposition~\ref{prop:variantCalderon}, simply by omitting the reflection in the definition of $\varphi_0$ and
proceeding with the argument for~\eqref{eq:convergenceOfSeriesForAllTempD} in the proof there.

Now we obtain $\supp k_j\subset\overline\R^n_+$ by choosing 
\begin{equation*}
  k_0=\psi_0*\varphi_0,\quad k =\psi*\varphi.
\end{equation*}
Then \eqref{calderon-id} states that $u=\sum_{j\ge0}k_j*u$, which
together with the condition $\eta(0)=1$ will imply that $Q$ is a
right-inverse of $r_0$. 

Since the supports of the $k_j$ are only confined to be in the
half-space $\overline\R^n_+$, we refer to the $k_j$ as kernels of
\emph{localised} means. (Triebel termed them local
in case the supports are compact.)

In addition we need to recall an $\cs'$-version of \cite[Prop.~3.5]{joh00}.

\begin{lemma}\label{cs-lem}
  There is an (algebraic) embedding $C_{\operatorname{b}}(\R,\cs'(\Rn))\subset\cs'(\Rn\times\R)$
  given by
  \begin{equation*}
    \langle \Lambda_f,\psi\rangle=\int_\R\langle f(t),\psi(\cdot,t)\rangle_{\Rn}\,dt
  \end{equation*}
  for each continuous, bounded map $f:\R\to \cs'(\Rn)$ and $\psi\in\cs(\Rn\times\R)$.
\end{lemma}

\begin{proof}
  By the boundedness, the family $\{f(t)\}_{t\in\R}$ is equicontinuous,
  so for some $M>0$ we have $|\langle f(t),\phi\rangle|\le cp_M(\phi)$
  for all $t\in\R$ and $\phi\in\cs(\Rn)$. Hence the integrand is continuous and 
  estimated crudely by $cp_{M+2}(\psi)/(1+t^2)$, so
  $\Lambda_f$ makes sense and 
  $|\langle \Lambda_f,\psi\rangle|\le c\pi p_{M+2}(\psi)$.
\end{proof}

Using this lemma, we can now improve on \eqref{Qu-id} by giving $Qu$ a
more precise meaning as an element of $C_{\operatorname{b}}(\R_t,\cs'(\Rn_x))$.
Namely, $Qu(\cdot,t)$ is the distribution given on $\phi\in\cs(\Rn)$ by 
\begin{equation}\label{Qu'-id}
  \langle Qu(\cdot,t),\phi\rangle  = \sum_{j=0}^\infty \eta(2^{ja_t}t)\langle k_j*u,\phi\rangle.
\end{equation}
This will be clear from the proof of

\begin{proposition}\label{prop:suppQu'-eq}
  The operator $Q$ is a well-defined w$^*$-continuous linear map
  $\cs'(\Rn)\to\cs'(\Rn\times\R)$ having range in
  $C_{\operatorname{b}}(\R_t,\cs'(\Rn_x))$. It is a right-inverse of $r_0$ preserving
  supports in $\overline\R^n_+$ in the strong form
  \begin{equation}
    \supp u\subset\overline\R^n_+ \implies \forall t: \supp
    Qu(\cdot,t)\subset\overline\R^n_+.
    \label{suppQu'-eq}
  \end{equation}
In particular, $Q: \overset{\circ}{\cs}{}'(\overline\R^n_+) \to \overset{\circ}{\cs'}(\overline\R^n_+\times\R)$,
cf.\ Definition~\ref{def:DefinitionOfTempDOnSubsets}.
\end{proposition}

\begin{remark}
We can of course add that
\eqref{suppQu'-eq}$\implies$\eqref{suppQu-eq}, for we may apply
Lemma~\ref{cs-lem} to $f=Qu$ and consider the $\psi(x,t)$ that vanish
for $x_n\ge0$: when \eqref{suppQu'-eq} holds, the integrand is
identically $0$.  (Unlike \eqref{suppQu'-eq}, property \eqref{suppQu-eq} is
  meaningful also without continuity of $Qu$ with respect to $t$.)
\end{remark}

\begin{proof}
  It is first noted that $\sum\langle k_j*u,\phi\rangle$ 
  converges absolutely for each test function $\phi\in\cs'(\Rn)$. In fact,
  using the notation $\check k_j(x) = k_j(-x)$, the estimate
  $|\langle u,\check k_j*\phi\rangle|\le c p_M(k_j*\phi)\le c 2^{-jN}$
  holds for any $N>0$; this follows from the infinitely many vanishing
  moments, i.e. $L_k=\infty$, cf.~\cite[Lem.~4.2]{HJS13a}.

  Hence $\sum\langle k_j*u,\phi\rangle\eta(2^{ja_t}t)$ is a 
  Cauchy series for each $\phi\in\cs(\Rn)$
  as $\eta(2^{ja_t}t)$ is a bounded sequence for fixed $t$.
  Since it converges, $Qu$ is defined in $\cs'(\Rn)$ for each $t$.

  The convergence is absolute and uniform in $t$, so
  $t\mapsto \langle Qu(t),\phi\rangle$ is continuous; and
  bounded by $c\sum|\langle k_j*u,\phi\rangle|$. Therefore 
  $Qu$ is in the subspace $C_{\operatorname{b}}(\R_t,\cs'(\Rn_x))$; cf.~Lemma~\ref{cs-lem}.
  
  Consequently $r_0Qu$ is defined by evaluation at $t=0$, which gives
  $\sum\eta(0)k_j*u(x)$, hence gives back $u$ because of
  \eqref{calderon-id}. Using the convergence in 
  $\cs'(\Rn)$, the support preservation in \eqref{suppQu'-eq} is
  immediate from \eqref{suppkju-eq} by test against any $\phi\in
  C_0^\infty(\Rn)$ vanishing for $x_n\ge 0$.

  Finally, continuity of $Q$ follows at once if $\langle
  Qu,\psi\rangle= \langle u,T\psi\rangle$ for $\psi\in\cs(\R^{n+1})$,
  i.e.\ if $Q$ is the transpose of $T:\cs(\R^{n+1})\to\cs(\Rn)$ given by
  \begin{equation*}
    (T\psi)(x)= \int_\R \sum_{j=0}^\infty \check k_j* \psi(x,t)\eta(2^{ja_t}t)\,dt.
  \end{equation*}
  Here the sum is a Cauchy series 
  in the space $\cs(\R^{n+1})$, for a
  seminorm $p_M$ applied to the general term is less than
  $p_M(\eta(2^{ja_t}t))= O(2^{ja_tM})$ times $p_M(\check k_j*
  \psi)$, which decays rapidly as $L_k=\infty$. Denoting the sum by
  $S(x,t)$, also $x\mapsto\int S(x,t)\,dt$ is a Schwartz function, so
  $T\psi$ is well defined and by the definition of tensor products we get 
  \begin{equation*}
   \langle u,T\psi\rangle=\langle u\otimes 1,
  S\rangle=\int\langle u,S(\cdot,t)\rangle\,dt
  =\int\langle Qu(\cdot,t),\psi(\cdot,t)\rangle\,dt=\langle
  Qu,\psi\rangle,   
  \end{equation*}
  using~\eqref{Qu'-id} and Lemma~\ref{cs-lem}.
\end{proof}

Before we go deeper into the boundedness
of $Q$ in the scales of mixed-norm Lizorkin--Triebel spaces, we first
sum up the fundamental estimate in the next result.
In the isotropic case it goes back at least to the trace
investigations of Triebel~\cite[p.~136]{tri83}.

\begin{proposition}\label{Q-prop}
 For $\vec p=(p_1,\ldots,p_n,r)$ in 
 $\,]0,\infty[\,^{n+1}$, a real number $a>0$ and $0<q\le\infty$
 there is a constant $c$ with the property that 
 \begin{equation*}
  \big\|\,\big\{v_j\otimes 2^{j\frac{a}{r}}f(2^{ja}\cdot)\big\}_{j=0}^\infty\,\big| L_{\vec
    p}(\ell_q)(\R^{n+1})\big\|
\le c
  \Big(\sum_{j=0}^\infty \|\,v_j\,| L_{p'}(\Rn)\|^{r}\Big)^{1/r}   
 \end{equation*}
 whenever $(v_j)$ is a sequence of measurable functions on $\Rn$ and
 $f\in C(\R)$ is such that $t^Nf(t)$ is bounded for some $N>0$ satisfying $Nr>1$.
\end{proposition}

\begin{proof}
  To save a page of repetition from \cite[Sec.~4.2.3]{JS08}, we leave it
  to the reader to carry over the proof given there with a few
  notational changes. (Note that $f$ itself is bounded, so the
  arguments there extend to our case without any Schwartz class
  assumptions on $f$.)
\end{proof}

\begin{theorem}\label{thm:mappingPropertiesOfQBetweenBFspaces}
  The operator $Q$ is for $0<\vec p<\infty$, $0<q\le\infty$ and $\vec a\ge 1$
  a bounded map 
  \begin{align}
    Q&: B^{s,a'}_{p',p_t}(\Rn)\to 
    F^{s+\frac{a_t}{p_t},\vec a}_{\vec p,q}(\Rn\times\R)
   \quad\text{for all } s\in\R,
\label{Qspq}
\\
  Q&: \overset{\circ}{B}{}^{s,a'}_{\!p',p_t}(\overline\R^n_+)\to 
      \overset{\circ}{F}{}^{s+\frac{a_t}{p_t},\vec a}_{\!\vec p,q}(\overline\R^n_+\times\R)\quad\text{for all
      }s\in\R.
\label{Qspq'}
  \end{align}
\end{theorem}

\begin{proof}
The property in \eqref{Qspq'} is a direct consequence of \eqref{Qspq} and 
Proposition~\ref{prop:suppQu'-eq}. 

As for \eqref{Qspq}, by means of an auxiliary function ${\cal F}\tilde\eta\in
  C^\infty_0(\R)$ taken so that $ {\cal F}\tilde\eta=1$ on
  $[1,2]\supset\supp\widehat\eta$ and $\supp{\cal
    F}\tilde\eta\subset\,]0,\infty[\,$, we rewrite $Qu$ in terms of
  convolutions on $\R^{n+1}$, using that $k_j=\psi_j*\varphi_j$,
  \begin{equation*}
    Qu=\sum_{j=0}^\infty \tilde\eta_j*\eta(2^{ja_t}\cdot)(t)k_j*u(x)
      =\sum_{j=0}^\infty (\psi\otimes\tilde\eta)_j*(\varphi_j*u\otimes\eta(2^{ja_t}\cdot)).
  \end{equation*}
Hereby it is understood for $j=0$ that the first factor is 
$\psi_0\otimes\tilde\eta$.

Now we may invoke Lemma~\ref{lemma:XinRychkov} as the function
$\psi\otimes\tilde\eta$ has all its moments equal to $0$, because its Fourier
transformed function is supported in a half-plane disjoint from the
origin in $\R^{n+1}$. This gives an estimate of the Lizorkin--Triebel
norm as follows,
\begin{equation*}
  \big\|\, Qu\,\big| F^{s+\frac{a_t}{p_t},\vec a}_{\vec p,q}(\R^{n+1})\big\|\le c
  \big\|\,\big\{2^{(s+\frac{a_t}{p_t})j}(\varphi_j*u\otimes \eta(2^{ja_t}\cdot))_j^*\big\}_{j=0}^\infty\,\big|
  L_{\vec p}(\ell_q)(\R^{n+1})\big\| .
\end{equation*}
Here the maximal function $(\cdot)_j^*$ considered in the lemma allow us to estimate
the $j^{\text{th}}$ term by
\begin{equation*}
  \sup_{y,y_t}\big|2^{sj}\varphi_j*u(y)2^{j\frac{a_t}{p_t}}\eta(2^{ja_t}y_t)\big|
  \prod_{l=1}^{n+1}(1+2^{ja_l}|x_l-y_l|)^{-r_l}
  \le v_j(x)2^{j\frac{a_t}{p_t}}f(2^{ja_t}t)
\end{equation*}
if we set
\begin{align*}
  v_j&=\sup_{y}|2^{sj}\varphi_j*u(y)|\prod_{l=1}^{n}(1+2^{ja_l}|x_l-y_l|)^{-r_l},
\\
  f(t)&=\sup_{y_t}{|\eta(y_t)|}{(1+|t-y_t|)^{-r_{t}}}.
\end{align*}
To invoke Proposition~\ref{Q-prop}, we note
that $v_j$, $f$ are continuous (by an argument similar to e.g.~
\cite[(6)--(7)]{joh11})
and, moreover, $\sup |t^Nf(t)|<\infty$ for $0<N\le r_t$.
We therefore apply the proposition for
$r=p_t$, $a=a_t$ and note that if we fix the above parameter $r_t$ such
that $r_tp_t>1$, then $Np_t>1$ is fulfilled at least for $N=r_t$. This gives
\begin{equation*}
  \big\|\, Qu\,\big| F^{s+\frac{a_t}{p_t},\vec a}_{\vec p,q}(\R^{n+1})\big\|
  \le c \big\|\,\big\{v_j\otimes 2^{j\frac{a_t}{p_t}}f(2^{ja_t}\cdot)\big\}_{j=0}^\infty\,\big|
  L_{\vec p}(\ell_q)(\R^{n+1})\big\|
  \le c
 \Big(\sum_{j=0}^\infty \|\,v_j\,| L_{p'}(\Rn)\|^{p_t}\Big)^{1/p_t}.   
\end{equation*}
So by writing the $v_j$ in terms of the Peetre--Fefferman--Stein
maximal function $\varphi_j^*u(x)$,
\begin{equation*}
  \big\|\, Qu\,\big| F^{s+\frac{a_t}{p_t},\vec a}_{\vec p,q}(\R^{n+1})\big\|
  \le c
 \Big(\sum_{j=0}^\infty \|2^{sj}\varphi_j^*u \,|
 L_{p'}(\Rn)\|^{p_t}\Big)^{1/p_t}
  \le c \|\,u\,|B^{s,a'}_{p',p_t}(\Rn)\|.   
\end{equation*}
The last inequality is essentially known from \cite[(4)]{ryc99}, but
to account for effects of the flaws pointed out in
\cite[Rem.~1.1]{HJS13a}, let us
briefly note the following: if we apply
\cite[(21)]{ryc99} to the very last formula in the proof of~\cite[Thm.~4.4]{HJS13a}, 
then we get an estimate of the above sum by
$\|\,2^{sj}({\mathcal{F}^{-1}} \Phi)_j^*u\,|\ell_{p_t}(L_{p'})\|$. 
This can be controlled by the $\ell_{p_t}(L_{p_0})$-norm of the convolutions 
$2^{sj}{\mathcal{F}^{-1}}\Phi_j*u$,
i.e.\ by the stated $\|\,u\,|B^{s,a'}_{p',p_t}(\Rn)\|$,
by following the argument for \cite[(23)]{ryc99} 
(after the remedy discussed in Section~\ref{sec:isotropicBesovManifolds} above), 
say for simplicity
with $r_0:=r_1=\ldots=r_n$ and $r_0p_0>n$.
\end{proof}

The operator $Q$ above is now used to replace the particular right-inverse to $r_0$
in~\eqref{eq:rightInverseToFlatTraceMathNRNew}  by an operator $Q_\Omega$ that preserves support in
$\overline\Omega$. 

The construction uses the partition of unity $1=\sum_{\lambda}\psi_\lambda + \psi$ on
$\overline\Omega$ constructed in Section~\ref{ssect:CompManf} as well as cut-off functions
$\eta_\lambda\in C_0^\infty(\Rn)$, $\lambda\in\Lambda$, chosen with $\supp\eta_\lambda\subset
B$ and $\eta_\lambda =1$ on $\supp\widetilde\psi_\lambda$;
and some $\eta_\Omega\in C^\infty_{L_\infty}(\Rn)$, cf.\ Lemma~\ref{multTraces} for the definition
of $C^\infty_{L_\infty}$, such that $\supp\eta_\Omega\subset \Omega$ with $\eta_\Omega=1$ on $\supp\psi$.

For convenience we let $u_{(\lambda)} = e_B \big((\psi_\lambda
u)\circ\lambda^{-1}\big)$, which enters the definition of the operator $Q_\Omega$:
\begin{equation}\label{eq:definitionOfQOmegaOverCylinder}
Q_\Omega u = 
    \sum_\lambda e_{U_\lambda\times\R}\big( (\eta_\lambda Q u_{(\lambda)}) 
    \circ (\lambda\times\id_\R)\big) + \eta_\Omega Q(\psi u).
\end{equation}

\begin{theorem}\label{thm:constructionOfRightInverseOnCylinderPerservingSupport}
Let $\Omega\subset\Rn$ be a smooth set with compact boundary $\partial \Omega$ as in Section~\ref{ssect:CompManf}.
When $\vec a, \vec p$ satisfy~\eqref{eq:conditionsOn_ap_anisotropic2}, $0<q\le\infty$ and $s\in\R$, 
then $Q_\Omega$ is a bounded linear map
\begin{equation*}
Q_\Omega: B^{s,a'}_{p',p_t}(\Rn)\to F^{s+a_t/p_t,\vec a}_{\vec p,q}(\R^{n+1}),
\end{equation*}
which on the subspace $\overset{\circ}{B}{}^{s,a'}_{\!p',p_t}(\overline\Omega)$ has $r_0$ as a left-inverse:
\begin{equation}
  \label{eq:rQI}
  r_0 Q_\Omega u=u \quad\text{whenever $\supp u\subset\overline\Omega$ for $u\in B^{s,a'}_{p',p_t}(\Rn)$}.
\end{equation}
Moreover, $Q_\Omega$ has range in $C(\R_t,\cd'(\Rn_x))$ and preserves supports in $\overline\Omega$ in the strong form
\begin{equation}\label{eq:supportImplicationForQOmegau}
\supp u \subset\overline\Omega \implies \forall t: \supp Q_\Omega u(\cdot,t)\subset \overline\Omega.
\end{equation}
\end{theorem}

\begin{proof}
For the terms in the sum over~$\lambda$ in~\eqref{eq:definitionOfQOmegaOverCylinder}, we note that
the multiplication result in~\cite[4.2.2]{tri92} together with the Besov version of
Theorem~\ref{thm:isotropicCompactSupport}, cf.\ Section~\ref{sec:isotropicBesovManifolds}, imply 
\begin{equation}\label{eq:defOfuLambdaInQOmega}
  u_{(\lambda)} = e_B \big((\psi_\lambda u)\circ\lambda^{-1}\big)\in B^{s,a'}_{p',p_t}(\Rn).
\end{equation}
(These results apply to isotropic Besov spaces, so we use Lemma~\ref{lem:lambdaLemmaBesov} to
rescale $(s, a')\leftrightarrow s/a_0$, cf.~\eqref{eq:conditionsOn_ap_anisotropic2}.) 

Theorem~\ref{thm:mappingPropertiesOfQBetweenBFspaces} and the paramultiplication
result~\cite[Lem.~7]{HJS13b} now gives that $\eta_\lambda Qu_{(\lambda)}$ belongs to $F^{s+a_t/p_t,\vec
  a}_{\vec p,q}(\R^{n+1})$, hence according to Theorem~\ref{thm:infiniteCylinder},   
\begin{equation*}
(\eta_\lambda Qu_{(\lambda)}) \circ( \lambda\times\id_\R) \in\overline F^{s+a_t/p_t,\vec a}_{\!\vec p,q}(U_\lambda\times\R).
\end{equation*}
As $\supp \eta_\lambda\subset B$, Lemma~\ref{lem:extensionBy0} gives that extension of this
composition by 0 belongs to $F^{s+a_t/p_t,\vec a}_{\vec p,q}(\R^{n+1})$. 

For the last term in~\eqref{eq:definitionOfQOmegaOverCylinder}, it is an immediate consequence
of~\cite[4.2.2]{tri92} that $\psi u$ belongs to $B^{s,a'}_{p',p_t}(\Rn)$, since $\psi\in
C^\infty_{L_\infty}$ (as $\psi = 1-\sum_\lambda\psi_\lambda$ on $\overline \Omega$ and
$\partial\Omega$ is compact). 

This shows that $Q_\Omega u\in F^{s+a_t/p_t,\vec a}_{\vec p,q}(\R^{n+1})$ and by applying the
quasi-norm estimates in the theorems and lemmas referred to above, we obtain 
\begin{equation*}
\|\, Q_\Omega u\, | F^{s+a_t/p_t,\vec a}_{\vec p,q}(\R^{n+1})\| \le c \|\, u\, | B^{s,a'}_{p',p_t}(\Rn)\|.
\end{equation*}

Furthermore, it follows from Proposition~\ref{prop:suppQu'-eq} that $Q_\Omega u\in
C(\R_t,\cd'(\Rn_x))$ and therefore the effect of $r_0$ on $Q_\Omega u$ is simply restriction to
$t=0$, cf.~\eqref{eq:defOfr0WhenfDependsSmoothlyOnt}. Hence for $u\in B^{s,a'}_{p',p_t}(\Rn)$, 
\begin{equation*}
r_0 Q_\Omega u 
  = \sum_\lambda e_{U_\lambda} \big( (\eta_\lambda Q u_{(\lambda)})(\cdot,0) \circ \lambda\big)  
  + \eta_\Omega Q(\psi u)(\cdot,0).
\end{equation*}
Since $Q$ according to Proposition~\ref{prop:suppQu'-eq} is a right-inverse of $r_0$, this sum
equals the following by using~\eqref{eq:defOfuLambdaInQOmega} as well as the properties of
$\eta_\lambda,\eta_\Omega$, and in the final step that $\supp u\subset\overline\Omega$, 
\begin{equation*}
\sum_\lambda e_{U_\lambda} \big( (\eta_\lambda u_{(\lambda)}) \circ \lambda \big) + \psi u
= \sum_\lambda e_{U_\lambda} \big( \eta_\lambda\circ\lambda \cdot \psi_\lambda u \big) + \psi u
= \sum_\lambda \psi_\lambda u + \psi u
= u.
\end{equation*}
This shows \eqref{eq:rQI}.
Finally, the support preserving property in~\eqref{eq:supportImplicationForQOmegau} follows
from~\eqref{suppQu'-eq}. Indeed, when $\supp u \subset\overline\Omega$, then the support of each
$u_{(\lambda)}$ is contained in $\overline\R^n_+$ and therefore $\supp (\eta_\lambda Q u_{(\lambda)}
)\circ (\lambda(\cdot),t) \subset \overline\Omega$ for all $t\in\R$, which immediately gives that
$\supp Q_\Omega u(\cdot,t)\subset \overline\Omega$. 
\end{proof}

\subsection{The Trace at the Curved Boundary}\label{subsec:traceCurvedBoundaryMathNr}
We now address the trace $\gamma$ of distributions in $\overline F^{s,\vec a}_{\!\vec p,q}(\Omega\times I)$, where for simplicity $I=\,]0,T[\,$, $T<\infty$, and~$\Omega$ is smooth as in Definition~\ref{defn:smoothDomain} with compact boundary $\Gamma$.

\subsubsection{Preliminaries}
The trace is first worked out locally
and then it is observed that the local pieces define a global trace. 
In this process we use that the trace $\gamma_{0,1}$ is a bounded surjection, cf.~\cite[Thm.~2.2]{JS08},
\begin{equation}\label{eq:thm22}  
\begin{split}
  \gamma_{0,1}: &F^{s,\vec a}_{\vec p,q}(\R^{n+1}) \to F^{s-\frac{a_0}{p_0},a''}_{p'',p_0}(\Rn)\\
  &\qquad\quad\text{for } \enskip
s > \frac{a_0}{p_0} + 
(n-1)\Big(\frac{a_0}{\min(1,p_0,q)}-a_0\Big) +\Big(\frac{a_t}{\min(1,p_0,p_t,q)}-a_t\Big).
\end{split}\end{equation}
This is also valid for $\gamma_{0,n}$ in view of~\eqref{eq:conditionsOn_ap_anisotropic2} and we prefer to work with this, for locally the boundary $\Gamma$ is defined by the equation $x_n=0$, as usual.
For the $s$ in~\eqref{eq:thm22}, we have by \cite[Thm.~2.1]{JS08}, since $r_k:=\max(1,p_k)$, 
\begin{equation}
\label{eq:CbEmbeddedInL1loc}
F^{s,\vec a}_{\vec p,q}(\R^{n+1})\hookrightarrow C_{\operatorname{b}}(\R,L_{r''}(\Rn))\hookrightarrow  L_{1,\loc}(\R^{n+1}).
\end{equation}
So when $u\in \overline F^{s,\vec a}_{\vec p,q}(\Omega\times I)$ for such $s$, an extension $f$ in the corresponding space on $\Rn$ is a function and for this we right away get 
\begin{equation}\label{eq:fComposedInL1loc}
f\circ (\lambda^{-1}\times\id_\R)\in L_{1,\loc}(B\times\R).
\end{equation}

Moreover, if we work locally with cut-off functions $\psi\in C_0^\infty(U_\lambda)$, $\varphi\in C_0^\infty(\R)$, then Lemma~\ref{multTraces} yields $\psi\otimes\varphi f \in F^{s,\vec a}_{\vec p,q}(\R^{n+1})$. Changing coordinates, Theorem~\ref{thm:infiniteCylinder} implies that $(\psi\otimes\varphi f)\circ (\lambda^{-1}\times\id_\R)$ is in $\overline F^{s,\vec a}_{\vec p,q}(B\times\R)$, hence it extends by 0 to $F^{s,\vec a}_{\vec p,q}(\R^{n+1})$. 
By~\eqref{eq:thm22}, 
\begin{equation*}
\gamma_{0,n}\big((\psi\otimes\varphi f) \circ (\lambda^{-1}\times\id_\R)\big) 
\in F^{s-a_0/p_0,a''}_{p'',p_0}(\Rn).
\end{equation*}
Strictly speaking, we should have inserted the extension by $0$, namely $e_{B\times \R}$, before applying~$\gamma_{0,n}$, but we have chosen not to burden notation with this. Now restriction to $B'\times\R$ gives an element in $\overline F^{s-a_0/p_0,a''}_{p'',p_0}(B'\times \R)$, and since it is easily seen using~\eqref{eq:CbEmbeddedInL1loc} that restriction to $\{x_n=0\}$ and $e_{B\times\R}$ can be interchanged, we obtain
\begin{equation}\label{eq:interchangedRandE}
  (\psi\otimes\varphi f)\circ(\lambda^{-1}(\cdot,0)\times\id_\R) \in \overline F^{s-a_0/p_0,a''}_{p'',p_0}(B'\times \R).
\end{equation}

Furthermore, to describe the range of $\gamma$,  we introduce for an open interval $I'\supset I$ the restriction (with notation as in Section~\ref{cylinders})
\begin{equation*}
r_I: F^{s,\vec a}_{\vec p,q;\loc}(\Gamma\times I')\to F^{s,\vec a}_{\vec p,q;\loc}(\Gamma\times I),
\end{equation*}
which for any $v\in F^{s,\vec a}_{\vec p,q;\loc}(\Gamma\times I')$ is defined as the distribution arising from the family $\{r_{B'\times I} v_{\kappa\times\id_{I'}}\}_{\kappa\in\cf_0}$ of distributions on $B'\times I$, cf.~the paragraph on restriction just below Lemma~\ref{lem:hor634}. 

Using $r_I$, we also introduce a space of restricted distributions (in the time variable only),
\begin{equation}\label{eq:introFspaceOverline}
\overline F^{s,\vec a}_{\vec p,q}(\Gamma\times I) 
:= r_I F^{s,\vec a}_{\vec p,q;\loc}(\Gamma\times \R)
= r_I \overset{\circ}{F}{}^{s,\vec a}_{\!\vec p,q}(\Gamma\times J)
\end{equation}
valid for any compact interval~$J \supset I$.  
Since $\overset{\circ}{F}{}^{s,\vec a}_{\!\vec p,q}(\Gamma\times J)$ is a quasi-Banach space, cf.~Theorem~\ref{thm:LTcurvedBoundary}, the space $\overline F^{s,\vec a}_{\vec p,q}(\Gamma\times I)$ is so too when equipped with
\begin{equation}\label{eq:normOnRestrictionCompactnessSpace}
\|\, u\, | \overline F^{s,\vec a}_{\vec p,q}(\Gamma\times I) \| 
:= \inf_{r_I v = u} 
\|\, v\, | \overset{\circ}{F}{}^{s,\vec a}_{\!\vec p,q}(\Gamma\times J)\|.
\end{equation}

\subsubsection{The Definition}\label{subsubsec:theDef}
To give sense to $\gamma u$ in $\cd'(\Gamma\times I)$, it is first observed that~\eqref{eq:fComposedInL1loc} induces invariantly defined functions. Indeed, in view of the identity $\kappa^{-1}(\cdot) = \lambda^{-1}(\cdot,0)$, we set
\begin{equation*}
f_\kappa = f\circ(\lambda^{-1}(\cdot,0)\times\id_\R)\in L_{1,\loc}(B'\times\R)
\end{equation*}
and as distributions they transform as in~\eqref{eq:distFamily}, since 
\begin{equation}\label{eq:transformationOff}
f_\kappa\circ(\kappa\circ\kappa_1^{-1}\times\id_\R) = f_{\kappa_1}\quad \text{on}\quad \kappa_1(\Gamma_{\kappa_1}\cap\Gamma_\kappa)\times\R.
\end{equation}
Hence by Lemma~\ref{lem:hor634} there exists a unique $v\in \cd'(\Gamma\times\R)$ with
\begin{equation}\label{eq:connectionBtwnvAndf}
v_{\kappa\times\id_\R} = f_\kappa.
\end{equation}
That $v$ is in $F^{s-a_0/p_0,a''}_{p'',p_0;\loc}(\Gamma\times\R)$ is a special case of~\eqref{eq:interchangedRandE}, cf.\ Definition~\ref{def:LTspaceCurvedBoundary}.

Note that the distribution $v$ does not depend on the atlas $\cf_0$, for when 
another atlas $\cf_1$ in the same way induces a distribution $v_1$, then formula~\eqref{eq:transformationOff} read with 
$\kappa$ running through $\cf_0$ and $\kappa_1$ running through $\cf_1$ implies that both $v$ and $v_1$ result by ``restriction" from the distribution $w$ induced by $\cf_0\cup \cf_1$.

Now we define the trace $\gamma u$ in $\cd'(\Gamma\times I)$ by
\begin{equation}\label{eq:defOfGamma}
\gamma u = r_I v.
\end{equation}
Indeed, to verify that $\gamma u$ is independent of the chosen~$f$, it suffices to derive that for any two extensions $f_1,f_2\in F^{s,\vec a}_{\vec p,q}(\R^{n+1})$, the following identity holds for each $\lambda\in\Lambda$ and $(x',t)\in B'\times I$:
\begin{equation}\label{eq:independentOfExtension}
f_1\circ(\lambda^{-1}(\cdot,0)\times\id_\R)(x',t) = f_2\circ(\lambda^{-1}(\cdot,0)\times\id_\R)(x',t).
\end{equation}
To do so, we choose $\psi\in C_0^\infty(U_\lambda)$, $\varphi\in C_0^\infty(\R)$ such that $\psi(\lambda^{-1}(x',0))\neq 0$ and $\varphi(t)\neq 0$. Since $f_1,f_2$ coincide in $\Omega\times I$, the functions
\begin{equation*}
e_{B\times\R}\big((\psi\otimes\varphi f_j)\circ (\lambda^{-1}\times\id_\R)\big)(x,t), \quad j=1,2
\end{equation*}
are identical for $(x,t)\in B\times I$ with $x_n>0$.
Letting $x_n\to 0^+$ therefore gives the same limits in
$L_{r''}(\R^{n-1}\times I)$, 
cf.~\eqref{eq:CbEmbeddedInL1loc}, in particular they coincide in $L_{r''}(B'\times I)$.
As $(\psi\otimes\varphi)\circ(\lambda^{-1}(\cdot,0)\times\id_\R)(x',t)\neq 0$, this yields~\eqref{eq:independentOfExtension}.

Furthermore,~\eqref{eq:independentOfExtension} can be used to show that $\gamma$ does not depend on the Lizorkin--Triebel space satisfying~\eqref{eq:thm22}. For when $u$ belongs to two different Lizorkin--Triebel spaces, we can take $f_1$ above to be an extension in one of the spaces and $f_2$ to be an extension in the other. The identity in~\eqref{eq:independentOfExtension} then gives that $\gamma u$ belongs to the intersection of the corresponding Lizorkin--Triebel spaces over the curved boundary.

We also note that the trace $\gamma$ has the natural property that $r_I \circ \gamma = \gamma \circ r_I$ on $\overline F^{s,\vec a}_{\vec p,q}(\Omega\times I')$ for any open interval $I'\supset I$.

Finally, $\gamma$ applied to any $u\in r_{\Omega\times I} C(\R^{n+1})$ gives the expected, namely $r_{\Gamma\times I} \widetilde u$ for any extension $\widetilde u \in C(\R^{n+1})$ of $u$.
Indeed using~\eqref{eq:defOfGamma},
\begin{align*}
  (\gamma u)_{\kappa\times\id_I} 
  = r_{B'\times I} \big(\widetilde u \circ (\lambda^{-1}(\cdot,0)\times\id_\R)\big) 
  = (r_{\Gamma\times I}\widetilde u)\circ (\kappa^{-1}\times\id_I)
  = (r_{\Gamma\times I}\widetilde u)_{\kappa\times\id_I},
\end{align*}
which shows that $\gamma u$ equals a restriction, $r_{\Gamma\times I}\widetilde u$, of the continuous function $\widetilde u$.

\subsubsection{The Theorem}\label{subsubsec:theThm}
To construct a right-inverse of $\gamma$, we use a bounded right-inverse $K_n$ of $\gamma_{0,n}$, where because of~\eqref{eq:conditionsOn_ap_anisotropic2} we may refer to~\cite[Thm.~2.6]{JS08} for a right-inverse of the similar trace $\gamma_{0,1}$ in~\eqref{eq:thm22},
\begin{equation}\label{eq:rightInverseKn}
  K_n: F^{s-a_0/p_0,a''}_{p'',p_0}(\Rn)\to F^{s,\vec a}_{\vec p,q}(\R^{n+1}), \quad s\in\R,
\end{equation}
given by, cf.~just above~\eqref{eq:defOfK1} for the $\psi$,
\begin{equation*}
K_n v(x) := \sum_{j=0}^\infty \psi(2^{ja_n}x_n) \cf^{-1}(\Phi_j(\xi',0,\xi_{n+1})\cf v(\xi',\xi_{n+1}))(x',x_{n+1}).
\end{equation*}
Hereby we have set $p''=(p_0,\ldots,p_0,p_t)\in \,]0,\infty[\,^n$, which results when $p_n=p_0$ is left out; cf.~\eqref{eq:conditionsOn_ap_anisotropic2}.

\begin{theorem}\label{thm:boundednessCurvedTrace}
When $\Gamma$ is compact, $\vec a,\vec p$ satisfy \eqref{eq:conditionsOn_ap_anisotropic2} 
and $(s,q)$
fulfils the inequality in~\eqref{eq:thm22},
then
\begin{equation*}
\gamma: \overline F^{s,\vec a}_{\!\vec p,q}(\Omega \times I) \to \overline F^{s-a_0/p_0,a''}_{\!p'',p_0}(\Gamma\times I)
\end{equation*}
is a bounded \emph{surjection}, which has a right-inverse $K_\gamma$. 
More precisely, the operator $K_\gamma$ can be chosen such that $K_\gamma: \overline F^{s-a_0/p_0,a''}_{p'',p_0}(\Gamma\times I)\to \overline F^{s,\vec a}_{\vec p,q}(\Omega \times I)$ is bounded for every $s\in\R$.
\end{theorem}

\begin{proof}
Since the space $\overline F^{s-a_0/p_0,a''}_{p'',p_0}(\Gamma\times I)$, cf.~\eqref{eq:introFspaceOverline}, does not depend on how the compact interval $J\supset I$ is chosen,  it is fixed in the following.
Moreover, $\gamma u$ does not depend on the extension $f$ of $u$, thus we take $f$ such that $\supp f \subset \Rn \times J$. 
By~\eqref{eq:connectionBtwnvAndf} 
and~\eqref{eq:introFspaceOverline}, $\gamma u = r_I v$ is in $\overline F^{s-a_0/p_0,a''}_{\!p'',p_0}(\Gamma\times I)$.

To prove boundedness, note that $v$ according to~\eqref{eq:equivSupport} belongs to $\overset{\circ}{F}{}^{s, a''}_{\! p'',q}(\Gamma\times J)$, 
since 
\begin{equation}\label{eq:suppOfv}
\supp v \subset \bigcup_{\lambda\in\Lambda} (\lambda^{-1}(\cdot,0)\times\id_\R)(B'\times J) = \Gamma\times J.
\end{equation}
Hence it can be inferred from Theorem~\ref{thm:LTcurvedBoundary} that
\begin{gather}\begin{split}\label{eq:boundednessOfGamma1}
\|\, \gamma u\, | \overline F^{s-a_0/p_0,a''}_{p'',p_0}(\Gamma\times I) &\|^d\\
\le \inf_{\substack{r_{\Omega\times I}f=u\\ \supp f\subset\Rn\times J}}
\sum_{\lambda\in\Lambda} &\|\, (\psi_\lambda\otimes\ind \R f)\circ(\lambda^{-1}(\cdot,0)\times\id_\R)\,
| \overline F^{s-a_0/p_0,a''}_{p'',p_0}(B'\times\R) \|^d.
\end{split}
\end{gather}
By choosing first a cut-off function on $\R$, we can use the infimum norm to fix~$f$ such that 
$\|\, f\, | F^{s,\vec a}_{\vec p,q}(\R^{n+1})\| \le 2\|\, u \, |\overline F^{s,\vec a}_{\vec p,q}(\Omega\times I)\|$. 
Using the arguments leading up to~\eqref{eq:interchangedRandE} and the boundedness of $\gamma_{0,n}$, cf.~\eqref{eq:thm22}, each summand in~\eqref{eq:boundednessOfGamma1} can be estimated by
\begin{equation*}
c \|\, (\psi_\lambda\otimes\ind \R f)\circ(\lambda^{-1}\times\id_\R)\, 
| \overline F^{s,\vec a}_{\vec p,q}(B\times\R)\|^d.
\end{equation*}
Finally, applying Theorem~\ref{thm:infiniteCylinder} and Lemma~\ref{multTraces}, since $\psi_\lambda\otimes\ind\R\in C^\infty_{L_\infty}(\R^{n+1})$, 
we obtain
\begin{equation*}
\|\, \gamma u\, | \overline F^{s-a_0/p_0,a''}_{p'',p_0}(\Gamma\times I)\|
\le c \|\, f\, | F^{s,\vec a}_{\vec p,q}(\R^{n+1})\|
\le 2c \|\, u\, | \overline F^{s,\vec a}_{\vec p,q}(\Omega\times I)\|.
\end{equation*}

The construction of a right-inverse $K_\gamma$ uses that for any $w\in\overline F^{s-a_0/p_0,a''}_{p'',p_0}(\Gamma\times I)$ there exists a $v\in \overset{\circ}{F}{}^{s-a_0/p_0,a''}_{\!p'',p_0}(\Gamma\times J)$ such that $r_I v=w$. It is easily verified that 
\begin{equation}\label{eq:defOfwKappa}
w^\kappa:=r_{\R^{n-1}\times I} \big( e_{B'\times\R}(\widetilde\psi_\kappa\otimes\ind \R v_{\kappa\times\id_\R})\big)
\end{equation}
is independent of the extension $v$; and obviously $w^\kappa$ is in $\overline F^{s-a_0/p_0,a''}_{p'',p_0}(\R^{n-1}\times I)$ with support in $B'\times I$.
For $\chi_1,\chi_2\in C_0^\infty(\R)$ such that $\chi_1+\chi_2\equiv 1$ on a neighbourhood of $I$ and such that $\chi_1$, $\chi_2$ vanish before the right, respective the left end point of $I$, we let, cf.~Theorem~\ref{thm:rychkovExtension} and Corollary~\ref{cor:rychkovExtension},
\begin{equation*}
  w^\kappa_{\ext} = \ce_{\univ}(\chi_1 w^\kappa) + \ce_{\univ,T} (\chi_2 w^\kappa),
\end{equation*}
where extension by 0 to $\R^n_+$ and $\R^{n-1}\times ]-\infty,T[$ before application of $\ce_{\univ}$, respectively $\ce_{\univ,T}$ is understood. 
Lemma~\ref{lem:extensionBy0inLastCoord} gives that this extension does not change the regularity of the elements, hence $w^\kappa_{\ext}$ belongs to $F^{s-a_0/p_0,a''}_{p'',p_0}(\R^n)$; and furthermore $r_{\R^{n-1}\times I} w^\kappa_{\ext} = w^\kappa$.

Now using $K_n$ in~\eqref{eq:rightInverseKn} as well as functions $\eta_\lambda\in C^\infty_0(\Rn)$, $\lambda\in\Lambda$, such that $\supp\eta_\lambda\subset B$ and $\eta_\lambda=1$ on $\supp \widetilde\psi_\lambda$, we define (using the $v$-independence of $w^\kappa_{\ext}$)
\begin{equation}\label{eq:defOfRightInverseOfGamma}
K_\gamma w = r_{\Omega\times I}\sum_{\lambda\in\Lambda} e_{U_\lambda\times\R} (\eta_\lambda K_n w^\kappa_{\ext})\circ(\lambda\times\id_\R).
\end{equation}

Boundedness of $K_\gamma$ is a consequence of first using Lemma~\ref{lem:extensionBy0} and Theorem~\ref{thm:infiniteCylinder}, $d:=\min(1,p_0,p_t,q)$,
\begin{align*}
\|\, K_\gamma w\, | \overline F^{s,\vec a}_{\vec p,q}(\Omega \times I)\|^d
&\le \sum_{\lambda\in\Lambda} \|\, (\eta_\lambda K_n w^\kappa_{\ext})\circ(\lambda\times\id_\R)\, | \overline F^{s,\vec a}_{\vec p,q}(U_\lambda\times\R)\|^d\\
&\le c \sum_{\lambda\in\Lambda} \|\, \eta_\lambda K_n w^\kappa_{\ext}\, | F^{s,\vec a}_{\vec p,q}(\R^{n+1})\|^d,
\end{align*}
and then Lemma~\ref{multTraces}, Lemma~\ref{lem:extensionBy0inLastCoord} and the mapping properties of $K_n$, $\ce_{\univ}$, $\ce_{\univ,T}$,
\begin{align*}
\|\, K_\gamma w\, | \overline F^{s,\vec a}_{\vec p,q}(\Omega \times I)\|^d
&\le
c\sum_{\substack{\kappa\in\cf_0\\j=1,2}} \|\, \chi_j w^\kappa\, | \overline F^{s-a_0/p_0,a''}_{p'',p_0}(\R^{n-1}\times I)\|^d\\
&\le c \sum_{\kappa\in\cf_0} \|\, (\psi_\kappa \otimes \ind \R v) \circ (\kappa^{-1}\times\id_\R)\, | \overline F^{s-a_0/p_0,a''}_{p'',p_0}(B'\times\R)\|^d.
\end{align*}
The extension $v$ is chosen arbitrarily among those in $\overset{\circ}{F}{}^{s-a_0/p_0,a''}_{\!p'',p_0}(\Gamma\times J)$ satisfying 
$r_I v = w$, thus taking the infimum over all such $v$ yields the boundedness of $K_\gamma$, cf.~\eqref{eq:normOnRestrictionCompactnessSpace} and~\eqref{eq:quasinormOnGammaTimesI}.

To verify that $K_\gamma$ is indeed a right-inverse, we use that an extension of $K_\gamma w$ is
\begin{equation*}
  f = \sum_{\lambda\in\Lambda} e_{U_\lambda\times\R} (\eta_\lambda K_n w^\kappa_{\ext})\circ(\lambda\times\id_\R).
\end{equation*}
Hence the definition of $\gamma$, cf.~\eqref{eq:defOfGamma}, gives that $\gamma(K_\gamma w) = r_I h$, where $h_{\kappa_1\times\id_\R} = f\circ(\lambda_1^{-1}(\cdot,0)\times\id_\R)$. 
We shall now prove that $r_{B'\times I} h_{\kappa_1\times\id_R} = w_{\kappa_1\times\id_I}$ for each $\kappa_1\in\cf_0$.
Indeed,
\begin{align}\label{eq:firstCalculationRestrictionTrace}
  r_{B'\times I} h_{\kappa_1\times\id_\R}
= r_{B'\times I} \sum_{\lambda\in\Lambda} (\eta_\lambda K_n w^\kappa_{\ext})\circ(\lambda\circ\lambda^{-1}_1(\cdot,0)\times\id_\R),
\end{align}
where extension by 0 from $\kappa_1(\Gamma_{\kappa_1}\cap\Gamma_\kappa)\times\R$ to $B'\times \R$ in each term is understood.
Using that $K_n$ is a right-inverse of $\gamma_{0,n}$ and that $w^\kappa_{\ext} = w^\kappa$ on $\kappa(\Gamma_{\kappa_1}\cap\Gamma_\kappa)\times I$, each summand in~\eqref{eq:firstCalculationRestrictionTrace} equals, cf.\ also~\eqref{eq:defOfwKappa},~\eqref{eq:distFamily}, 
\begin{align*}
  (\eta_\lambda w^\kappa) \circ (\lambda\circ\lambda_1^{-1}(\cdot,0)\times\id_\R) 
  =  (\eta_\lambda \circ\lambda \cdot \psi_\kappa\otimes\ind \R)\circ(\lambda_1^{-1}(\cdot,0)\times\id_\R) v_{\kappa_1\times\id_\R} .
\end{align*}
As $\eta_\lambda\circ\lambda \equiv 1$ on $\supp \psi_\kappa$ and $\sum \psi_\kappa \equiv 1$ on $\Gamma$, we finally obtain, using that $r_I v = w$,
\begin{equation*}
r_{B'\times I} h_{\kappa_1\times\id_\R} 
= r_{B'\times I}\Big(v_{\kappa_1\times\id_\R} \sum_{\lambda\in\Lambda} (\psi_\kappa\otimes\ind \R)\circ(\lambda_1^{-1}(\cdot,0)\times\id_\R) \Big)
= w_{\kappa_1\times\id_I},
\end{equation*}
hence $K_\gamma$ is a right-inverse of $\gamma$.
\end{proof}

\subsection{The Traces at the Corner}\label{subsec:tracesAtTheCorner}
The trace from either the flat or the curved boundary to the corner $\Gamma\times\{0\}\simeq \Gamma$ cannot simply be obtained by applying $r_0$ and then $\gamma$, or vice versa, since these operators are defined on spaces over the whole cylinder. 

In the following, under the assumptions that $I=\,]0,T[\,$ is finite and $\Gamma$ compact, the trace operators $r_{0,\Gamma}$, $\gamma_\Gamma$ will therefore be introduced (the subscript $\Gamma$ indicates that we end up at $\Gamma\times\{0\}\simeq \Gamma$). We note that focus will not be on optimality regarding the co-domains, since the purpose of this section merely is to prepare for a discussion of compatibility conditions in connection with PDEs; and from this point of view the interesting question is whether the following identity holds in $\cd'(\Gamma)$,
\begin{equation}
\label{eq:compatibilityConditions1}
  r_{0,\Gamma} \circ \gamma u = \gamma_\Gamma\circ r_0 u.
\end{equation}

Recall that when working with spaces on the boundary, the anisotropy and the vector of integral exponents only have $n$ entries. Since it is different entries that need to be left out, depending on whether we are studying $\Gamma\times I$ or $\Omega$, it will in the following be convenient to use $a''=(a_1,\ldots,a_{n-1},a_t)$ as well as $a'=(a_1,\ldots,a_n)$; and likewise for $p'$, $p''$. Moreover,~\eqref{eq:conditionsOn_ap_anisotropic2} is a standing assumption on $\vec a$, $\vec p$.

We assume that $s$ satisfies the inequality in~\eqref{spqn-cnda} adapted to vectors of $n$ entries, i.e. for the trace from the curved boundary $\Gamma\times I$,
\begin{equation}\label{eq:conditionOnsCornerTrace}
s>\frac{a_t}{p_t}+ (n-1)\Big(\frac{a_0}{\min(1,p_0)}-a_0\Big),
\end{equation}
and for the trace from the flat boundary $\Omega$,
\begin{equation}\label{eq:conditionOnsCornerTraceFlat}
s>\frac{a_0}{p_0}+ (n-1) \Big(\frac{a_0}{\min(1,p_0)}-a_0\Big).
\end{equation}

\begin{remark}\label{rem:regardingContinuityInTimeAndCurvedTrace}
When $v\in \overset{\circ}{F}{}^{s,a''}_{\!p'',q}(\Gamma\times J)$ for a compact interval $J$ and $s$ fulfils~\eqref{eq:conditionOnsCornerTrace}, then 
\begin{equation*}
v_{\kappa\times\id_\R} \in C_{\operatorname{b}}(\R_t,L_{1,\loc}(B'))\quad\text{for each }\kappa\in\cf_0.
\end{equation*}
This follows if for every compact set $K\subset B'$, the map $t\mapsto v_{\kappa\times\id_\R}(\cdot,t)$ is continuous with values in $L_1(K)$.
In Theorem~\ref{thm:LTcurvedBoundary} we may, if necessary, change the partition of unity (using some $\varphi\in C_0^\infty(B')$ equalling 1 on $K$) such that $\psi_\kappa\equiv 1$ on $\kappa^{-1}(K)$. Then $\widetilde\psi_\kappa v_{\kappa\times\id_\R}$ is in $\overline F^{s,a''}_{p'',q}(B'\times\R)$, which because of~\eqref{spqn-cnda} and~\eqref{eq:conditionOnsCornerTrace} is contained in $C_{\operatorname{b}}(\R_t,L_1(B'))$. Hence $v_{\kappa\times\id_\R}$ is in $L_1(K)$, continuously in time.
\end{remark}

\subsubsection{The Curved Boundary}
For $w\in \overline F^{s,a''}_{p'',q}(\Gamma\times I)$ there exists a $v\in \overset{\circ}{F}{}^{s,a''}_{\!p'',q}(\Gamma\times J)$, where $J\supset I$ is any compact interval, such that $r_I v= w$, cf.~\eqref{eq:introFspaceOverline}.
By exploiting that $v_{\kappa\times\id_\R}$ is continuous with respect to $t$, cf.\ Remark~\ref{rem:regardingContinuityInTimeAndCurvedTrace}, we define for $x\in\Gamma$,
\begin{equation}\label{eq:defOfTheTraceToTheCornerFromCurvedBoundary}
r_{0,\Gamma} w (x) = \sum_{\kappa\in\cf_0} \psi_\kappa(x) v_{\kappa\times\id_\R}(\kappa(x),0)
\end{equation}
with the understanding that the product $\psi_\kappa(x)v_{\kappa\times\id_\R}(\kappa(x),0)$ is defined to be 0 outside $\Gamma_\kappa$.
On $\Gamma_\kappa$ the product is meaningful, since $v_{\kappa\times\id_\R}$ is in $C_{\operatorname{b}}(\R, L_{1,\loc}(B'))$.

The trace $r_{0,\Gamma}$ in~\eqref{eq:defOfTheTraceToTheCornerFromCurvedBoundary} is independent of the chosen $v\in \overset{\circ}{F}{}^{s,a''}_{\!p'',q}(\Gamma\times J)$, 
since for any two extensions $v_1,v_2$ in this space, $\widetilde\psi_\kappa\cdot r_{B'\times I} v_{j,\kappa\times\id_\R}$, $j=1,2$, coincide on $B'\times I$, hence by continuity also on $B'\times\{0\}$. 

Moreover, the trace depends neither on the atlas nor on the subordinate partition of unity. Indeed, considering another atlas $\cf_1$ with a subordinate partition of unity $1=\sum_{\kappa_1\in\cf_1}\varphi_{\kappa_1}$, we have on $\Gamma$, cf.~\eqref{eq:distFamily} for the atlas $\cf_0\cup\cf_1$,
\begin{align*}
  \sum_\kappa \psi_\kappa v_{\kappa\times\id_\R}(\kappa(\cdot),0)
  = \sum_{\kappa,\kappa_1} \psi_\kappa \varphi_{\kappa_1} v_{\kappa_1\times\id_\R}(\kappa_1(\cdot),0)
  = \sum_{\kappa_1} \varphi_{\kappa_1} v_{\kappa_1\times\id_\R}(\kappa_1(\cdot),0).
\end{align*}

In the following theorem the co-domain of the trace is $B^{s-\frac{a_t}{p_t},a_0}_{p_0,p_t}(\Gamma)$; the definition and properties of this space follow from Section~\ref{sec:isotropicBesovManifolds}, since it coincides with an isotropic space in view of~\eqref{eq:conditionsOn_ap_anisotropic2} and Lemma~\ref{lem:lambdaLemmaBesov}. Note that we have abbreviated the $(n-1)$-vector $(a_0,\ldots,a_0)$ to $a_0$, and similarly for $p_0$.

\begin{theorem}\label{thm:traceAtCornerFromCurvedBoundary}
When $a'', p''$  are as above with $0<p''<\infty$ and $s$ satisfies~\eqref{eq:conditionOnsCornerTrace},
then $r_{0,\Gamma}$ is bounded,
\begin{equation*}
  r_{0,\Gamma}:\overline F^{s, a''}_{ p'',q}(\Gamma\times  I)\to B^{s-\frac{a_t}{p_t},a_0}_{p_0,p_t}(\Gamma).
\end{equation*}
\end{theorem}

\begin{proof}
From Remark~\ref{rem:regardingContinuityInTimeAndCurvedTrace} we have that $v_{\kappa\times\id_\R}$
is in $C_{\operatorname{b}}(\R,L_{1,\loc}(B'))$, hence using the bounded trace operator,
cf.~\cite[Thm.~2.5]{JS08} and~\eqref{eq:conditionOnsCornerTrace},  
\begin{equation}\label{eq:traceGamma0n}
  \gamma_{0,n}: F^{s, a''}_{p'',q}(\Rn) \to B^{s-\frac{a_t}{p_t},a_0}_{p_0,p_t}(\R^{n-1}),
\end{equation}
it is readily seen that
\begin{align*}
  \widetilde\psi_\kappa v_{\kappa\times\id_\R}(\cdot,0)
  = r_{B'} \gamma_{0,n} e_{B'\times\R} (\widetilde\psi_\kappa v_{\kappa\times\id_\R}).
\end{align*}
Since $\widetilde\psi_\kappa v_{\kappa\times\id_\R} \in F^{s, a''}_{p'',q}(B'\times\R)$,
we therefore have by~\eqref{eq:traceGamma0n} that 
$\widetilde\psi_\kappa v_{\kappa\times\id_\R}(\cdot,0)$ belongs to $\overline B^{s-\frac{a_t}{p_t},a_0}_{p_0,p_t}(B')$.
Now Corollary~\ref{cor:extensionByZeroLTspaces} adapted to Besov spaces,
cf.~Section~\ref{sec:isotropicBesovManifolds}, implies that $r_{0,\Gamma} w\in
B^{s-\frac{a_t}{p_t},a_0}_{p_0,p_t}(\Gamma)$. 

To prove $r_{0,\Gamma}$ is bounded, we use~\eqref{eq:theNormOnBesovOverManifold} to estimate $\|\, r_{0,\Gamma} w\, |B^{s-\frac{a_t}{p_t},a_0}_{p_0,p_t}(\Gamma)\|^d$, $d:=\min(1,p_0,p_t)$, by
\begin{align*}
  \sum_{\kappa,\kappa_1\in\cf_0} \big\|\, \psi_{\kappa_1}\circ\kappa^{-1} \cdot
  \widetilde\psi_\kappa v_{\kappa_1\times\id_\R}(\kappa_1\circ\kappa^{-1}(\cdot),0)\, \big|
  \overline
  B^{s-\frac{a_t}{p_t},a_0}_{p_0,p_t}(\kappa(\Gamma_\kappa\cap\Gamma_{\kappa_1}))\big\|^d. 
\end{align*}
After a change of coordinates $x\mapsto \kappa\circ\kappa_1^{-1}(x)$ and a slight restriction of the
domain to a suitable open subset $W$ such that $\overline W\subset
\kappa_1(\Gamma_{\kappa_1}\cap\Gamma_\kappa)$, and finally multiplication by a $\chi_{\kappa_1}\in
C_0^\infty(B')$ where $\chi_{\kappa_1}\equiv 1$ on $\supp \widetilde\psi_{\kappa_1}$, this can be
estimated by, cf.~\cite[4.2.2]{tri92} for an $s_1$ large enough,  
\begin{align*}
  c \sum_{\kappa,\kappa_1\in\cf_0} \Big(\sum_{|\alpha|\le s_1} \|\, D^\alpha
      e_{B'}(\psi_\kappa\circ\kappa_1^{-1}\chi_{\kappa_1})\, | L_\infty \|\Big)^d \,
  \big\|\, e_{B'}(\widetilde\psi_{\kappa_1} v_{\kappa_1\times\id_\R}(\cdot,0))\,
  \big|B^{s-\frac{a_t}{p_t},a_0}_{p_0,p_t}(\R^{n-1})\big\|^d; 
\end{align*}
the constant $c$ contains $\sup_{\overline W} |\det J(\kappa\circ\kappa_1^{-1})|^d$ as a finite
factor ($J$ denotes the Jacobian matrix). 
Now boundedness of $\gamma_{0,n}$ in~\eqref{eq:traceGamma0n} gives
\begin{align*}
  \big\|\, r_{0,\Gamma} w\, \big| B^{s-\frac{a_t}{p_t}, a_0}_{p_0,p_t}(\Gamma)\big\|^d
  \le c \sum_{\kappa_1\in\cf_0} \|\, \widetilde\psi_{\kappa_1} v_{\kappa_1\times\id_\R}\, |
  \overline F^{s,a''}_{ p'',q}(B'\times\R)\|^d, 
\end{align*}
hence taking the infimum over all admissible $v$ (as we may since $r_{0,\Gamma}$ is independent of the extension) 
proves that $r_{0,\Gamma}$ is bounded.
\end{proof}

\subsubsection{The Flat Boundary}
\label{subsubsec:flatBoundary}
In this section we consider the trace operator $\gamma_\Gamma$, which simply is the trace at $\Gamma$ of distributions defined on~$\Omega$. In view of~\eqref{eq:compatibilityConditions1} and Theorem~\ref{main11}, the domain of interest for $\gamma_\Gamma$ is the unmixed Besov space $\overline B^{s,a'}_{p',q}(\Omega)$, which according to Lemma~\ref{lem:lambdaLemmaBesov} even equals an isotropic space, cf.~\eqref{eq:conditionsOn_ap_anisotropic2},

The operator is defined by carrying over the definition and results for $\gamma$ in Section~\ref{subsec:traceCurvedBoundaryMathNr}. 
Indeed, we remove the time dependence and use the Besov space result in~\cite[Thm.~1]{FJS00} for
$\gamma_{0,n}$. An embedding similar to~\eqref{eq:CbEmbeddedInL1loc} also holds in the case of Besov
spaces, cf.~\cite[Prop.~1]{FJS00} and~\eqref{eq:conditionOnsCornerTraceFlat}, and
Theorem~\ref{thm:LTcurvedBoundary},~\ref{thm:infiniteCylinder} are replaced by the Besov versions,
cf.~Section~\ref{sec:isotropicBesovManifolds}, of
Theorem~\ref{thm:compact},~\ref{thm:isotropicCompactSupport} respectively. 
Recalling that the $(n-1)$-vector $(a_0,\ldots,a_0)$ is abbreviated $a_0$, and likewise for $p_0$, this yields

\begin{theorem}\label{thm:traceToCornerFromFlatBoundary}
When $a'=(a_0,\ldots,a_0)\in [1,\infty[\,^n$, $p'=(p_0,\ldots,p_0)\in \,]0,\infty[\,^n$ and $s$ satisfies~\eqref{eq:conditionOnsCornerTraceFlat}, then $\gamma_\Gamma$ is a bounded operator,
\begin{equation*}
  \gamma_\Gamma: \overline B^{s,a'}_{p',q}(\Omega)\to B^{s-a_0/p_0,a_0}_{p_0,q}(\Gamma).
\end{equation*}
\end{theorem}

We note that, as usual for Besov spaces, the sum exponent is not changed and, moreover, a formula similar to the one in~\eqref{eq:defOfTheTraceToTheCornerFromCurvedBoundary}  for $r_{0,\Gamma}$ holds for $\gamma_\Gamma$. 
I.e. for any extension $f$ of $w\in \overline B^{s,a'}_{p',q}(\Omega)$,
with~\eqref{eq:connectionBtwnvAndf}--\eqref{eq:defOfGamma} adapted to $\gamma_\Gamma$ for the
$f_\kappa$, we have when extension by 0 outside $\Gamma_\kappa$ is suppressed, 
\begin{align}\label{eq:traceToCornerAsSum}
\gamma_\Gamma w = \sum_{\kappa\in\cf_0} \psi_\kappa\cdot f_\kappa \circ\kappa.
\end{align}
Indeed, $\big(\sum_{\kappa\in\cf_0} \psi_\kappa\cdot f_\kappa \circ\kappa\big)_{\kappa_1} 
= \sum_{\kappa\in\cf_0} \psi_\kappa\circ \kappa_1^{-1}\cdot f_{\kappa_1}
= f_{\kappa_1} = (\gamma_\Gamma w)_{\kappa_1}$ for each $\kappa_1\in \cf_0$.
This formula is convenient in a discussion of compatibility conditions, cf.\ the next section.

\subsection{Applications}\label{subsec:TracesPaperApplicationsSubSec}
Without proof, we now indicate, by merely adapting~\cite[Ch.~6]{GS90} to the present set-up, what
the above considerations yield in a study of e.g.\ the heat equation. That is, for $\Delta
= \partial_{x_1}^2+\ldots + \partial_{x_n}^2$ we consider 
\begin{alignat}{4}
\partial_t u -\Delta u &= g \enskip&&\text{in}\enskip &\Omega&&&\times I , \label{eq:heatEquation1Traces}\\
\gamma u &= \varphi \enskip&&\text{on}\enskip &\Gamma&&&\times I \label{eq:heatEquation2Traces}, \\
r_0 u &= u_0 \enskip&&\text{on}\enskip &\Omega &&&\times \{0\}\label{eq:heatEquation3Traces}.
\end{alignat}

Under the assumption that $\vec a = (1,\ldots,1,2)$ and $\vec p =(p_0,\ldots,p_0,p_t)<\infty$, we
give in the theorem below necessary conditions for the existence of a solution $u$ in $\overline
F^{s,\vec a}_{\vec p,q}(\Omega\times I)$, when $\gamma$ and $r_0$
in~\eqref{eq:heatEquation2Traces},~\eqref{eq:heatEquation3Traces} make sense, i.e.\ when $s$ fulfils
the two conditions 
\begin{equation}\label{eq:tracesPaperHeatEquationS}\begin{split}
s &> \frac1{p_0} + (n-1)\Big(\frac{a_0}{\min(1,p_0,q)}-a_0\Big) +\Big(\frac{a_t}{\min(1,p_0,p_t,q)}-a_t\Big) \quad\text{and}\\
s &>\frac2{p_t} + n\Big(\frac1{\min(1,p_0)}-1\Big).
\end{split}\end{equation}

\begin{theorem} \label{thm:compatibility}
Let $\vec a$, $\vec p$ and $s$ satisfy the requirements above.
When the boundary value problem in~\eqref{eq:heatEquation1Traces}--\eqref{eq:heatEquation3Traces} has a solution $u \in \overline F^{s,\vec a}_{\vec p,q}(\Omega\times I)$, then the data $(g,\varphi,u_0)$ necessarily satisfy
\begin{equation}
g\in \overline F^{s-2,\vec a}_{\!\vec p,q}(\Omega\times I),\quad \varphi\in \overline
F^{s-\frac1{p_0},a''}_{\!p'',p_0}(\Gamma\times I),\quad u_0\in \overline
B^{s-\frac2{p_t}}_{p_0,p_t}(\Omega).
\label{eq:data} 
\end{equation}
Moreover, for all $l\in\N_0$ fulfilling both
\begin{equation}\label{eq:tracesPaperHeatEquationl1}
\begin{split}
  2l &< s - \frac1{p_0}-\frac2{p_t}-(n-1)\Big(\frac1{\min(1,p_0)}-1\Big) \quad\text{and}\\
  2l &< s - \frac1{p_0} - (n-1)\Big(\frac{a_0}{\min(1,p_0,q)}-a_0\Big) -\Big(\frac{a_t}{\min(1,p_0,p_t,q)}-a_t\Big),
\end{split}\end{equation}
the data are compatible in the sense that
\begin{equation}\label{eq:finalLabelPutHere}
r_{0,\Gamma} \partial_t^l \varphi 
= \gamma_\Gamma \Big(\Delta^l u_0 + \sum_{j=0}^{l-1} \Delta^j r_0 (\partial_t^{l-1-j}g)\Big),
\end{equation}
which reduces to $r_{0,\Gamma} \varphi = \gamma_\Gamma u_0$ for $l=0$ (the sum is void).
\end{theorem}

We recall that the corrections containing the minima
in~\eqref{eq:tracesPaperHeatEquationS},~\eqref{eq:tracesPaperHeatEquationl1} amount to 0 in the
classical case in which $\vec p,q\ge 1$. 

\begin{remark} \label{rem:compatibility}
In the construction of solutions to
e.g.~\eqref{eq:heatEquation1Traces}--\eqref{eq:heatEquation3Traces}, it is well known
from work of Grubb and Solonnikov~\cite[Thm.~6.3]{GS90} that the problem is solvable for
$p_0=2=p_t$, provided the data $(g,\varphi,u_0)$ are given as in \eqref{eq:data} and subjected to the
compatibility \emph{conditions} in \eqref{eq:tracesPaperHeatEquationl1}--\eqref{eq:finalLabelPutHere}.  
In an extension of this to general $p_0, p_t$, 
a first step could be to reduce to problems having $\varphi\equiv 0$, $u_0\equiv 0$. 
By linearity, the surjectivity of $\gamma$ in Theorem~\ref{thm:boundednessCurvedTrace} allows a
first reduction to the case $\varphi\equiv0$. Secondly, reduction to $u_0\equiv0$ is obtained by adding
and subtracting $\tilde u=r_{\Omega\times I}Q_\Omega e_\Omega u_0$ (the compatibility condition for
$l=0$ entails that $e_\Omega u_0$ is defined, and for low $s$ has the same regularity as $u_0$), 
for here $r_0\tilde u= u_0$ holds as well as $\gamma \tilde u=0$ because the operator $Q_\Omega$ 
preserves support in $\overline\Omega$, as shown 
in Theorem~\ref{thm:constructionOfRightInverseOnCylinderPerservingSupport}. 
(Details were given by the second author \cite[Ch.~7]{SMHphd}.)
\end{remark}

\providecommand{\bysame}{\leavevmode\hbox to3em{\hrulefill}\thinspace}
\vspace{\baselineskip}


\begin{thebibliography}{JSH13b}

\bibitem[BePa61]{BePa61}
A.~Benedek and R.~Panzone, \emph{The space {$L\sp{p}$}, with mixed norm}, Duke
  Math. J. \textbf{28} (1961), 301--324.

\bibitem[Ber85]{Ber85}
M.~Z.~Berkola{\u\i}ko, \emph{Theorems on traces on coordinate subspaces for some spaces of
  differentiable functions with anisotropic mixed norm}, Dokl. Akad. Nauk SSSR
  \textbf{282} (1985), no.~5, 1042--1046, English translation: Soviet Math.
  Dokl. 31 (1985), no. 3, 540--544.

\bibitem[BIN79]{BIN79}
O.~V. Besov, V.~P. Ilin, and S.~M. Nikol'skij, \emph{Integral
  representations of functions and imbedding theorems}, V. H. Winston \& Sons,
  Washington, D.C., 1978--79, Translated from the Russian, Scripta Series in
  Mathematics, Edited by Mitchell H. Taibleson.

\bibitem[BIN96]{BIN96}
\bysame, \emph{Integral representations of functions and imbedding theorems},
  Nauka, Moscow, 1996, 2nd edition (in russian).

\bibitem[DHP07]{DHP07}
R.~Denk, M.~Hieber, J.~Pr{\"u}ss, \emph{Optimal {$L^p$}--{$L^q$}-estimates for parabolic boundary value problems with inhomogeneous data},
Math. Z. \textbf{257}(1) (2007), 193--224.

\bibitem[FJS00]{FJS00}
W.~Farkas, J.~Johnsen, W.~Sickel, \emph{{Traces of anisotropic {B}esov-{L}izorkin-{T}riebel spaces---a
              complete treatment of the borderline cases}}, Math.~Bohem. \textbf{125}(1) (2000), 1--37.

\bibitem[Gru96]{gru96}
G.~Grubb, \emph{{Functional calculus of pseudodifferential boundary problems}},
second ed., Progress in Mathematics, vol. 65, Birkh\"auser, Boston (1996).

\bibitem[Gru09]{gru09}
\bysame, \emph{Distributions and Operators}, Springer, 2009.

\bibitem[GS90]{GS90}
G.~Grubb, V.~A.~Solonnikov, \emph{{Solution of parabolic pseudo-differential initial-boundary value problems}},
J. Differential Equations \textbf{87}(2) (1990).

\bibitem[H\"{o}r90]{hor90}
L.~H\"{o}rmander, \emph{The Analysis of Linear Partial Differential Operators I},
Springer, Berlin, 1990, 2nd edition.

\bibitem[H\"{o}r07]{hor07}
\bysame, \emph{The Analysis of Linear Partial Differential Operators III},
Springer, Berlin, 2007, reprint of the 1994 edition.

\bibitem[Joh95]{JJ94mlt}
J.~Johnsen, \emph{{Pointwise multiplication of Besov and Triebel--Lizorkin
  spaces}}, Math. Nachr. \textbf{175} (1995), 85--133.

\bibitem[Joh00]{joh00}
\bysame, \emph{Traces of {B}esov spaces revisited},
J. Funct. Spaces Appl. \textbf{3} (2000), 763--779.

\bibitem[Joh11]{joh11}
\bysame, \emph{{Pointwise estimates of pseudo-differential operators}},
J. of Pseudo-Differential Operators and Applications \textbf{2}(3) (2011), 377--398.

\bibitem[JS07]{JS07}
J.~Johnsen, W.~Sickel, \emph{{A direct proof of Sobolev embeddings for quasi-homogeneous Lizorkin--Triebel spaces with mixed norms}},
J. Funct. Spaces Appl. \textbf{5} (2007), 183--198.

\bibitem[JS08]{JS08}
\bysame, \emph{{On the trace problem for Lizorkin--Triebel spaces with mixed norms}}, 
Math. Nachr. \textbf{281} (2008), 669--696.

\bibitem[JSH13a]{HJS13a}
J.~Johnsen, S.~Munch~Hansen, W.~Sickel, \emph{{Characterisation by local means of anisotropic Lizorkin--Triebel spaces with mixed norms}}, Z. Anal. Anwend. \textbf{32}(3) (2013), 257--277.

\bibitem[JSH13b]{HJS13b}
\bysame, \emph{{Anisotropic, mixed-norm Lizorkin--Triebel spaces and
    diffeomorphic maps}},  Journal of Function Spaces, \textbf{2014} (2014), Article ID 964794, 15 pages.

\bibitem[MH13]{SMHphd}
S.~Munch~Hansen, \emph{{On parabolic boundary problems treated in mixed-norm Lizorkin--Triebel
    spaces}}, Ph.D thesis, Aalborg University; Aalborg, Denmark, 2013.

\bibitem[RS96]{RS}
T.~Runst and W.~Sickel, \emph{Sobolev spaces of fractional order, Nemytskij
operators and nonlinear partial differential equations}, de Gryuter, Berlin 1996.
  
\bibitem[Ryc99a]{ryc99}
V.~Rychkov, \emph{On a theorem of Bui, Paluszynski and Taibleson}, Proc. Steklov Institute
{\bf 227} (1999), 280--292.  

\bibitem[Ryc99]{ryc99ext}
\bysame, \emph{On restrictions and extensions of the Besov and Triebel-Lizorkin spaces
   with respect to Lipschitz domains}, J. London Math. Soc. (2) {\bf 60} (1999), 237--257.

\bibitem[ScTr87]{ScTr87}
H.-J.~Schmeisser and H.~Triebel, \emph{Topics in Fourier analysis and function spaces}, Geest \and
Portig, 1987, Wiley, Chichester 1987. 

\bibitem[See64]{see64}
R.~T.~Seeley, \emph{Extension of {$C^{\infty }$} functions defined in a half space},
Proc. Amer. Math. Soc. \textbf{15} (1964), 625--626.

\bibitem[Ste93]{ste93}
E.~M.~Stein, \emph{Harmonic analysis: real-variable methods, orthogonality, and oscillatory
  integrals}, Princeton Univ. Press, Princeton, 1993. 

\bibitem[Tri83]{tri83}
H.~Triebel, \emph{{Theory of function spaces}}, Monographs in Mathematics,
  vol.~78, Birkh{\"a}user Verlag, Basel, 1983.

\bibitem[Tri92]{tri92}
\bysame, \emph{{Theory of function spaces II}}, Monographs in Mathematics,
  vol.~84, Birkh{\"a}user Verlag, Basel, 1992.
  
\bibitem[Tri06]{tri06}
\bysame, \emph{{Theory of function spaces III}}, Monographs in Mathematics,
  vol.~100, Birkh{\"a}user Verlag, Basel, 2006.

\bibitem[Wei98]{wei98}
P.~Weidemaier, \emph{Existence results in {$L\sb p$}-{$L\sb q$} spaces for
  second order parabolic equations with inhomogeneous {D}irichlet boundary
  conditions}, Progress in partial differential equations, Vol. 2
  (Pont-\`a-Mousson, 1997), Pitman Res. Notes Math. Ser., vol. 384, Longman,
  Harlow, 1998, pp.~189--200.

\bibitem[Wei02]{wei02}
\bysame, \emph{Maximal regularity for parabolic equations with inhomogeneous
  boundary conditions in {S}obolev spaces with mixed {$L\sb p$}-norm},
  Electron. Res. Announc. Amer. Math. Soc. \textbf{8} (2002), 47--51
  (electronic).
  
\bibitem[Wei05]{wei05}
\bysame, \emph{Lizorkin-{T}riebel spaces of vector-valued functions and sharp
              trace theory for functions in {S}obolev spaces with a mixed
              {$L_p$}-norm in parabolic problems}, Mat. Sb., 
              Rossi\u\i skaya Akademiya Nauk. Matematicheski\u\i\ Sbornik, 
              vol. 196, (2005), pp.~3--16.	

\bibitem[Yam86]{Y1}
M.~Yamazaki, \emph{{A quasi-homogeneous version of paradifferential operators,
  I. Boundedness on spaces of Besov type}}, J. Fac. Sci. Univ. Tokyo Sect. IA,
  Math. \textbf{33} (1986), 131--174.

\end{thebibliography}
\end{document}